\RequirePackage{fix-cm}
\documentclass[smallextended]{svjour3}   

\smartqed  
\usepackage[12pt]{extsizes}
\usepackage{graphicx}
\usepackage[utf8]{inputenc}
\usepackage[english]{babel}
\usepackage{bbm}

\usepackage[nottoc]{tocbibind}
\usepackage{mathrsfs,amsfonts,amssymb,amsmath}

\usepackage{amsthm}
\usepackage{enumerate}
\usepackage{graphicx,cite}
\usepackage{romannum}
\usepackage{hyperref}
\usepackage{authblk}
\usepackage{graphicx}

\allowdisplaybreaks

\hypersetup{colorlinks=true, linkcolor=blue, citecolor=red}

\DeclareMathOperator{\diam}{diam}
\DeclareMathOperator{\Leb}{Leb}
\DeclareMathOperator{\dist}{dist}

\DeclareMathOperator{\Proj}{Proj}

\numberwithin{equation}{section}

\usepackage[square,numbers]{natbib}
\begin{document}
\pagenumbering{arabic}
\setlength{\belowdisplayskip}{0pt}

\title{Which subsets and when orbits prefer to visit in  multidimensional non-Markovian non-uniformly expanding systems }


\author{Leonid A. Bunimovich  \and
        Yaofeng Su
}


\institute{Leonid A. Bunimovich \at
              School of Mathematics, Georgia Institute of Technology,  Atlanta, USA\\
             \email{leonid.bunimovich@math.gatech.edu}           
           \and
            Yaofeng Su   \at
              Department of Mathematics, Southern University of Science and Technology, Shenzhen, China\\
              \email{suyaofeng@sustech.edu.cn} 
              \and
}


\titlerunning{Polynomial escapes and Finite time dynamics}
\maketitle
\begin{abstract}
The paper addresses some basic questions about finite time dynamics for slowly mixing non-uniformly expanding dynamical systems. It is concerned with transport in phase spaces of such systems, and analyzes which subsets and when the orbits prefer to visit. An asymptotic expansion with explicit coefficients for the decay of polynomial escape rates is obtained. Applications to a large class of (non)Markovian non-uniformly expanding systems are considered.


\tableofcontents

\end{abstract}

\section{Introduction}
Traditionally, only asymptotic properties (when time tends to infinity) were considered for random processes and for chaotic (stochastic) dynamical systems. This long list includes ergodic theorems (or strong laws of large numbers) and various limit theorems (usually starting with the central limit theorem). However, for dealing with real systems in science and applications, it is necessary to know how a system evolves in a finite time. ( ``In a long run we all are dead" as John Maynard Keynes used to say in his most popular quotation). Therefore, the scientists (first of all  physicists) have been trying for a long time to develop, perhaps non-rigorous, but useful for applications finite time dynamics theory. Numerous computer experiments were conducted, mostly trying to analyze a structure in the phase space of finite-time Lyapunov exponents. Unfortunately, no conclusive results have been obtained.
A rigorous mathematical approach to this problem was started in \cite{bunimovichijm}.  In this paper a new question was raised concerning dependence of the process of escape on the position of a hole in the phase space. It was inspired by a breakthrough in experiments in quantum chaos, conducted with atomic billiards by the Davidson and Raizen groups \cite{Raizen, Davidson}. Initially, the theory of open dynamical systems studied only escape rates and conditionally invariant measures for open systems (see a beautiful review \cite{DemersLSY}).
The new question essentially was about transport in a phase space. Observe that in equilibrium statistical mechanics the main problem is about phase transitions, i.e., the existence of several equilibrium states (reasonable invariant measures) of the system. We consider instead a system in an equilibrium state with an invariant SRB measure and study how orbits move (i.e., a finite-time transport) in the phase space of a closed system.
Another popular and powerful topic in analysis of chaotic dynamics deals with recurrences to some fixed subset in the phase space. However, consideration of recurrences does not give information about transport in phase space, i.e., how actually the orbits move from one subset to another.
It was shown \cite{boldingbun} that the first hitting (first passage) probabilities provide important information about transport in a phase space. In other words, we consider an equilibrium state of the system and study where the orbits will most likely  move in a finite time. However, in this paper only ``extremely" uniformly expanding dynamical systems were studied. (This class of dynamical systems was  called fair dice-like (FDL) systems \cite{bunimovich2012fair, bunimovich2026}).
In order to study a finite-time evolution of slowly mixing non-uniformly smooth hyperbolic dynamical systems, more sophisticated approaches and techniques are needed. To this end, we derive in the present paper an \textbf{asymptotic expansion with explicit coefficients of the polynomial decay of escape rates, for the systems with arbitrarily slow polynomial mixing}. It turned out that such an expansion for the polynomial escape rates has a surprising connection to  billiards theory and to number theory, e.g.,to the Riemann hypothesis, see \cite{rh,rh1}. It also allows us to deduce which subsets of the phase space and when orbits of slowly mixing dynamical systems visit with relatively high probabilities.

Our approach is based on a new type of operator renewal equations and on a new  spectral analysis for open systems under a general framework. Below we mention some modifications/improvements obtained in the present paper in comparison with previous results on open systems.

\begin{enumerate}
    \item  A problem on the dependence of the escape rate in open dynamical systems on the position of a hole in the phase space was first formulated in \cite{bunimovichijm}. Since then, only one-dimensional exponentially mixing systems were considered in this respect \cite{Demersexp, bunimovichijm, boldingbun, haydn, pollicot, fresta}. These papers employed  combinatorial, probabilistic, and spectral gap techniques. To the best of our knowledge, this question is addressed for the first time in the present paper for slowly mixing non-uniformly expanding dynamical systems, which can be multidimensional and non-Markovian.
    \item It is also worth mentioning that, given a SBR measure, Theorem \ref{thm} contains more information than previously obtained in the case of polynomial escape rates (e.g., \cite{Demers}). An asymptotic expansion with explicit coefficients for escape rates is obtained, particularly for non-Markovian holes, claiming which subsets the orbits prefer to visit. 
    \item  The Keller-Liverani operator perturbation theory \cite{liveranikeller,liveranikeller1} is not directly applicable in our setting, particularly for a two dimensional intermittent dynamics. The reasons for that are presented before the Proposition \ref{extendliverani1}, Lemma \ref{extendliverani2}, Lemma \ref{smialrtoliveranikeller}. These lemmas/propositions provide an alternative spectral analysis, which does not use the conditions in the Keller-Liverani theory. 
\end{enumerate}

\textbf{Structure of the paper:} In Section \ref{defsection}, we introduce definitions of the objects under study throughout the paper. The section \ref{resultsection} presents the main (abstract) theorem: Theorem \ref{thm}. In the section \ref{prelimsection} some preparations and preliminary estimates are given. A scheme of a proof for Theorem \ref{thm} is presented at the end of the Section \ref{prelimsection}. The section \ref{proof} gives a proof of Theorem \ref{thm}. The section \ref{appsection} deals with the applications and verifies conditions of Theorem \ref{thm} for specific models, derives Theorem \ref{thm2} for one dimensional intermittent maps and Theorem \ref{thm3} for two dimensional intermittent maps. In the section \ref{corproof} a corollary for Theorem \ref{thm2} and Theorem \ref{thm3} is considered, claiming which subsets of the phase space the orbits prefer to visit in our models.  
\medskip
\noindent

Throughout the paper, we use the following notations
\begin{enumerate}
 \item $C_z$ denotes a constant depending on $z$.
    \item The notation $``a_n \precsim_z b_n"$  (or $``a_n=O_{z}(b_n)"$) means that there is a constant $C_z \ge 1$ such that $ a_n \le C_z  b_n$ for all $n \ge 1$, while the notation $``a_n \precsim b_n"$ (or $``a_n=O(b_n)"$) means that there is a constant $C \ge 1$, such that $ a_n \le C  b_n$ for all $n \ge 1$. Next, $``a_n \approx_z b_n"$ and $``a_n=C_z^{\pm 1}b_n"$ mean that there is a constant $C_z \ge 1$, such that  $ C_z^{-1}  b_n \le a_n \le C_z b_n$ for all $n \ge 1$.  Additionally, the notations $``a_n=C^{\pm1} b_n"$ and $``a_n \approx b_n"$ mean that there is a constant $C \ge 1$, such that $ C^{-1}  b_n \le a_n \le C b_n$ for all $n \ge 1$. Finally, $``a_n =o(b_n)"$ means that  $\lim_{n \to \infty}|a_n/b_n|=0$. 
    We will use the following notations for the operators $``P_n\precsim_z b_n"$ and $``P_n=O_z(b_n)"$ in order to indicate that there is $C_z \ge 1$,  such that $||P_n||\le C_z b_n$ for all $n \ge 1$. Similarly, $``P_n\precsim b_n"$ means that there is $C \ge 1$,  such that $||P_n||\le C b_n$ for all $n \ge 1$.
    
    \item $\mu_A$ (resp. $\Leb_A$) denotes a normalized measure (resp. a normalized Lebesgue measure) of a measurable set $A$, unless it is specially defined. $\mathbbm{1}_A$ is a characteristic function of $A$.
    \item An operator $I$ denotes the identity operator $Id$ for any (complex) Banach space. For any $z\in \mathbb{C}$, $zI$ is usually abbreviated as $z$ if the context is clear.
    \item $\mathbb{N}=\{1,2,3,\cdots\}$, $\mathbb{N}_0=\{0,1,2,3,\cdots\}$. 
    \item $\overline{\mathbb{D}}:=\{z\in \mathbb{C}: |z|\le 1\}$, $\mathbb{D}:=\{z\in \mathbb{C}: |z|<1\}$, $S^1:=\{z\in \mathbb{C}: |z|=1\}$. $B_r(z)\subsetneq \mathbb{C}$ means an open disk with a center $z$ and a radius $r$.
    \item  $P_n^a := {n!}/(n-a)!$ for any $a\le n$.
\end{enumerate}

\section{An abstract model}\label{defsection}

In order to better understand the behavior of chaotic dynamical systems, it was suggested in \cite{Bunimovich_2007} that by making one or more holes in the phase space and by analyzing leaky dynamics and leakage speed it is possible to obtain useful information about the dynamics of the closed system (i.e., without a hole in the phase space) and transport in the phase space of the closed system. In this section, we introduce a general abstract tower model considered throughout this paper, with a hole opened in the base of the tower as follows. 
\begin{definition}[Open polynomial towers]\label{defyoung} \par
    Let $X$ be a finite union of connected and bounded $d$-dimensional Riemannian manifolds, a tower $\Delta$ is built over $X$ by defining $\Delta:=\{(x,n) \in X \times \mathbb{N}_0: n<R(x)\}$ where $X$ is called a base (here we identify $X\times \{0\}$ with $X$) and $R$ is called a return time defined on  $X$. Assume that the  tower satisfies the following conditions:
\begin{enumerate}
\item \textbf{Partition}: $X$ has a countable measurable partition $\{X_{i}\}_{i \ge 1}$ up to a Lebesgue measure zero set, where $X_{i}$ is a connected subdomain.
    \item \textbf{Locally constant}: $R|_{X_{i}}\equiv R_i \in \mathbb{N}$.
    \item \textbf{Polynomial tails}: there are constants $k\in \mathbb{N}$ and $\beta\in [0,1)$ with $k+\beta>1$ such that for any $n \in \mathbb{N}$,  \[\Leb_X(R>n)\approx n^{-k-\beta},\quad  \Leb_X(R=n)\precsim n^{-k-\beta-1}\]where the constants in $\approx,\precsim$ do not depend on $n$.
    \item \textbf{Dynamics}: $F:\Delta \to \Delta$ is defined so that it sends $(x,n)$ to $(x,n+1)$ if $n+1<R(x)$ and maps $X_{i}\times \{R_i-1\}$ into $X \times \{0\}$. Assume that the return map $F^R: X \to X$ maps each $X_i$ $C^1$-diffeomorphically onto its image. 
    \item \textbf{Expansion}: there is $\theta \in (0,1)$ such that $\inf_{v\in \mathbb{R}^d:|v|=1} |DF^R|_{X_i}v|\ge  \theta^{-1}$. 
    
     \item \textbf{Holes:} Define $S:=\bigcup_{i\ge 0}(F^R)^{-i}(\overline{\bigcup_j \partial X_j})$. Given a fixed point (called a center) $z_0 \in \bigcup_i X_i \setminus S$, a hole in this paper is referred to a metric ball $H \subsetneq X$ containing $z_0$. Define \[\diam_{\theta}H:=\sup_{x,y \in H}d_{\theta}(x,y), \text{ where }  d_{\theta}(x,y):=\theta^{s(x,y)} \text{ and}\] \[s(x,y):=\inf\{n\ge 0: (F^R)^nx, (F^R)^ny \text{ do not belong to the same }X_i\}.\]
     
     Let the transfer operator of $F^R$ be $T: L^1(X, \Leb_{X})\to L^1(X, \Leb_{X})$ and relevant operators, that is, for any $\phi \in L^1(X, \Leb_{X})$, \[(T\phi)(x):=\sum_{F^Ry=x}\frac{\phi(y)}{|\det DF^R(y)|},\] \[ {R}_n(\cdot):={T}(\mathbbm{1}_{\{R=n\}}\cdot), \quad {R}(z)(\cdot):={T}(z^R\cdot)=\sum_{n\ge 1}z^n{R}_n(\cdot).\]where $DF^R$ is the derivative w.r.t. $\Leb_X$. 
    Define a family of open transfer operators: for any  $z \in \overline{\mathbb{D}}$ and any hole $H \subsetneq X$,  \[\mathring{T}(\cdot):=T(\mathbbm{1}_{H^c}\circ F^R \cdot),\quad \mathring{R}_n(\cdot):=\mathring{T}(\mathbbm{1}_{\{R=n\}}\cdot), \quad \mathring{R}(z)(\cdot):=\mathring{T}(z^R\cdot)=\sum_{n\ge 1}z^n\mathring{R}_n(\cdot).\]
    
\item \textbf{Spectrum}: suppose that there is a Banach space $(B, ||\cdot||)$  compactly-embedded into $L^1(X, \Leb_{X})$ and constants $C\ge 1, \eta_s \in (0,1], \bar{\theta},\epsilon_1 \in (0,1)$ such that $\mathbbm{1}_{X} \in B$, and for any hole $H\subsetneq X$ with $\Leb_X(H) \le \epsilon_1$, any $\phi \in B$,$n \in \mathbb{N}$, $z \in \overline{\mathbb{D}}$, any measurable $E$, 
\begin{enumerate}
    \item Uniform Lasota-Yorke inequalities: \[||\mathring{R}(z)^n\phi ||\le C |z|^n\bar{\theta}^n ||\phi||+C|z|^n|\phi|_1, \quad ||{R}(z)^n\phi ||\le C |z|^n\bar{\theta}^n ||\phi||+C|z|^n|\phi|_1.\]
    \item Aperiodicity: $I-R(z): B \to B$ is invertible for any $z\in S^1\setminus \{1\}$.
    \item $||\mathring{R}_{n}\phi||\le C\Leb_X(\{R=n\})||\phi ||,\quad ||R_{n}\phi||\le C\Leb_X(\{R=n\})||\phi ||$.
    \item $|\mathbbm{1}_{\{R>n\}}\phi|_1\le C\Leb_X(R>n)\|\phi\|,\quad   |\mathbbm{1}_E\phi|_1\le C\Leb_X(E)^{\eta_s}\|\phi\|.$
\end{enumerate}

\item \textbf{Simple eigenvalues}: suppose that the only leading eigenvalue of $R(1)=T: B\to B$ on $S^1$ is $1$ and simple, its eigenvector is a positive function $h \in B$ with $h=C^{\pm 1}, \int_X h d\Leb_X=1$. The moduli of other spectrum points are strictly smaller than $1$.

Then $d\mu_X:=h d\Leb_X$ is an invariant SRB probability measure on $X$. Extend it to $\Delta$ by normalizing the measure $\sum_{i\ge 0}F_*^i(\mu_X|_{R>i})$, denote this SRB probability measure by $\mu_{\Delta}$. 
\end{enumerate}  
\end{definition}

\textbf{Escape rates}: in this open tower $\Delta$, the orbits escape $\Delta$ through the hole $H \subsetneq X$. One would like to investigate the leakage speed (called an escape rate), that is, the decay rate of  the tail of a first hitting time $\tau_H:=\inf\{n \ge 1: F^n \in H\}$. By the ergodic theorem, $\tau_H< \infty$ a.s. In our main result below, an asymptotic expansion for the escape rate is obtained, which, of course, implies the escape rates. 

\begin{remark}
    Definition \ref{defyoung} is a general framework for slowly mixing open dynamical systems. The first return map $F^R$ might not be a Gibbs-Markov map. Therefore, our tower $\Delta$ might not be a Young tower. The norm of $B$ might not dominate $|\cdot|_{\infty}$. Then a trade-off $\eta_s$ is assumed. This is crucial for two- dimensional intermittent maps in section \ref{appsection}.
\end{remark}

\section{Main theorems}\label{resultsection}
In this section, we present the asymptotic expansion of the decay of polynomial escape rates, where the coefficients are explicitly given. 
\begin{theorem}[Asymptotic expansions for polynomial escape rates]\ \label{thm}\ \par
Consider the system ($\Delta, F, \mu_{\Delta}$) in Definition \ref{defyoung} with $\Leb_X(R>n)\approx n^{-k-\beta}$ and a fixed center $z_0 \in \bigcup_i X_i \setminus S$. Then there are constants $C,\sigma, \epsilon>0$, for any fixed hole $H$ with $\mu_X(H) \in (0, \sigma)$,  we have 
\begin{enumerate}
\item $\mathring{T}$ has a simple isolated leading eigenvalue $\mathring{\lambda}(1) \in (0,1)$.
\item For any $n\in \mathbb{N}$,
\begin{align}\label{31}
    \frac{\mu_{\Delta}(\tau_H>n)}{\mu_{\Delta}(X)}=\sum_{i \ge n}\mu_X(R>i) +\frac{\mathring{\lambda}(1)+e_n(H)}{1-\mathring{\lambda}(1)} b_n+O_H(n^{-k-\beta-1}).
\end{align}
\end{enumerate}
Here $e_n(H)$ is a real sequence, possibly depending on $n$, such that
\[
 \sup_{n\ge1}|e_n(H)|\le C\bigl[(\diam_{\theta}H)^\epsilon
                         +1-\mathring{\lambda}(1)\bigr],
\]
where $C$ does not depend on $n$ or the hole $H$ satisfying $\mu_X(H)\in(0,\sigma)$. The exact value of $\epsilon>0$ can be found in subsection \ref{18}, and $b_n:=
\sum_{a+b=n, b>0}\mu_X(R> a)\mu_X(R\ge b)\approx n^{-k-\beta}$, $\sum_{i \ge n}\mu_X(R>i)\approx n^{-k-\beta+1}$ do not depend on $H$. The last error term $O_H(n^{-k-\beta-1})$ depends on the fixed hole $H$, decays faster than $b_n$. \(e_n(H)\)  controls the coefficient error as the hole shrinks, uniformly in \(n\). 
\end{theorem} 

\begin{remark}
    The general result in Theorem \ref{thm} shows that $\mu_{\Delta}(\tau_H>n)\approx n^{-k-\beta+1}$, which has been already proved for the class of intermittent maps in \cite{Demers,sujsp}. A new conclusion of Theorem \ref{thm} says that, no matter where $H$ is placed, the quantity $\mu_{\Delta}(\tau_H>n)$ has the same first-order term $\mu_{\Delta}(X)\sum_{i \ge n}\mu_X(R>i)\approx n^{-k-\beta+1}$. Only the second-order term makes a real difference in $\mu_{\Delta}(\tau_H>n)$. It has the same decay rate $b_n \approx n^{-k-\beta}$, but with a different coefficient $\frac{\mathring{\lambda}(1)+O[(\diam_{\theta}H)^{\epsilon}+1-\mathring{\lambda}(1)]}{1-\mathring{\lambda}(1)}$, which will be calculated explicitly for the examples in applications. 

    It was not assumed in \cite{Demers}, while dealing with intermittent maps, that the initial measure is  SRB-measure. As a result, different escape rates appear, which depend on regularities of the density functions of the initial measure. This special situation is discussed in the Remark \ref{noninvariant}.
 \end{remark}
 
 \begin{remark}
     In Theorem \ref{thm} we only consider a hole $H$ in the base $X$ of the tower. It is justified by the fact that in applications the base $X$ (for various classes of dynamical systems) can be chosen so that it contains a given hole in the original system. It differs from  \cite{demers1,demers2, BRUIN_DEMERS_MELBOURNE_2010}, where a hole in the Young tower is a union of infinitely many pieces which fill the tail of the tower. The construction of a tower in \cite{demers1,demers2, BRUIN_DEMERS_MELBOURNE_2010} is defined by a given single hole in the original system. However, in our applications,  we aim to compare two different holes in the original system. Then the construction of a single tower independent of two holes (e.g., the base $X$ includes these holes) gives the same first-order term for $\mu_{\Delta}(\tau_H>n)$ as shown in Theorem \ref{thm}, which is convenient to achieve our goal, and is simpler than the approach used in \cite{demers1,demers2, BRUIN_DEMERS_MELBOURNE_2010}.
 \end{remark}

 \begin{remark}
    Unlike \cite{Demersexp}, which was treating  $1-\mathring{\lambda}$ by verifying conditions in \cite{liveranikeller} directly, extra work is needed in section \ref{appsection} for specific models, specifically,  for the non-Markovian two dimensional intermittent maps. It will play an important role when we will deduce, where orbits prefer to go (visit) in section \ref{corproof}.
 \end{remark}
 
\section{Preliminary results}\label{prelimsection}
Before proving the main theorems, we need more preliminary results, including introducing the spectra associated with the operators in Definition \ref{defyoung}. Our scheme of the proof of Theorem \ref{thm} will be summarized at the end of Section \ref{prelimsection}. 

\subsection{Several (complex) operators}
Besides the operators in Definition \ref{defyoung}, now we define a family of new (complex) operators. For any $n \in \mathbb{N}_0, z\in \overline{\mathbb{D}}$, \[\mathring{R}_n(\cdot):=\mathring{T}(\mathbbm{1}_{\{R=n\}}\cdot), \quad \mathring{R}(z)(\cdot):=\sum_{n\ge 1}z^n\mathring{R}_n(\cdot)=\mathring{T}(z^R\cdot), \quad \mathring{R}_{\ge n}(\cdot):=\mathring{T}(\mathbbm{1}_{\{R\ge n\}}\cdot),\]
\[ \mathring{R}_{> n}(\cdot):=\mathring{T}(\mathbbm{1}_{\{R> n\}}\cdot), \quad \mathring{R}_{>}(z):=\sum_{n \ge 0}z^n\mathring{R}_{>n}, \quad \mathring{R}_{\ge }(z):=\sum_{n \ge 1}z^n\mathring{R}_{\ge n}.\]

In particular, when $\mu_X(H)=0$, 
\[{R}_n(\cdot):={T}(\mathbbm{1}_{\{R=n\}}\cdot), \quad {R}(z)(\cdot):=\sum_{n\ge 1}z^n{R}_n(\cdot)={T}(z^R\cdot), \quad {R}_{\ge n}(\cdot):={T}(\mathbbm{1}_{\{R\ge n\}}\cdot),\]
\[ {R}_{> n}(\cdot):={T}(\mathbbm{1}_{\{R> n\}}\cdot), \quad {R}_{>}(z):=\sum_{n \ge 0}z^n{R}_{>n}, \quad {R}_{\ge }(z):=\sum_{n \ge 1}z^n{R}_{\ge n}.\] These operators mapping $B$ into $B$ are well-defined. Extending $\Leb_X$ from $X$ to $\Delta$, and denoting it by $\Leb_{\Delta}:=\sum_{i\ge 0}F^i_*(\Leb_X|_{R>i})$, we can define a new transfer operator $\mathbb{T}:L^1(\Delta, \Leb_{\Delta})\to L^1(\Delta, \Leb_{\Delta})$ by \[(\mathbb{T}\phi )(x):=\sum_{F(y)=x}\frac{\phi(y)}{JF(y)},\] and an open transfer operator by \[\mathring{\mathbb{T}}(\cdot):=\mathbb{T}(\mathbbm{1}_{H^c}\circ F\cdot),\] where $JF$ is the Jacobian w.r.t. $\Leb_{\Delta}$. They have the following basic properties that can be proved directly from the definitions (we will skip these direct proofs). 
\begin{lemma}\label{basicopentransfer}
    $\int \psi\mathbb{T}(\phi) d\Leb_{\Delta}=\int \psi\circ F \phi d\Leb_{\Delta}$ for any $\phi \in L^1(\Delta, \Leb_{\Delta})$ and for any bounded $\psi$. $JF(a,i)=1$ if  $i<R(a)-1$ and $JF(a,i)=|\det DF^R(a)|$ when $i=R(a)-1$. $\mathring{\mathbb{T}}^n(\cdot)=\mathbb{T}^n(\mathbbm{1}_{\tau_H>n}\cdot)$ for any $n\in \mathbb{N}$.
\end{lemma}

\textbf{Convention}: from now on if $\mu_X(H)=0$, then we remove the symbol $``\circ"$ above the operators and call them closed operators.  If $\mu_X(H)>0$, we keep the symbol $``\circ"$ above the operators, referring to its dependence on a hole $H$ with a positive measure. We call them open operators.

Next, we consider some elementary properties of bounded linear operators $A(z):B\to B$. All power series used below are analytic on $\mathring{\mathbb{K}}$ and continuous on $\overline{\mathbb K}$ where $\mathbb K:=\mathbb D, \overline{\mathbb D}, S^1$.

\begin{definition}\label{defop}
Given $a\in\mathbb N$ and $b\in[0,1)$, we write $A(\cdot)\in C^{a+b}({\mathbb K})$ if its complex derivatives through order $a$ extend continuously to $\overline{\mathbb K}$ and its $a$th derivative is $b$-H\"older when $b>0$. When $b=0$, only continuity of the $a$th derivative is required. We write $A(\cdot)\in O(n^{-a-b})$ if for any $z\in \mathbb{D}$,
\[
 A(z)=\sum_{n\ge0}a_nz^n,\qquad \|a_n\|=O(n^{-a-b}).
\]
These coefficients equal $(2\pi)^{-1}\int_0^{2\pi}A(e^{it})e^{-int}\,dt$ for $n\ge0$. Negative Fourier coefficients vanish because $A$ is analytic in $\mathbb{D}$.
\end{definition}

The following properties will be used throughout the paper.
\begin{lemma}[see \cite{sarig, gouezel}]\label{frequentlemma}\ \par
    Given $a\in \mathbb{N}, b\in [0,1),  c\ge d>1$ and the operators $A(\cdot), B(\cdot), C(\cdot)$ defined in $\mathbb{K}$. \begin{enumerate}
        \item If $C(\cdot)\in O(n^{-d}), B(\cdot) \in O(n^{-c})$, then $C(\cdot) \circ B(\cdot) \in O(n^{-d})$.
        \item  If $A(\cdot)\in C^{a+b}(\mathbb{K})$, $\mathbb{K}=\mathbb{D}, \overline{\mathbb{D}}$, then $A(\cdot)\in O(n^{-a-b})$. If $\mathbb{K}=S^1$ and $a+b>1$ additionally, the Fourier coefficient $$a_n=\frac{1}{2\pi }\int_{0}^{2\pi}A(e^{it})e^{-int}dt=O(n^{-a-b})$$ where the constant in $``O(\cdot)"$ depends only on $\sup_{z\in S^1}||\frac{d^aA}{dz^a}(z)||, \sup_{z\in S^1}||\frac{dA}{dz}(z)||$ and the H\"older coefficient of $\frac{d^aA}{dz^a}\big|_{S^1}$.  
        \item If $A(\cdot)\in C^{a+b}(\mathbb{K})$, $A(\cdot)^{-1}$ exists and is continuous on $\overline{\mathbb{K}}$, then $A(\cdot)^{-1}\in C^{a+b}(\mathbb{K})$. 
    \end{enumerate} 
\end{lemma}
\begin{proof}
    The proofs of the first and second items can be found in Lemma 3 of \cite{sarig}, Proposition 2.4 and Lemma 4.3 of \cite{gouezel}. A remark about Lemma 3 of \cite{sarig} is that only the case $a=1,b\in(0,1)$ for $A(\cdot)$ was considered. In our case when $a\in \mathbb{N},b \in [0,1)$, if $\mathbb{K}=\mathbb{D}$ or $\overline{\mathbb{D}}$, we need the integration by parts, i.e., for any $r \in (0,1), n \ge a$ \begin{align*}
        a_n=\frac{1}{2\pi r^n}\int_{0}^{2\pi}A(re^{it})e^{-int}dt=\frac{r^a}{2\pi r^nP_n^a}\int_0^{2\pi}\frac{d^aA}{dz^a}(re^{it})e^{-i(n-a)t}dt.
    \end{align*} If $b=0$, then let $r\to 1$, and use the fact that $\frac{d^aA}{dz^a}(\cdot)$ can be continuously extended to $\overline{\mathbb{D}}$. If $b\in (0,1)$, we repeat the proof of Lemma 3 in \cite{sarig} for $\frac{d^aA}{dz^a}$. If $\mathbb{K}=S^1$ we simply let $r=1$. The proof of the dependence of the constant in $``O(\cdot)"$ for $``a_n=O(n^{-a-b})"$ can be found in the proof of Lemma 3 in \cite{sarig}, which is essentially reduced to integration by parts of Stieltjes integrals.

 For the regularity of $A(\cdot)^{-1}$, it is easy to see that $A(\cdot)^{-1}$ belongs to $C^1(\mathbb{K})$ due to $\sup_{z \in \mathbb{K}}||A(z)^{-1}||< \infty$, $(A(\cdot)^{-1})'=-A(\cdot)^{-1}\circ A(\cdot)' \circ A(\cdot)^{-1}$ and $A(\cdot) \in C^1(\mathbb{K})$. In order to prove more regularities, it suffices to notice that \[\frac{d^aA(z)^{-1}}{dz^a}=p\Big(A(z)^{-1}, \frac{d^aA}{dz^a}(z), \frac{d^{a-1}A}{dz^{a-1}}(z), \cdots, \frac{dA}{dz}(z)\Big),\]where $p$ is a polynomial. Observe that $\frac{d^aA}{dz^a}(z), \frac{d^{a-1}A}{dz^{a-1}}(z), \cdots, \frac{dA}{dz}(z)$ are $b$-H\"older (hence, continuous) on $\mathbb{K}$ and $A(z)^{-1}$ is continuous on $\overline{\mathbb{K}}$, then $$\sup_{z\in \overline{\mathbb{K}} } \Big\{||A(z)^{-1}||, \Big|\Big|\frac{d^aA}{dz^a}(z)\Big|\Big|, \Big|\Big|\frac{d^{a-1}A}{dz^{a-1}}(z)\Big|\Big|, \cdots, \Big|\Big|\frac{dA}{dz}(z)\Big|\Big|\Big\}<\infty.$$ Therefore, $\frac{d^aA(z)^{-1}}{dz^a}$ is $b$-H\"older. 
\end{proof}

\begin{lemma}
\label{renewalcenteredconvolution}
 Given $p>1$ and the operators $A(\cdot), D(\cdot), B(\cdot)$ defined in $\mathbb{K}$, suppose that $A(z)=\sum_{i \ge 0}a_i z^i, D(z)=\sum_{i \ge 0}d_i z^i, B(z)=\sum_{i \ge 0}b_i z^i$ and $A(\cdot) \in O(n^{-p}), B(\cdot) \in O(n^{-p}), D(\cdot) \in O(n^{-p-1})$ and 
\begin{align*}
 \sum_{n\ge0}d_n=0, \quad \|b_n-b_{n-1}\|\precsim n^{-p-1}, \quad \|a_n-a_{n-1}\|\precsim n^{-p-1} \text{ for any } n \ge 1.
\end{align*}
Then $A(\cdot) \circ D(\cdot) \in O(n^{-p-1}),  D(\cdot) \circ A(\cdot) \in O(n^{-p-1})$, and the sums of their Fourier coefficients satisfy
\[
 \sum_{n\ge0}\int_0^{2\pi} A(e^{it})\circ D(e^{it})e^{-int}dt =\sum_{n\ge0}\int_0^{2\pi} D(e^{it})\circ A(e^{it})e^{-int}dt =0.
\]
Replace $D(\cdot)$ by $A(\cdot)\circ D(\cdot)$, $A(\cdot)$ by $B(\cdot)$, we have $A(\cdot) \circ D(\cdot) \circ B(\cdot) \in O(n^{-p-1})$.
\end{lemma}
\begin{proof}
For $n\ge2$, using $\sum_{n \ge 0}d_n=0$, we have
\begin{align*}
 (2\pi)^{-1}&\int_0^{2\pi} A(e^{it})\circ D(e^{it})e^{-int}dt=\sum_{j=0}^{\lfloor n/2\rfloor}(a_{n-j}-a_n)d_j
       +\sum_{j=\lfloor n/2\rfloor+1}^{n}a_{n-j}d_j
       -a_n\sum_{j>\lfloor n/2\rfloor}d_j\\
       &=\sum_{j=0}^{\lfloor n/2\rfloor}O(j n^{-p-1})||d_j||
       +\sum_{j=\lfloor n/2\rfloor+1}^{n}||a_{n-j}||O(n^{-p-1})
       +O(n^{-p})\sum_{j>\lfloor n/2\rfloor}||d_j||
\end{align*}where the first term is due to 
$\|a_{n-j}-a_n\|\le \sum_{i\le j}||a_{n-i}-a_{n-i+1}|| \precsim j n^{-p-1}$ for any $j\le n/2$, the second term is due to $||d_j||=O(n^{-p-1})$ for any $j\ge n/2$.

Since
$\sum_jj\|d_j\|<\infty$,
$\sum_{l\ge0}\|a_l\|=O(1)$, $\sum_{j>\lfloor n/2\rfloor}||d_j||=O(n^{-p})$, we have $(2\pi)^{-1}\int_0^{2\pi} A(e^{it})\circ D(e^{it})e^{-int}dt=O(n^{-p-1})$. Similarly, $(2\pi)^{-1}\int_0^{2\pi} D(e^{it})\circ A(e^{it})e^{-int}dt=O(n^{-p-1})$.
Since $(2\pi)^{-1}\int_0^{2\pi} A(e^{it})D(e^{it})e^{-int}dt=\sum_{i+j=n}a_id_j$, hence
\[\sum_{n}(2\pi)^{-1}\int_0^{2\pi} A(e^{it})\circ D(e^{it})e^{-int}dt=\left(\sum_na_n\right)\left(\sum_nd_n\right)=0.
\] Similarly, we have $\sum_{n}(2\pi)^{-1}\int_0^{2\pi} D(e^{it})\circ A(e^{it})e^{-int}dt=0$. 
\end{proof}

\subsection{Regularities of $\mathring{R}(\cdot)$ and $R(\cdot)$}
Knowing the regularities of operators is crucial in the method of operator renewal theory. So the lemmas in this subsection dedicate to the regularities of $\mathring{R}(\cdot), R(\cdot)$. We claim not only the regularities of these operators but also their dependence on the hole $H$. We start with estimating the Fourier coefficients of  $\mathring{R}(\cdot), R(\cdot)$ and their dependence on small $H$. 

Recall the constants $\epsilon_1, C$ in Definition \ref{defyoung}, we have 

\begin{lemma}\label{Rnbound}
For any such hole $H$ satisfying $\mu_X(H)\le\epsilon_1/|h^{-1}|_{\infty}$,
\begin{gather*}
 \|\mathring R_n\|+\|R_n\|\le C\Leb_X(R=n)\precsim n^{-k-\beta-1},\\
 \max\{\|\mathring R_{\ge n}\|,\|R_{\ge n}\|\}
 \le C\Leb_X(R\ge n)\precsim n^{-k-\beta}.
\end{gather*}
\end{lemma}
\begin{proof}
The component bounds are from the spectrum property of Definition \ref{defyoung}. Summing them over $i\ge n$ gives the tail bounds.
\end{proof}

\begin{lemma}[Regularity of $\mathring R$ and $R$]\label{regulaofR}\ \par
 For any hole $H$ satisfying $\mu_X(H)\le \epsilon_1/|h^{-1}|_{\infty}$, $\mathring{R}_{\ge}(\cdot)$, $R_{>}(\cdot) \in O(n^{-k-\beta})$, $R(\cdot)$, $\mathring{R}(\cdot) \in O(n^{-k-\beta-1})$.
 
  If $\beta>0$, $R(\cdot), \mathring{R}(\cdot)\in C^{k+\beta}(\overline{\mathbb{D}})$, and these operators are analytic on $\mathbb{D}$. Moreover, $\sup_{j \le k}\sup_{\mu_X(H)\le \epsilon_1/|h^{-1}|_{\infty}}\sup_{z\in \overline{\mathbb{D}}}||\frac{d^j\mathring{R}(z)}{dz^j}||< \infty$, the $\beta$-H\"older coefficients of $\frac{d^k\mathring{R}(z)}{dz^k}$ do not depend on any hole $H$ satisfying $\mu_X(H)\le \epsilon_1/|h^{-1}|_{\infty}$. 
 
 If $\beta=0$, so that $k\ge2$, then
$R(\cdot),\mathring R(\cdot)\in C^{k-1+1/2}(\overline{\mathbb D})$ and these operators are analytic on $\mathbb{D}$. Moreover, $\sup_{j \le k-1}\sup_{\mu_X(H)\le \epsilon_1/|h^{-1}|_{\infty}}\sup_{z\in \overline{\mathbb{D}}}||\frac{d^j\mathring{R}(z)}{dz^j}||< \infty$, the $1/2$-H\"older coefficients of $\frac{d^{k-1}\mathring{R}(z)}{dz^{k-1}}$ do not depend on any hole $H$ satisfying $\mu_X(H)\le \epsilon_1/|h^{-1}|_{\infty}$. 
\end{lemma}
\begin{proof}
    From Lemma \ref{Rnbound} it follows that $\mathring{R}_{\ge}(\cdot), R_{>}(\cdot), \in O(n^{-k-\beta}) $, $R(\cdot), \mathring{R}(\cdot) \in O(n^{-k-\beta-1})$, so they are analytic on $\mathbb{D}$. 
    
    When $\beta>0$, direct calculations give $\frac{d^k\mathring{R}}{dz^k}(z)=\sum_{i}P_{i+k}^k\mathring{R}_{i+k}z^i$ for any $z\in \overline{\mathbb{D}}$, this convergence holds because, by Lemma \ref{Rnbound}, 
\begin{align*} 
\Big|\Big|\frac{d^k\mathring{R}}{dz^k}(z)\Big|\Big|&\le \sum_{i}P_{i+k}^k||\mathring{R}_{i+k}||\precsim \sum_{i\ge 1}i^k||\mathring{R}_{i+k}|| \precsim \sum_{i \ge 1} \frac{i^{k}}{i^{k+\beta+1}}< \infty,
\end{align*}where the constants in ``$\precsim"$ do not depend on any hole $H$ satisfying $\mu_X(H)\le \epsilon_1/|h^{-1}|_{\infty}$. Same arguments give $\sup_{j \le k}\sup_{\mu_X(H)\le \epsilon_1/|h^{-1}|_{\infty}}\sup_{z\in \overline{\mathbb{D}}}||\frac{d^j\mathring{R}(z)}{dz^j}||< \infty$. Then for any distinct $z_1,z_2 \in \overline{\mathbb{D}}$, any $u\in \mathbb{N}$, \begin{align*}
    \Big|\Big|\frac{d^k\mathring{R}}{dz^k}&(z_1)-\frac{d^k\mathring{R}}{dz^k}(z_2)\Big|\Big|
\precsim \sum_{i\ge u}P_{i+k}^k||\mathring{R}_{i+k}||+\sum_{1\le i\le u}iP_{i+k}^k||\mathring{R}_{i+k}|||z_1-z_2|\\
&\precsim \sum_{i\ge u}i^k||\mathring{R}_{i+k}||+\sum_{1\le i\le u}i^{k+1}||\mathring{R}_{i+k}|||z_1-z_2|\\
&\precsim \sum_{i \ge u}[i^k-(i-1)^{k}]\sum_{j \ge i}||\mathring{R}_{j+k}||+(u-1)^{k}\sum_{j \ge u}||\mathring{R}_{j+k}||\\
& \quad + \sum_{1 \le i \le u}[i^{k+1}-(i-1)^{k+1}]\sum_{j \ge i}||\mathring{R}_{j+k}|||z_1-z_2|\\
&\precsim \sum_{i \ge u}\frac{i^{k-1}}{i^{k+\beta}}+ u^{-\beta} +\sum_{1 \le i \le u}\frac{i^k}{i^{k+\beta}}|z_1-z_2|\\
&\precsim u^{-\beta}+u^{1-\beta}|z_1-z_2|\precsim |z_1-z_2|^{\beta},
\end{align*}where the last ``$\precsim$" is due to the choice $
u=\left\lceil |z_1-z_2|^{-1}\right\rceil
$, and the constants in ``$\precsim$" do not depend on any hole $H$ satisfying $\mu_X(H)\le \epsilon_1/|h^{-1}|_{\infty}$. Hence, $\mathring{R}(\cdot)\in C^{k+\beta}(\overline{\mathbb{D}})$. The same arguments give $R(\cdot)\in C^{k+\beta}(\overline{\mathbb{D}})$.

If $\beta=0$, we follow the same argument by replacing $k, \beta$ above with $k-1, 1/2$. This completes the proof.
\end{proof}
We end this subsection with a simple lemma.
\begin{lemma}\label{renewaleq1} 
    When $\mu_{X}(H) \in (0,\epsilon_1/|h^{-1}|_{\infty})$, $[I-\mathring{R}(z)]^{-1}$ is analytic and equal to $\sum_{m \ge 1}\mathring{R}(z)^{m-1}$ for any $|z|< 1$.
\end{lemma}
\begin{proof}
    When $|z|< 1$, the uniform Lasota-Yorke inequalities in Definition \ref{defyoung} imply that the spectral radius of $\mathring{R}(z)$ is at most $|z|$. Hence $\sum_{m \ge 1}\mathring{R}(z)^{m-1}$ converges and is equal to $[I-\mathring{R}(z)]^{-1}$. The fact that $\mathring{R}(z)$ is analytic on $\mathbb{D}$ implies that $[I-\mathring{R}(z)]^{-1}$ is also analytic on $\mathbb{D}$.
\end{proof}

\subsection{Spectral analysis of $[I-\mathring{R}(z)]^{-1}$ for any $z \in \overline{\mathbb{D}}$ and a small hole $H$}

To analyze the structures of the spectrum, we need the uniform Lasota-Yorke inequalities (Definition \ref{defyoung}), as well as the regularities of $\mathring{R}(\cdot), R(\cdot)$ (see Lemma \ref{regulaofR}). So, from now on,  we only consider a small $H$ with $\mu_X(H)\in [0, \epsilon_1/|h^{-1}|_{\infty})$ same as Lemma \ref{regulaofR}.

It follows from Hennion's theorem that for any $z \in \overline{\mathbb{D}}$, $\sigma(R(z))$, $\sigma(\mathring{R}(z))$ consist of an essential spectrum contained in $\{w\in \overline{\mathbb{D}}: |w|\le \bar{\theta}\}$ and of finitely many eigenvalues in $\{w\in \overline{\mathbb{D}}: |w|\ge r\}$ for any $r>\bar{\theta}$. By Lemma \ref{renewaleq1}, $[I-\mathring{R}(z)]^{-1}$ is analytic on $\mathbb{D}$. Hence, we only need to study the spectrum of $\mathring{R}(z)$ when $z$ is in the vicinity of $S^1$. 

In order to do it, we will verify the condition of Proposition \ref{extendliverani1} in Section \ref{appendix} Appendix, that is, the weak perturbation w.r.t. the weak norm $|||\cdot |||$.

\begin{lemma}
For any $z\in \overline{\mathbb{D}}$ in the vicinity of a given $e^{it}\in S^1$, we have
    \[|||\mathring{R}(z)(\cdot)-R(e^{it})(\cdot)|||\precsim |z-e^{it}|+ \mu_X(H)^{\eta_s}\]where the constant in $\precsim$ does not depend on $z,t$ and any $\mu_X(H) \in [0,\epsilon_1/|h^{-1}|_{\infty})$.
\end{lemma}
\begin{proof}
 We can estimate $|||\mathring{R}(z)(\cdot)-R(e^{it})(\cdot)|||=\sup_{||\phi||\le 1}\int |T(z^R\mathbbm{1}_{H^c}\circ F^R\phi)-T(e^{itR} \phi)|d\Leb_X$ as follows. Using the spectrum property in Definition \ref{defyoung}, 
\begin{align*}
    &\le \sup_{||\phi||\le 1}\int |\mathring{R}(z)\phi-\mathring{R}(e^{it})\phi|d\Leb_X +\int |T(e^{itR}\mathbbm{1}_{H^c}\circ F^R\phi)-T(e^{itR} \phi)|d\Leb_X\\
     &\le \sup_{||\phi||\le 1}||\mathring{R}(z)\phi-\mathring{R}(e^{it})\phi||+\int |T(e^{itR}\mathbbm{1}_{H}\circ F^R\phi)|d\Leb_X\\
     &\le \sup_{||\phi||\le 1}\sum_{n}|z^n-e^{itn}|||\mathring{R}_n(\phi)||+|h^{-1}|_{\infty}\int \mathbbm{1}_H \circ F^R |\phi| d\mu_X \\
     &\precsim |z-e^{it}|\sum_{n}n \Leb_X(R=n)+ \mu_X(H)^{\eta_s}\precsim |z-e^{it}|+ \mu_X(H)^{\eta_s},
\end{align*}where the first $\precsim$ is due to $|z^n-e^{itn}|\le n|z-e^{it}|$.
\end{proof}

Hence, all the conditions of Proposition \ref{extendliverani1} are completely verified. Roughly speaking, Proposition \ref{extendliverani1} says that when $H$ is small enough and $z \approx e^{it}$, $\mathring{R}(z), R(e^{it})$ have similar spectra. Hence, Proposition \ref{extendliverani1} and the spectrum information in Definition \ref{defyoung} lead to the following local structure of the spectrum of $\mathring{R}(\cdot), R(\cdot)$ and of their global regularities, where $\mathring{R}(\cdot), R(\cdot)$ are restricted in the vicinity of $S^1$ inside $\overline{\mathbb{D}}$.

\begin{lemma}[Local structure of the spectrum]\label{localspectrumpic}\ \par The structure of the spectrum  of $\mathring{R}(u), R(u)$ can be described in several cases, where $u$ is always restricted in $\overline{\mathbb{D}}$.

   \begin{enumerate}
       \item  For any $z \in S^1\setminus \{1\}$, there are constants $\sigma_z, \sigma'_z>0$ such that for any hole satisfying $\mu_X(H)\le \sigma_z$ and any $\dist(u,z)\le \sigma'_z$ with $\dist(1,z)>\sigma'_z$, $[I-\mathring{R}(u)]^{-1}, [I-R(u)]^{-1}$ exist and \begin{gather*}
           \sup_{\mu_X(H)\le \sigma_z}\sup_{u \in \overline{B_{\sigma_z'}(z)}}\max\{||[I-\mathring{R}(u)]^{-1}||, ||[I-R(u)]^{-1}||\}< \infty,\\
            [I-\mathring{R}(\cdot)]^{-1}, [I-R(\cdot)]^{-1}\in \begin{cases}
 C^{k+\beta}(\overline{B_{\sigma'_z}(z)}\cap\overline{\mathbb D}),&\beta>0,\\
 C^{k-1+1/2}(\overline{B_{\sigma'_z}(z)}\cap\overline{\mathbb D}),&\beta=0.
 \end{cases}
       \end{gather*}
       \item When $z=1$, there are fixed numbers $\delta \in (0, 1-\bar{\theta}^{1/4}), \theta' \in (\sqrt{\bar{\theta}}, (1-\delta)^2)$ (depending on $R(1)$ only and satisfying $\max\{\theta^{\frac{\eta_s(k+\beta-1)}{2(k+\beta)}},\bar{\theta}\}<(1-\delta)^4$) and  constants $\sigma_1, \sigma'_1>0$,  such that for any $0< \mu_X(H)\le \sigma_1$ and $\dist(u,1)\le \sigma'_1$, $[\partial B_{\delta}(1) \bigcup \partial B_{\theta'}(0)]\cap \sigma(\mathring{R}(u))=\emptyset$, $\sigma(\mathring{R}(u))\subseteq B_{\delta}(1)\bigcup B_{\theta'}(0)$ and 
       \[\mathring{R}(u)=\mathring{\lambda}(u)\mathring{\Proj}(u)+\mathring{Q}(u).\]
       \[
       \mathring\lambda(\cdot),\mathring\Proj(\cdot),\mathring Q(\cdot)
       \in \begin{cases}
 C^{k+\beta}(\overline{B_{\sigma'_1}(1)}\cap\overline{\mathbb D}),&\beta>0,\\
 C^{k-1+1/2}(\overline{B_{\sigma'_1}(1)}\cap\overline{\mathbb D}),&\beta=0.
 \end{cases}
       \] 
       where $|\mathring{\lambda}(u)|<1$ is the leading eigenvalue of $\mathring{R}(u)$, and $\mathring{\Proj}(u):=\frac{1}{2\pi i}\int_{\partial B_{\delta}(1)}[z-\mathring{R}(u)]^{-1}dz$ is a one-dimensional projection. Moreover, $\mathring{Q}^n(u):=\frac{1}{2\pi i} \int_{\partial B_{\theta'}(0)}z^n[z-\mathring{R}(u)]^{-1}dz$ and the spectral radius of $\mathring{Q}(u)$ is not greater than $\theta'$, \begin{equation}\label{boundforprojlambdaq}
       \begin{aligned}
       &\sup_{\mu_X(H)\in(0,\sigma_1]}
       \sup_{u\in\overline{B_{\sigma'_1}(1)}\cap\overline{\mathbb D}}
       \max\left\{
       \left\|\frac{d^j\mathring\Proj(u)}{du^j}\right\|,
       \left|\frac{d^j\mathring\lambda(u)}{du^j}\right|,
       \left\|\frac{d^j\mathring Q(u)}{du^j}\right\|
       \right\}<\infty
       \end{aligned}
       \end{equation}holds for $j\le k$ when $\beta>0$ and for $j\le k-1$ when $\beta=0$.
       If $\beta>0$, the $k$th derivatives of
       $\mathring\Proj$, $\mathring\lambda$ and $\mathring Q$ have
       uniformly bounded $\beta$-H\"older coefficients.
       If $\beta=0$, their $(k-1)$st derivatives have uniformly bounded
       $1/2$-H\"older coefficients. In both cases, the derivative bounds
       and the H\"older bounds depend only on
       $\delta,\theta',\sigma_1,\sigma'_1$, and not on the hole
       $H$ with $\mu_X(H)\in(0,\sigma_1]$.
       Furthermore, we have the following for any $n\ge 1$, $0<\mu_X(H)\le \sigma_1$ and $\dist(u,1)\le \sigma'_1$,
       \begin{gather}
           [I-\mathring{R}(u)]^{-1}=\frac{1}{1-\mathring{\lambda}(u)}\mathring{\Proj}(u)+[I-\mathring{Q}(u)]^{-1}[I-\mathring{\Proj}(u)]\label{localnear1}\\
           \sup_{\mu_X(H)\in (0,\sigma_1],  \dist(u,1)\le \sigma'_1}||\mathring{Q}^n(u)||=O_{\sigma_1,\sigma_1'}(\theta'^n)\label{essspectrumradius}
       \end{gather}
       For each fixed hole in this range,
       \[
       [I-\mathring R(\cdot)]^{-1}\in \begin{cases}
 C^{k+\beta}(\overline{B_{\sigma'_1}(1)}\cap\overline{\mathbb D}),&\beta>0,\\
 C^{k-1+1/2}(\overline{B_{\sigma'_1}(1)}\cap\overline{\mathbb D}),&\beta=0.
 \end{cases}
       \]
       \item In particular, if $\mu_X(H)=0$, for any $\dist(u,1)\le \sigma'_1$, $[\partial B_{\delta}(1)\bigcup \partial B_{\theta'}(0)]\cap \sigma(R(u))=\emptyset$, $\sigma(R(u))\subseteq B_{\delta}(1)\bigcup B_{\theta'}(0)$ and \[R(u)=\lambda(u)\Proj(u)+Q(u).\]
       \[
       \lambda(\cdot),\Proj(\cdot),Q(\cdot)\in \begin{cases}
 C^{k+\beta}(\overline{B_{\sigma'_1}(1)}\cap\overline{\mathbb D}),&\beta>0,\\
 C^{k-1+1/2}(\overline{B_{\sigma'_1}(1)}\cap\overline{\mathbb D}),&\beta=0.
 \end{cases}
       \]
       where $|\lambda(u)|\le 1$ is the leading eigenvalue of $R(u)$, and $\Proj(u):=\frac{1}{2\pi i}\int_{\partial B_{\delta}(1)}[z-R(u)]^{-1}dz$ is a one-dimensional projection. $\lambda(1)=1$, and $\lambda(u)\neq 1$ if $u\neq 1$. The spectral radius of $Q(u):=\frac{1}{2\pi i} \int_{\partial B_{\theta'}(0)}z[z-R(u)]^{-1}dz$ is not greater than $\theta'$, \[
       \begin{aligned}
       &\sup_{u\in\overline{B_{\sigma'_1}(1)}\cap\overline{\mathbb D}}
       \max\left\{
       \left\|\frac{d^j\Proj(u)}{du^j}\right\|,
       \left|\frac{d^j\lambda(u)}{du^j}\right|,
       \left\|\frac{d^j Q(u)}{du^j}\right\|
       \right\}<\infty
       \end{aligned}
       \]holds for $j\le k$ when $\beta>0$ and for $j\le k-1$ when $\beta=0$.
       If $\beta>0$, the $k$th derivatives of $\Proj$, $\lambda$ and $Q$
       have bounded $\beta$-H\"older coefficients.
       If $\beta=0$, their $(k-1)$st derivatives have bounded
       $1/2$-H\"older coefficients. These derivative and H\"older bounds
       depend only on $\delta,\theta',\sigma'_1$.
   \end{enumerate} 
\end{lemma}

\begin{proof}
\textit{Step 1: we prove item 1.} If $z\in S^1\setminus\{1\}$, then $R(z)$ does not have an eigenvalue $1$ according to Definition \ref{defyoung}. Apply now Proposition \ref{extendliverani1} to $o=(z,0)$ and $R(z)$, then there are small $\sigma_z, \sigma_z'>0$ such that for any hole satisfying $\mu_X(H)\le \sigma_z$ and any $ \dist(u,z)\le \sigma_z'$ (such that it excludes $u=1$), $[I-\mathring{R}(u)]^{-1}$ and $[I-R(u)]^{-1}$ exist and 
\[\sup_{\mu_X(H)\le \sigma_z} \sup_{\dist(u,z)\le \sigma_z'}\max\{||[I-\mathring{R}(u)]^{-1}||, ||[I-R(u)]^{-1}||\}< \infty.\]  Continuity of $[I-\mathring{R}(\cdot)]^{-1}$ and $[I-R(\cdot)]^{-1}$ follows from this uniform bound and from the following equalities \begin{gather}
    [I-\mathring{R}(u_1)]^{-1}-[I-\mathring{R}(u_2)]^{-1}=[I-\mathring{R}(u_1)]^{-1}[\mathring{R}(u_1)-\mathring{R}(u_2)][I-\mathring{R}(u_2)]^{-1},\nonumber\\
[I-R(u_1)]^{-1}-[I-R(u_2)]^{-1}=[I-R(u_1)]^{-1}[R(u_1)-R(u_2)][I-R(u_2)]^{-1}\label{resolentidentity}
\end{gather}which hold for any $\dist(u_1,z)\le \sigma_z'$ and $\dist(u_2,z)\le \sigma_z'$. By combining Lemma \ref{frequentlemma} with Lemma \ref{regulaofR} we obtain
\[
[I-\mathring R(\cdot)]^{-1},[I-R(\cdot)]^{-1}\in \begin{cases}
 C^{k+\beta}(\overline{B_{\sigma'_z}(z)}\cap\overline{\mathbb D}),&\beta>0,\\
 C^{k-1+1/2}(\overline{B_{\sigma'_z}(z)}\cap\overline{\mathbb D}),&\beta=0.
 \end{cases}
\]
Indeed, for $\beta>0$ we differentiate the inverse identities through
order $k$ and use the $\beta$-H\"older bound in Lemma \ref{regulaofR}.
For $\beta=0$ we differentiate only through order $k-1$ and use its
$1/2$-H\"older bound. The resolvent bounds above control every inverse
factor in these formulas. 
\medskip
\noindent

\textit{Step 2: prove the spectrum decompositions in item 2 and item 3.} If $z=1$, then the simple eigenvalue in Definition \ref{defyoung} guarantees that there are constants $\delta>0, \theta'\in (\sqrt{\bar{\theta}},(1-\delta)^2)$ such that $B_{\delta}(1)\bigcap \sigma(R(1))=B_{\theta'}(0)^c \bigcap \sigma(R(1))=\{1\}$. (Here we choose sufficiently small $\delta>0$ such that $\max\{\theta^{\frac{\eta_s(k+\beta-1)}{2(k+\beta)}},\bar{\theta}\}<(1-\delta)^4$. A reason of such choice will be clear in the proof of Lemma \ref{lastapproximateproj}). Let $V_{\delta, \theta'}=B_{\theta'}(0) \bigcup B_{\delta}(1)$ be the domain used in Proposition \ref{extendliverani1}, $o=(1,0)$ and notice that $R(1)$ has a simple eigenvalue $1$ with a one-dimensional eigenspace. Applying these to Proposition \ref{extendliverani1}, we get that there are small constants $\sigma_1, \sigma'_1>0$  such that for any hole satisfying $\mu_X(H)\le \sigma_1$ and any $\dist(u,1)\le \sigma'_1$, 
\[\mathring{R}(u)=\mathring{\lambda}(u)\mathring{\Proj}(u)+\mathring{Q}(u),\quad  R(u)=\lambda(u)\Proj(u)+Q(u)\]
\[|\mathring{\Proj}(u)h-\Proj(1)h|_1\le |||\mathring{\Proj}(u)-\Proj(1)|||||h||\le C_{\delta, \theta'} [|u-1|+\mu_X(H)^{\eta_s}]^{1/2} \le 1/2\]
\[|\Proj(u)h-\Proj(1)h|_1\le |||\Proj(u)-\Proj(1)|||||h||\le C_{\delta, \theta'} |u-1|^{1/2} \le 1/2\]
where the spectral radius of $Q(u), \mathring{Q}(u)$ is not greater than $\theta'$, and $\Proj(u), \mathring{\Proj}(u)$ are one-dimensional, and \[\sup_n\sup_{\mu_X(H)\le \sigma_1}\sup_{u\in \overline{B_{\sigma_1'}(1)}}\max\{||\mathring{\Proj}(u)||, ||\Proj(u)||, ||Q^n(u)||/\theta'^n, ||\mathring{Q}^n(u)||/\theta'^n\}< \infty\] and this bound depends on $\delta, \theta', \sigma_1, \sigma_1'$ only. $\Proj(1)h=h$ implies that 
\begin{gather}\label{26}
\Big|\int \Proj(u)hd\Leb_X\Big|\ge 1/2, \quad \Big|\int\mathring{\Proj}(u)hd\Leb_X\Big|\ge 1/2
\end{gather}
for any hole satisfying $\mu_X(H)\le \sigma_1$ and any $\dist(u,1)\le \sigma'_1$. Notice that $\mathring{R}^n(u)\mathring{\Proj}(u)h=\mathring{\lambda}^n(u)\mathring{\Proj}(u)h$ and $R^n(u)\Proj(u)h=\lambda^n(u)\Proj(u)h$ imply 
\[\int u^{\sum_{0\le i \le n-1}R\circ (F^R)^i}\mathbbm{1}_{\bigcap_{1\le i \le n}(F^R)^{-i}H^c}\mathring{\Proj}(u)hd\Leb_X= \mathring{\lambda}^n(u)\int \mathring{\Proj}(u)hd\Leb_X,\]
\[\int u^{\sum_{0\le i \le n-1}R\circ (F^R)^i}\Proj(u)hd\Leb_X= \lambda^n(u)\int \Proj(u)hd\Leb_X\]for any $n \ge 1$. By letting $n \to \infty$ we have that $ |\mathring{\lambda}(u)|< 1$ holds for any hole satisfying $\mu_X(H)\in (0, \sigma_1]$ and any $\dist(u,1)\le \sigma'_1$, $|\lambda(u)|\le 1$ holds for any $\dist(u,1)\le \sigma'_1$. ``$\lambda(1)=1$" and ``$u\neq 1$ iff $\lambda(u)\neq 1$" follow from the aperiodicity of Definition \ref{defyoung}.
\medskip
\noindent

\textit{Step 3: prove (\ref{localnear1}).} In difference with \cite{sarig, gouezel, melbourneinvention},  (\ref{localnear1}) at $u=1$ makes sense due to $|\mathring{\lambda}(u)|<1$. When $\mu_X(H)\in (0,\sigma_1]$ and $\dist(u,1)\le \sigma'_1$, we verify it directly \begin{align*}
    [I-\mathring{R}(u)] &\Big[\frac{1}{1-\mathring{\lambda}(u)}\mathring{\Proj}(u)+[I-\mathring{Q}(u)]^{-1}[I-\mathring{\Proj}(u)]\Big]\\
    &=\frac{\mathring{\Proj}(u)-\mathring{\lambda}(u)\mathring{\Proj}(u)}{1-\mathring{\lambda}(u)}+[I-\mathring{R}(u)] [I-\mathring{Q}(u)]^{-1}[I-\mathring{\Proj}(u)]\\
    &=\mathring{\Proj}(u)+[I-\mathring{R}(u)][I+\sum_{i\ge 1}\mathring{Q}(u)^i][I-\mathring{\Proj}(u)]\\
    &=\mathring{\Proj}(u)+[I-\mathring{R}(u)][I-\mathring{\Proj}(u)+\sum_{i\ge 1}\mathring{Q}(u)^i]\\
    &=\mathring{\Proj}(u)+[I-\mathring{R}(u)][I-\mathring{\Proj}(u)]+\mathring{Q}(u)\\
    &=\mathring{\Proj}(u)+[I-\mathring{R}(u)]-\mathring{\Proj}(u)+\mathring{\lambda}(u)\mathring{\Proj}(u)+\mathring{Q}(u)=I.
\end{align*} 

The identity of (\ref{localnear1}) is concluded by noting that $[I-\mathring{R}(u)]^{-1}$ exists.
\medskip
\noindent

\textit{Step 4: prove the regularities in item 2 and item 3.}  From Proposition \ref{extendliverani1}, \begin{gather}\label{5}
    \sup_{\mu_X(H)\in (0,\sigma_1]}\sup_{z \in \partial B_{\delta}(1)\cup \partial B_{\theta'}(0), u \in \overline{B_{\sigma_1'}(1)}}\max\{||[z-\mathring{R}(u)]^{-1}||, ||[z-R(u)]^{-1}||\}< \infty.
\end{gather} 

Combining (\ref{5}) and Lemma \ref{frequentlemma} with Lemma \ref{regulaofR}, these give
\[
\Proj(\cdot),\mathring\Proj(\cdot)\in \begin{cases}
 C^{k+\beta}(\overline{B_{\sigma'_1}(1)}\cap\overline{\mathbb D}),&\beta>0,\\
 C^{k-1+1/2}(\overline{B_{\sigma'_1}(1)}\cap\overline{\mathbb D}),&\beta=0.
 \end{cases}
\] Notice that \begin{gather}\label{6}
    \mathring{\lambda}(u)=\frac{\int \mathring{R}(u)\mathring{\Proj}(u)hd\Leb_X}{\int \mathring{\Proj}(u)hd\Leb_X}, \quad \lambda(u)=\frac{\int R(u)\Proj(u)hd\Leb_X}{\int \Proj(u)hd\Leb_X}
\end{gather}
The denominators are bounded away from zero by (\ref{26}).
Products and reciprocals therefore preserve the regularity of the
projections in both cases, and
\[
\lambda(\cdot),\mathring\lambda(\cdot)\in \begin{cases}
 C^{k+\beta}(\overline{B_{\sigma'_1}(1)}\cap\overline{\mathbb D}),&\beta>0,\\
 C^{k-1+1/2}(\overline{B_{\sigma'_1}(1)}\cap\overline{\mathbb D}),&\beta=0.
 \end{cases}
\]
Since $Q=R-\lambda\Proj$ and
$\mathring Q=\mathring R-\mathring\lambda\mathring\Proj$, we also have
\[
Q(\cdot),\mathring Q(\cdot)\in \begin{cases}
 C^{k+\beta}(\overline{B_{\sigma'_1}(1)}\cap\overline{\mathbb D}),&\beta>0,\\
 C^{k-1+1/2}(\overline{B_{\sigma'_1}(1)}\cap\overline{\mathbb D}),&\beta=0.
 \end{cases}
\]
For $0\le j\le k$ when $\beta>0$, and for $0\le j\le k-1$ when
$\beta=0$, differentiation of the contour integrals gives
\[\frac{d^j\mathring{\Proj}(u)}{du^j}=\frac{1}{2\pi i}\int_{\partial B_{\delta}(1)}p\Big([z-\mathring{R}(u)]^{-1}, \frac{d\mathring{R}(u)}{du},\cdots, \frac{d^j\mathring{R}(u)}{du^j}\Big) dz \]
\[\frac{d^j\mathring{Q}(u)}{du^j}=\frac{1}{2\pi i}\int_{\partial B_{\theta'}(0)}zp\Big([z-\mathring{R}(u)]^{-1}, \frac{d\mathring{R}(u)}{du},\cdots, \frac{d^j\mathring{R}(u)}{du^j}\Big) dz \]
\[\frac{d^j{\Proj}(u)}{du^j}=\frac{1}{2\pi i}\int_{\partial B_{\delta}(1)}p\Big([z-{R}(u)]^{-1}, \frac{d{R}(u)}{du},\cdots, \frac{d^j{R}(u)}{du^j}\Big) dz \]
\[\frac{d^j{Q}(u)}{du^j}=\frac{1}{2\pi i}\int_{\partial B_{\theta'}(0)}zp\Big([z-{R}(u)]^{-1}, \frac{d{R}(u)}{du},\cdots, \frac{d^j{R}(u)}{du^j}\Big) dz \]
where $p$ is a polynomial. Lemma \ref{regulaofR} and (\ref{5}) give uniform bounds for all the
derivatives in these ranges. If $\beta>0$, take differences of the
formulas with $j=k$: Lemma \ref{regulaofR} bounds the differences of
the $k$th derivatives by a constant times $|u_1-u_2|^\beta$.
If $\beta=0$, use the formulas with $j=k-1$ and the bound by a constant
times $|u_1-u_2|^{1/2}$. In each case, differences of inverse factors
are controlled by (\ref{resolentidentity}) and (\ref{5}), and lower derivatives are
Lipschitz on the parameter domain. Thus the highest derivatives have
the asserted $\beta$-H\"older or $1/2$-H\"older bounds, uniformly for
$\mu_X(H)\in(0,\sigma_1]$. Applying the same product and reciprocal
estimates to (\ref{6}), with (\ref{26}) and Lemma \ref{regulaofR}, gives the corresponding bounds
for $\mathring\lambda$ and $\lambda$. The identity (\ref{localnear1}) gives the regularity
\[
[I-\mathring R(\cdot)]^{-1}\in \begin{cases}
 C^{k+\beta}(\overline{B_{\sigma'_1}(1)}\cap\overline{\mathbb D}),&\beta>0,\\
 C^{k-1+1/2}(\overline{B_{\sigma'_1}(1)}\cap\overline{\mathbb D}),&\beta=0.
 \end{cases}
\]
\end{proof}

\begin{lemma}[Global regularities]\label{globalregularity}\ \par
There is a small constant $\sigma\in(0,\epsilon_1/|h^{-1}|_{\infty})$ (see $\epsilon_1$
in Lemma \ref{regulaofR}) such that, for every hole $H$ with
$\mu_X(H)\in(0,\sigma]$,
\[
[I-\mathring R(\cdot)]^{-1}\in
\begin{cases}
C^{k+\beta}(\overline{\mathbb D}),&\beta>0,\\
C^{k-1+1/2}(\overline{\mathbb D}),&\beta=0.
\end{cases}
\]
Thus, when $\beta>0$, the inverse has derivatives through order $k$
and its $k$th derivative is $\beta$-H\"older; when $\beta=0$, it has
derivatives through order $k-1$ and its $(k-1)$st derivative is
$1/2$-H\"older. These regularity bounds may depend on the fixed hole.
\end{lemma}
\begin{proof}
By Lemma \ref{localspectrumpic}, for each $z\in S^1$ there are
$\sigma_z,\sigma'_z>0$ such that, whenever
$\mu_X(H)\in(0,\sigma_z]$,
\[
[I-\mathring R(\cdot)]^{-1}\in
\begin{cases}
C^{k+\beta}(\overline{\mathbb D}\cap\overline{B_{\sigma'_z}(z)}),
 &\beta>0,\\
C^{k-1+1/2}(\overline{\mathbb D}\cap\overline{B_{\sigma'_z}(z)}),
 &\beta=0.
\end{cases}
\]
Since $S^1$ is compact, finitely many of these neighborhoods, centered
at $z_1,\ldots,z_m$, cover $S^1$. Set
$\sigma=0.5\min\{\sigma_{z_1},\ldots,\sigma_{z_m},\epsilon_1/|h^{-1}|_{\infty}\}$ and fix a
hole with $0<\mu_X(H)\le\sigma$. On each neighborhood, the inverse has
$k$ derivatives with a $\beta$-H\"older highest derivative if
$\beta>0$, and $k-1$ derivatives with a $1/2$-H\"older highest
derivative if $\beta=0$.

The inverse is analytic in $\mathbb D$. The local derivatives agree
on overlaps, since they are derivatives of the same inverse.
Combining the finitely many boundary neighborhoods with a compact
subset of $\mathbb D$ gives bounded derivatives and a global H\"older
bound of the same order: sufficiently close pairs lie in one member
of this finite cover, and for the remaining pairs boundedness of the
highest derivative gives the H\"older estimate. Equivalently, using
Lemma \ref{frequentlemma}, we obtain
\[
[I-\mathring R(\cdot)]^{-1}\in
\begin{cases}
C^{k+\beta}(\overline{\mathbb D}),&\beta>0,\\
C^{k-1+1/2}(\overline{\mathbb D}),&\beta=0.
\end{cases}
\]
Restriction to the circle consequently gives
\[
[I-\mathring R(\cdot)]^{-1}|_{S^1}\in
\begin{cases}
C^{k+\beta}(S^1),&\beta>0,\\
C^{k-1+1/2}(S^1),&\beta=0.
\end{cases}
\]
\end{proof}

\subsection{Scheme of the proof}
With all preparations in this section,  we can outline the scheme of the proof of Theorem \ref{thm}. More details will be clear in Section \ref{proof}.
\begin{enumerate}
    \item We will express $\mu_{\Delta}(\tau_H>n)$ in terms of $\mathring{R}(\cdot)$, $R$, $\mathring{R}_{\ge}(\cdot)$ and $R_{>}(\cdot)$: $$\sum_{n \ge 0}z^n\frac{\mu_{\Delta}(\tau_H>n)}{\mu_{\Delta}(X)}=\sum_{n\ge 0}z^n\sum_{i \ge n}\mu_X(R> i)+\int \big[R_{>}(z)\circ [I-\mathring{R}(z)]^{-1}\circ \mathring{R}_{\ge }(z)\big](h)d\Leb_X.$$

    Then we just need to estimate the Fourier coefficients of the integral.

\item We will show that $$\int \big[R_{>}(z)\circ [I-\mathring{R}(z)]^{-1}\circ \mathring{R}_{\ge }(z)\big](h)d\Leb_X \approx \int \Big[R_{>}(z)\circ  [I-\mathring{R}(1)]^{-1}\circ \mathring{R}_{\ge }(z)\Big](h)d\Leb_X.$$
\item We will use  (\ref{localnear1}) to show that, $$\int \big[R_{>}(z)\circ [I-\mathring{R}(1)]^{-1}\circ \mathring{R}_{\ge }(z)\big](h)d\Leb_X \approx \int \Big[R_{>}(z)\circ \frac{\mathring{\Proj}(1)}{1-\mathring{\lambda}(1)}\circ \mathring{R}_{\ge }(z)\Big](h)d\Leb_X.$$    
\item We will prove that, $$\int \Big[R_{>}(z)\circ \frac{\mathring{\Proj}(1)}{1-\mathring{\lambda}(1)}\circ \mathring{R}_{\ge }(z)\Big](h)d\Leb_X\approx \int \Big[R_{>}(z)\circ \frac{\Proj(1)}{1-\mathring{\lambda}(1)}\circ R_{\ge }(z)\Big](h)d\Leb_X.$$ Using $\Proj(1)(\cdot)=h\int (\cdot)d\Leb_X $, we will prove the Fourier coefficient
    \[\int_{0}^{2\pi}e^{-int}dt\int \Big[R_{>}(e^{it})\circ \frac{\Proj(1)}{1-\mathring{\lambda}(1)}\circ R_{\ge }(e^{it})\Big](h)d\Leb_X \approx [1-\mathring{\lambda}(1)]^{-1}b_n.\]
\end{enumerate}

In these steps, the meaning of ``$\approx$" will be clear as we proceed with these estimates in Section \ref{proof}. In addition, we will keep track of the range of $\mu_X(H)$. Now we start to prove Theorem \ref{thm}.

\section{Proof of Theorem \ref{thm}}\label{proof}
Lemma \ref{localspectrumpic} already proved that $\mathring{T}=\mathring{R}(1)$ has a leading simple eigenvalue $\mathring{\lambda}(1)$. Now we can give a proof for the asymptotic expansions of $\mu_{\Delta}(\tau_H>n)$.
\subsection{$\mu_{\Delta}(\tau_H>n)$ in terms of operators}

Now we can turn $\mu_{\Delta}(\tau_H>n)$ into an operator renewal equation in terms of the operators that we studied in Section \ref{prelimsection}.
\begin{lemma}[Operator renewal equations]\label{renewaleq2}\ \par For any hole satisfying $\mu_X(H)\in (0, \epsilon_1/|h^{-1}|_{\infty}]$ (see $\epsilon_1$ in Lemma \ref{regulaofR}) and any $|z|< 1$, 
\begin{multline}
\label{ORE}
\sum_{n\ge0} z^n
\frac{\mu_\Delta(\tau_H>n)}{\mu_\Delta(X)}
=
\sum_{n\ge0} z^n
\sum_{i\ge n}\mu_X(R>i)
\\
{}+
\int_X
\Bigl[
R_{>}(z)\circ[I-\mathring R(z)]^{-1}
\circ\mathring R_{\ge}(z)
\Bigr](h)\,d\Leb_X .
\end{multline}

     We note that (\ref{ORE}) does not hold if $z=1$ and $k=1$ because the power series diverge.
\end{lemma}
\begin{proof}
    We classify the orbits in $\Delta$ never hitting $H$ in the time window $[1,n]$ into two classes, one never returning to $X$, another one returning to $X$ at least once. For the first class, the orbits must start from the set $\{(x, i)\in \Delta: R(x)> n+i\}$, whose measure is $\sum_{i \ge n }\mu_{\Delta}|_X\{R>i\}=\mu_{\Delta}(X)\sum_{i \ge n}\mu_{X}(R>i)$. 
    
    For the second class, define $\Delta_i:=\{(x,i): (x,i)\in \Delta\}$ and consider an orbit $\{F(x), F^2(x), \cdots F^n(x)\}$ never reaching $H$ where $x \in \Delta_a, F^n(x)\in \Delta_b$. Since this orbit returns to $X$ at least once, so $b< n$. Recall that $H\subsetneq X$ and $F_{*}\mu_{\Delta}=\mu_{\Delta}$, the measure of the set having such $x$ is\begin{align*}
        \mu_{\Delta}(x\in \Delta_a: &R\circ F^{n-b}(x)>b, \tau_H(x)>n-b)\\
        &=\mu_{\Delta}|_X( R\circ F^{n-b+a}>b, \tau_H>n-b+a, R>a)\\
        &=\mu_{\Delta}(X)\mu_X( R\circ F^{n-b+a}>b, \tau_H>n-b+a, R>a)\\
        &=\mu_{\Delta}(X)\int \mathbbm{1}_{R>b}\circ F^{n-b+a}\mathbbm{1}_{\tau_H>n-b+a} \mathbbm{1}_{R>a}hd\Leb_X\\
         &=\mu_{\Delta}(X)\int \mathbbm{1}_{R>b}\circ F^{n-b+a}\mathbbm{1}_{\tau_H>n-b+a} \mathbbm{1}_{R>a}hd\Leb_{\Delta}\\
         &=\mu_{\Delta}(X)\int \mathbbm{1}_{R>b}\mathring{\mathbb{T}}^{n-b+a}(\mathbbm{1}_{R>a}h)d\Leb_{\Delta}\\
        &=\mu_{\Delta}(X)\int \mathbbm{1}_{R>b}\mathring{\mathbb{T}}^{n-b+a}(\mathbbm{1}_{R>a}h)d\Leb_X
    \end{align*} where the last three ``$=$" are due to Lemma \ref{basicopentransfer} and $\Leb_{\Delta}|_X=\Leb_X$.
    
    Hence the measure of the starting points of the orbits in the second class is \begin{align}
    \mu_{\Delta}(\tau_H>n)&-\mu_{\Delta}(X)\sum_{i \ge n}\mu_{X}(R>i)\nonumber \\
    &= \sum_{a \ge 0,b<n}\mu_{\Delta}(x\in \Delta_a: R\circ F^{n-b}(x)>b, \tau_H>n-b)\nonumber\\
    &=\mu_{\Delta}(X)\sum_{a \ge 0,b<n}\int \mathbbm{1}_{R>b}\mathring{\mathbb{T}}^{n-b+a}(\mathbbm{1}_{R>a}h)d\Leb_X\label{2ndorbit}.
    \end{align}
    
    Denote cylinder sets by $[k_{m-1},k_{m-2}, \cdots, k_0]:=\bigcap_{i\le m-1}(F^R)^{-i}\{R=k_i\}$, and note that \begin{align*}
    \mathring{\mathbb{T}}^{k_{m-1}+k_{m-2}+\cdots +k_0}(\mathbbm{1}_{[k_{m-1},k_{m-2}, \cdots, k_0]}\cdot)&=\mathring{T}^m(\mathbbm{1}_{[k_{m-1},k_{m-2}, \cdots, k_0]}\cdot)=\mathring{R}_{k_{m-1}}\circ\mathring{R}_{k_{m-2}}\circ\cdots \circ \mathring{R}_{k_{0}}(\cdot).\end{align*} 
    
    Now we can continue our estimate for (\ref{2ndorbit}). \[=\mu_{\Delta}(X)\sum_{a \ge 0,b<n}\sum_{1\le m\le n-b}\sum_{k_{m-1}+\cdots +k_0=n-b+a, k_0>a}\int \mathbbm{1}_{R>b}\mathring{\mathbb{T}}^{n-b+a}(\mathbbm{1}_{[k_{m-1},k_{m-2}, \cdots, k_0]}h)d\Leb_X \]
    \begin{align}
        & =\mu_{\Delta}(X)\sum_{a \ge 0,b<n}\sum_{1\le m\le n-b}\sum_{k_{m-1}+\cdots +k_0=n-b+a, k_0>a}\int \mathbbm{1}_{R>b}\mathring{T}^{m}(\mathbbm{1}_{[k_{m-1},k_{m-2}, \cdots, k_1, k_0]}h)d\Leb_X\nonumber\\
        &  =\mu_{\Delta}(X)\sum_{a \ge 0,b<n}\sum_{1\le m\le n-b}\sum_{k_{m-1}+\cdots +k_{1}+k_0'=n-b}\int \mathbbm{1}_{R>b}\mathring{T}^{m}(\mathbbm{1}_{[k_{m-1},k_{m-2}, \cdots, k_1, k_0'+a]}h)d\Leb_X\nonumber\\
        &  =\mu_{\Delta}(X)\sum_{b<n}\sum_{1\le m\le n-b}\sum_{k_{m-1}+\cdots +k_{1}+k_0'=n-b}\int \mathbbm{1}_{R>b}\mathring{R}_{k_{m-1}}\circ\cdots \circ \mathring{R}_{k_1} \circ \mathring{R}_{\ge k_{0}'}(h)d\Leb_X\nonumber\\
        & =\mu_{\Delta}(X)\sum_{b<n}\sum_{1\le m\le n-b}\sum_{k_{m-1}+\cdots +k_{1}+k_0'=n-b}\int R_{>b}\circ\mathring{R}_{k_{m-1}}\circ\cdots \circ \mathring{R}_{k_1} \circ \mathring{R}_{\ge k_{0}'}(h)d\Leb_X \label{4}
    \end{align}where the last ``=" is due to the definition of the transfer operator of $F^R$.

    On the other hand,  Lemma \ref{regulaofR} and Lemma \ref{renewaleq1} show that $\mathring{R}(z), \mathring{R}_{\ge}(z), R_{>}(z)$ and $\sum_{m \ge 1}\mathring{R}(z)^{m-1}$ are analytic on $\mathbb{D}$. Then when $|z|<1$, \begin{align*}
    \sum_{m \ge 1}\mathring{R}(z)^{m-1}\circ \mathring{R}_{\ge }(z)&=\sum_{m \ge 1}\sum_{k \ge m}z^k \sum_{k_{m-1}+k_{m-2}+\cdots+ k_0=k}\mathring{R}_{k_{m-1}}\circ\mathring{R}_{k_{m-2}}\circ\cdots \circ\mathring{R}_{k_{1}}\circ\mathring{R}_{\ge k_{0}}\\
        &=\sum_{k \ge 1}z^k\sum_{1\le m \le k}\sum_{k_{m-1}+k_{m-2}+\cdots+ k_0=k}\mathring{R}_{k_{m-1}}\circ\mathring{R}_{k_{m-2}}\circ\cdots \circ\mathring{R}_{k_{1}}\circ\mathring{R}_{\ge k_{0}}
    \end{align*}and 
    \begin{align}
    &R_{>}(z)\circ \sum_{m \ge 1}\mathring{R}(z)^{m-1}\circ \mathring{R}_{\ge }(z)\nonumber \\
    &=\sum_{n \ge 1}z^n\sum_{b<n}R_{>b}\circ \sum_{1\le m \le n-b}\sum_{k_{m-1}+k_{m-2}+\cdots+ k_0=n-b}\mathring{R}_{k_{m-1}}\circ\mathring{R}_{k_{m-2}}\circ\cdots \circ\mathring{R}_{k_{1}}\circ\mathring{R}_{\ge k_{0}}. \label{27}
    \end{align} 
    
    Now we can establish an operator renewal equation for the second class orbits. By Lemma \ref{renewaleq1}, (\ref{27}) and (\ref{4}): 
    \begin{align*}
        \sum_{n \ge 1}z^n\Big[\frac{\mu_{\Delta}(\tau_H>n)}{\mu_{\Delta}(X)}-\sum_{i \ge n}\mu_{X}(R>i)\Big]&=\int \Big[R_{>}(z)\circ \sum_{m \ge 1}\mathring{R}(z)^{m-1}\circ \mathring{R}_{\ge }(z)\Big](h) d\Leb_X\\
        &=\int \Big[R_{>}(z)\circ [I-\mathring{R}(z)]^{-1}\circ \mathring{R}_{\ge }(z)\Big](h) d\Leb_X.
    \end{align*} 
    
    Adding the terms of ``$n=0$" we conclude the proof.
    \end{proof}

    Although (\ref{ORE})  does not hold for $z=1$, the following lemma shows that the Fourier coefficients of the integral in (\ref{ORE}) can be calculated by studying $z\in S^1$ only.
    \begin{lemma}\label{renewalcoeffi} Suppose that $\int \big[R_{>}(z)\circ [I-\mathring{R}(z)]^{-1}\circ \mathring{R}_{\ge }(z)\big](h) d\Leb_X=\sum_{n\ge 1}c_nz^n$ holds for any $|z|<1$, any hole satisfying $\mu_X(H)\in (0, \epsilon_1]$ and for some $c_n$ depending on $H$. Then the Fourier coefficient \[c_n=\frac{1}{2\pi }\int_0^{2\pi} e^{-int}\int \Big[R_{>}(e^{it})\circ [I-\mathring{R}(e^{it})]^{-1}\circ \mathring{R}_{\ge }(e^{it})\Big](h) d\Leb_X dt=O_H(n^{-k-\beta})\] holds for any hole $H$ satisfying $0<\mu_X(H)\le \sigma \le \epsilon_1/|h^{-1}|_{\infty}$ (see $\sigma$ in Lemma \ref{globalregularity}).
    \end{lemma}
    \begin{proof}
        It follows from Lemma \ref{regulaofR} that
$R_>(\cdot),\mathring R_{\ge}(\cdot)\in O(n^{-k-\beta})$.
Lemma \ref{renewalanalyticinverse}, proved below independently of this
lemma, gives $[I-\mathring R(\cdot)]^{-1}\in O(n^{-k-\beta-1})$ for
each fixed hole. Therefore Lemma \ref{frequentlemma} gives
$R_>(\cdot)\circ[I-\mathring R(\cdot)]^{-1}\circ
\mathring R_{\ge}(\cdot)\in O(n^{-k-\beta})$.
        
        Therefore, the relation $\int \big[R_{>}(z)\circ [I-\mathring{R}(z)]^{-1}\circ \mathring{R}_{\ge }(z)\big](h) d\Leb_X=\sum_{n\ge 1}c_nz^n$ holds for any $z\in \overline{\mathbb{D}}$ and we can compute its Fourier coefficient $c_n$ by integrating it over $S^1$. 
    \end{proof}
    
\begin{remark}\label{remarkonore}
Note that $\sum_{i \ge n}\mu_X(R> i)\approx n^{-k-\beta+1}$, (\ref{ORE}), and the proofs of Lemma \ref{renewalcoeffi} and Lemma \ref{renewaleq2} imply that $$\mu_{\Delta}(\tau_H>n)\approx n^{-k-\beta+1}+O_H(n^{-k-\beta}),$$ where the first term refers to the orbits, which never return to the base $X$ (a reference hyperbolic set) in the time window $[1,n]$, and the second term corresponds to the orbits returning to the base $X$ at least once in this time window. A key observation for the second term is that such orbits gain extra expansion by returning to the reference hyperbolic set $X$, which results in a faster escape. Hence, the second term decays faster.
\end{remark}

We conclude this subsection with an important lemma for $[I-\mathring{R}(\cdot)]^{-1}$.
\begin{lemma}
\label{renewalanalyticinverse}
For any hole $H$ satisfying $\mu_X(H)\in (0, \sigma)$, $[I-\mathring{R}(\cdot)]^{-1} \in O(n^{-k-\beta-1})$. 
\end{lemma}
\begin{proof}
    By the polynomial tails and spectrum property in Definition \ref{defyoung}, $-\mathring{R}(\cdot)\in O(n^{-k-\beta-1})$. Besides, by Lemma \ref{globalregularity} and Lemma \ref{frequentlemma}, we have $[I-\mathring{R}(\cdot)]^{-1}:=\sum_{n\ge 0}G_nz^n$ with
\[
 \|G_n\|=
 \begin{cases}
 O_H(n^{-k-\beta}),&\beta>0,\\
 O_H(n^{-k+1/2}),&\beta=0.
 \end{cases}
\]
Both bounds are summable: $k+\beta>1$ when $\beta>0$, and
$k-1+1/2>1$ when $\beta=0$. Hence $\sum_n\|G_n\|<\infty$. By Lemma 4.5 of \cite{gouezel}, $||G_n||=O(n^{-k-\beta-1})$. 
\end{proof}

\subsection{Approximate $[I-\mathring{R}(\cdot)]^{-1}$ with $[I-\mathring{R}(1)]^{-1}$}
\begin{lemma}\label{approfor1} For any hole $H$ satisfying $\mu_X(H)\in (0, \sigma)$,
\begin{align*}
    &\frac{1}{2\pi}\int_0^{2\pi}e^{-int}\int \Big[R_{>}(e^{it})\circ [I-\mathring{R}(e^{it})]^{-1}\circ \mathring{R}_{\ge }(e^{it})\Big](h) d\Leb_Xdt\\
    &=\frac{1}{2\pi}\int_0^{2\pi}e^{-int}\int \Big[R_{>}(e^{it})\circ [I-\mathring{R}(1)]^{-1}\circ \mathring{R}_{\ge }(e^{it})\Big](h) d\Leb_Xdt+O_H(n^{-k-\beta-1}).
    \end{align*}
\end{lemma}
\begin{proof}
Since $[I-\mathring{R}(1)]^{-1}-[I-\mathring{R}(1)]^{-1}=0$, the sum of the Fourier coefficients of $[I-\mathring{R}(\cdot)]^{-1}-[I-\mathring{R}(1)]^{-1}$ is zero. By Lemma \ref{renewalanalyticinverse}, $ [I-\mathring{R}(\cdot)]^{-1}-[I-\mathring{R}(1)]^{-1}\in O(n^{-k-\beta-1})$. By Lemma \ref{regulaofR}, $R_{>}(\cdot), \mathring{R}_{\ge }(\cdot)\in O(n^{-k-\beta})$. Thus by Lemma \ref{renewalcenteredconvolution}, $R_{>}(\cdot)\circ \{[I-\mathring{R}(\cdot)]^{-1}-[I-\mathring{R}(1)]^{-1}\}\circ \mathring{R}_{\ge }(\cdot) \in O(n^{-k-\beta-1})$.  Hence
\begin{align*}
    &\frac{1}{2\pi}\int_0^{2\pi}e^{-int}\int \Big[R_{>}(e^{it})\circ  \{[I-\mathring{R}(e^{it})]^{-1}-[I-\mathring{R}(1)]^{-1}\}\circ \mathring{R}_{\ge }(e^{it})\Big](h) d\Leb_Xdt\\
    &=O_H(n^{-k-\beta-1}),
    \end{align*}which concludes the proof.
\end{proof}

\subsection{Approximations for $[I-\mathring{R}(1)]^{-1}$}
We will use (\ref{localnear1}), \begin{align}\label{localnear2}
  [I-\mathring{R}(1)]^{-1}=\frac{1}{1-\mathring{\lambda}(1)}\mathring{\Proj}(1) +[I-\mathring{Q}(1)]^{-1}[I-\mathring{\Proj}(1)]
\end{align}  Next lemma shows that we can continue to reduce the equation in Lemma \ref{approfor1} by dropping $[I-\mathring{Q}(1)]^{-1}[I-\mathring{\Proj}(1)]$.
\begin{lemma}\label{approfortrucation} For any hole $H$ satisfying $\mu_X(H)\in (0, \sigma)$,
\begin{align*}
    &\frac{1}{2\pi}\int_0^{2\pi}e^{-int}\int \Big[R_{>}(e^{it})\circ[I-\mathring{R}(1)]^{-1}\circ \mathring{R}_{\ge }(e^{it})\Big](h) d\Leb_Xdt\\
    &=\frac{1}{2\pi}\int_0^{2\pi}e^{-int}\int \Big[R_{>}(e^{it})\circ \frac{\mathring{\Proj}(1)}{1-\mathring{\lambda}(1)}\circ \mathring{R}_{\ge }(e^{it})\Big](h) d\Leb_Xdt+O_{\sigma,\sigma_1'}(n^{-k-\beta})
    \end{align*}where  the constant in $O_{\sigma, \sigma_1'}(\cdot)$ depends on $\sigma, \sigma_1'$ and is independent of any hole $H$ satisfying $\mu_X(H)\in (0,\sigma)$.
\end{lemma}
\begin{proof}
By (\ref{essspectrumradius}), we have $\sup_{\mu_X(H)\in (0,\sigma)}||[I-\mathring{Q}(1)]^{-1}||\le \sum_n \sup_{\mu_X(H)\in (0,\sigma)}||\mathring{Q}(1)^n||=O_{\theta',\sigma_1',\sigma_1}(1)$.
Together with (\ref{boundforprojlambdaq}), $\sup_{\mu_X(H)\in (0,\sigma)}\{||[I-\mathring{\Proj}(1)]||+||[I-\mathring{Q}(1)]^{-1}||\}< \infty$. By Lemma \ref{regulaofR} and Lemma \ref{frequentlemma}, the Fourier coefficients 
   \begin{align*}
    \frac{1}{2\pi}\int_0^{2\pi} &\big[R_{>}(e^{it})\circ [I-\mathring{Q}(1)]^{-1} [I-\mathring{\Proj}(1)]\big]   \circ \mathring{R}_{\ge}(e^{it}) e^{-int} dt=O_{\sigma_1',\sigma}(n^{-k-\beta})
\end{align*}where the constant in $O(\cdot)$ depends on $\sigma_1', \sigma$ but is independent of any hole $H$ satisfying $\mu_X(H)\in (0,\sigma)$. Using (\ref{localnear2}) we conclude the proof.
\end{proof}

Lemma \ref{approfortrucation} shows that we only need to calculate the Fourier coefficients of $R_{>}(\cdot)\circ \frac{\mathring{\Proj}(1)}{1-\mathring{\lambda}(1)}\circ \mathring{R}_{\ge }(\cdot)$.

\subsection{Approximations for $\frac{1}{1-\mathring{\lambda}(1)}\mathring{\Proj}(1)$}

 Here we will reduce the equations in Lemma \ref{approfortrucation} by approximating $\frac{1}{1-\mathring{\lambda}(1)}\mathring{\Proj}(1)$ with $\frac{1}{1-\mathring{\lambda}(1)}\Proj(1)$ where $\Proj(1)(\cdot)=h\int (\cdot) d\Leb_X$ is an easy-to-compute formula. In order to do it, we start with a technical lemma.

\begin{lemma}\label{avoidsingularityatfinitetime}
    For any $n \in \mathbb{N}$, there is $m_n \in \mathbb{N}$ such that for any $m \ge m_n$ and $\diam_{\theta}H\le \theta^{n+2}$, $H \bigcap  (F^R)^{-i}\{R \ge m\}=\emptyset$ for any $i=0,1,\cdots, n$.
\end{lemma}
\begin{proof}
    Since $z_0\in X \setminus S$ and $\diam_{\theta}H\le \theta^{n+2}$, therefore, for any $i=0,1,\cdots, n$, $(F^R)^iH$ is completely contained in an $X_j$. So when $m$ is sufficiently large (e.g., larger than some $m_n$), $\{R\ge m\}$ does not intersect $(F^R)^iH$ for all $i=0,1,\cdots, n$, which concludes the proof.
\end{proof}
Now we start to approximate $\frac{1}{1-\mathring{\lambda}(1)}\mathring{\Proj}(1)$ with $\frac{1}{1-\mathring{\lambda}(1)}\Proj(1)$.

\begin{lemma}\label{beginapproximateproj}
For any hole $H$ satisfying $\mu_X(H)\in (0, \sigma)$, \begin{align*}
        &\frac{1}{2\pi}\int_0^{2\pi}e^{-int}\int \Big[R_{>}(e^{it})\circ \frac{\mathring{\Proj}(1)}{1-\mathring{\lambda}(1)}\circ \mathring{R}_{\ge }(e^{it})\Big](h) d\Leb_Xdt\\
        &=\frac{\mathring{\lambda}(1)}{1-\mathring{\lambda}(1)}\sum_{a+b=n, b>0}\int_{\{R>a\}}[\mathring{\Proj}(1)-\Proj(1)](\mathbbm{1}_{\{R\ge b\}}h)d\Leb_X\\
        &\quad +\frac{\mathring{\lambda}(1)}{1-\mathring{\lambda}(1)}\sum_{a+b=n, b>0}\mu_X(R>a)\mu_X(R\ge b).
    \end{align*}
\end{lemma}
\begin{proof}
By the definition of transfer operators of $F^R$ and Lemma \ref{localspectrumpic}
\begin{align*}
    &\frac{1}{2\pi}\int_0^{2\pi}e^{-int}\int \Big[R_{>}(e^{it})\circ \frac{\mathring{\Proj}(1)}{1-\mathring{\lambda}(1)}\circ \mathring{R}_{\ge }(e^{it})\Big](h) d\Leb_Xdt
    \end{align*}
    \begin{align*}
    &=\frac{1}{1-\mathring{\lambda}(1)}\sum_{a+b=n, b>0}\int R_{>a}\circ \mathring{\Proj}(1)\circ \mathring{R}_{\ge b}(h)d\Leb_X\\
    &=\frac{1}{1-\mathring{\lambda}(1)}\sum_{a+b=n, b>0}\int_{\{R>a\}}\mathring{\Proj}(1)\circ \mathring{R}(1)(\mathbbm{1}_{\{R\ge b\}}h)d\Leb_X\\
    &=\frac{\mathring{\lambda}(1)}{1-\mathring{\lambda}(1)}\sum_{a+b=n, b>0}\int_{\{R>a\}}\mathring{\Proj}(1)(\mathbbm{1}_{\{R\ge b\}}h)d\Leb_X
    \end{align*}

By using $\Proj(1)(\cdot)=h\int (\cdot)d\Leb_X$, we can continue to estimate
\begin{align*}
     &=\frac{\mathring{\lambda}(1)}{1-\mathring{\lambda}(1)}\sum_{a+b=n, b>0}\int_{\{R>a\}}[\mathring{\Proj}(1)-\Proj(1)](\mathbbm{1}_{\{R\ge b\}}h)d\Leb_X\\
     &\quad +\frac{\mathring{\lambda}(1)}{1-\mathring{\lambda}(1)}\sum_{a+b=n, b>0}\int_{\{R>a\}}\Proj(1)(\mathbbm{1}_{\{R\ge b\}}h)d\Leb_X\\
      &=\frac{\mathring{\lambda}(1)}{1-\mathring{\lambda}(1)}\sum_{a+b=n, b>0}\int_{\{R>a\}}[\mathring{\Proj}(1)-\Proj(1)](\mathbbm{1}_{\{R\ge b\}}h)d\Leb_X\\
     &\quad +\frac{\mathring{\lambda}(1)}{1-\mathring{\lambda}(1)}\sum_{a+b=n, b>0}\mu_X(R>a)\mu_X(R\ge b),
\end{align*}hence the proof is concluded.
\end{proof}

So we just need to estimate the error term: \begin{gather}\label{19}
    \sum_{a+b=n, b>0}\int_{\{R>a\}}[\mathring{\Proj}(1)-\Proj(1)](\mathbbm{1}_{\{R\ge b\}}h)d\Leb_X.
\end{gather}

Note that \cite{liveranikeller1} only allowed an $L^1$-estimate of $\mathring{\Proj}(1)-\Proj(1)$, i.e., $\int |\mathring{\Proj}(1)h-\Proj(1)h|d\Leb_X=O(\mu_X(H)^{\eta})$ for some $\eta>0$, which is insufficient to reach our goal. The following lemma improves and sharpens the Proposition \ref{extendliverani1} and \cite{liveranikeller1}.

\begin{lemma}\label{extendliverani2}
    For any $N\in \mathbb{N}$ (to be determined) and any hole satisfying $\mu_X(H)\in (0, \sigma)$ and $ \diam_{\theta}H \le \theta^{N+2}$, there is $a_N>0$ such that for any $a>a_N$, any $b\in \mathbb{N}$,  \begin{align*}
    &\Big|\int_{\{R>a\}}[\mathring{\Proj}(1)-\Proj(1)](\mathbbm{1}_{\{R\ge b\}}h)d\Leb_X\Big|\\
    &\precsim \frac{\mu_X(R>a)}{(1-\delta)^{N+1}}\Big[\bar{\theta}^{N}\mu_X(R\ge b)+ \mu_X(H)^{\eta_s}\mu_X(R\ge b)\Big]\\
    &\quad +\mu_X(R>a)\frac{\bar{\theta}^N}{(1-\delta)^{2N+1}} \mu_X(R\ge b)+\frac{\mu_X(R>a)}{(1-\delta)^{2N+3}}\mu_X(H)^{\eta_s}\mu_X(R\ge b).
    \end{align*}where $\delta, \bar{\theta}$ are the ones in Definition \ref{defyoung} and Lemma \ref{localspectrumpic} where we require $\bar{\theta}<(1-\delta)^2$ and the constant in ``$\precsim$" does not depend on $a,b,N$, or on any hole $H$ satisfying $\mu_X(H)\in (0,\sigma)$.

    Without restrictions on $a$, we have an alternative estimate \begin{align*}
    &\Big|\int_{\{R>a\}}[\mathring{\Proj}(1)-\Proj(1)](\mathbbm{1}_{\{R\ge b\}}h)d\Leb_X\Big|\\
    &\precsim\frac{1}{(1-\delta)^{N+3}}\mu_X(R>a)^{\frac{k+\beta+1}{2(k+\beta)}}\mu_X(H)^{\frac{\eta_s(k+\beta-1)}{2(k+\beta)}}\mu_X(R\ge b)\\
    &\quad +\frac{\mu_X(R>a)}{(1-\delta)^{N+1}}\Big[\bar{\theta}^{N}\mu_X(R\ge b)+ \mu_X(H)^{\eta_s}\mu_X(R\ge b)\Big]\\
    &\quad +\mu_X(R>a)\frac{\bar{\theta}^N}{(1-\delta)^{2N+1}} \mu_X(R\ge b)+\frac{\mu_X(R>a)}{(1-\delta)^{2N+3}}\mu_X(H)^{\eta_s}\mu_X(R\ge b)
    \end{align*}where the constant in ``$\precsim$" does not depend on $a,b,N$, or on any hole $H$ satisfying $\mu_X(H)\in (0,\sigma)$.
\end{lemma}
\begin{proof}
    By definitions of $\Proj, \mathring{\Proj}$ in Lemma \ref{localspectrumpic}, \begin{align}
        &-\int_{\{R>a\}}[\mathring{\Proj}(1)-\Proj(1)](\mathbbm{1}_{\{R\ge b\}}h)d\Leb_X \nonumber\\
        &= -\frac{1}{2\pi i}\int_{\partial B_\delta(1)}\int_{\{R>a\}}\Big[[z-\mathring{R}(1)]^{-1}-[z-R(1)]^{-1}\Big](\mathbbm{1}_{\{R\ge b\}}h)d\Leb_Xdz \nonumber\\
        &= \frac{1}{2\pi i}\int_{\partial B_\delta(1)}\int_{\{R>a\}}[z-T]^{-1}[T-\mathring{T}][z-\mathring{T}]^{-1}(\mathbbm{1}_{\{R\ge b\}}h)d\Leb_Xdz \nonumber\\
        &=\frac{1}{2\pi i}\int_{\partial B_\delta(1)}\int_{\{R>a\}}\sum_{0\le i\le N-1}\frac{T^i}{z^{i+1}}[T-\mathring{T}]\Big\{[z-\mathring{T}]^{-1}-\frac{I}{z}\Big\}(\mathbbm{1}_{\{R\ge b\}}h)d\Leb_Xdz\label{11}\\
        &\quad +\frac{1}{2\pi i}\int_{\partial B_\delta(1)}\int_{\{R>a\}}\sum_{0\le i\le N-1}\frac{T^i}{z^{i+1}}[T-\mathring{T}]\frac{I}{z}(\mathbbm{1}_{\{R\ge b\}}h)d\Leb_Xdz\label{12}\\
        &\quad +\frac{1}{2\pi i}\int_{\partial B_\delta(1)}\int_{\{R>a\}}\frac{T^N}{z^N}\circ [z-T]^{-1}[T-\mathring{T}]\Big\{[z-\mathring{T}]^{-1}-\frac{I}{z}\Big\}(\mathbbm{1}_{\{R\ge b\}}h)d\Leb_Xdz\label{13}\\
        &\quad +\frac{1}{2\pi i}\int_{\partial B_\delta(1)}\int_{\{R>a\}}\frac{T^N}{z^N}\circ [z-T]^{-1}[T-\mathring{T}]\frac{I}{z}(\mathbbm{1}_{\{R\ge b\}}h)d\Leb_Xdz\label{14}
    \end{align} 
    where in the last ``$=$"  we use \begin{gather}
        [z-T]^{-1}=\sum_{0\le i\le N-1}\frac{T^i}{z^{i+1}}+\frac{T^N}{z^N}\circ [z-T]^{-1} \label{15}\\
        \sum_{1\le i\le N-1}\frac{\mathring{T}^i}{z^{i+1}}+\frac{\mathring{T}^N}{z^N}\circ [z-\mathring{T}]^{-1}=[z-\mathring{T}]^{-1}-\frac{I}{z}=[z-\mathring{T}]^{-1}\circ \frac{\mathring{T}}{z} \label{16}
    \end{gather} for any $N\in \mathbb{N}$ to be determined. Therefore we have 
    \[-\int_{\{R>a\}}[\mathring{\Proj}(1)-\Proj(1)](\mathbbm{1}_{\{R\ge b\}}h)d\Leb_X=(\ref{11})+(\ref{12})+(\ref{13})+(\ref{14}).\]
    
    First we consider (\ref{12}): use $[T-\mathring{T}](\cdot)=\mathbbm{1}_HT(\cdot)$ and Lemma \ref{avoidsingularityatfinitetime}, there is $a_N>0$ such that for any $a>a_N$
    \begin{align*}
        &\int_{\{R>a\}}\sum_{0\le i\le N-1}\frac{T^i}{z^{i+1}}[T-\mathring{T}]\frac{I}{z}(\mathbbm{1}_{\{R\ge b\}}h)d\Leb_X\\
        &=\sum_{0\le i\le N-1}\frac{1}{z^{i+2}}\int_{\{R>a\}}T^i[T-\mathring{T}](\mathbbm{1}_{\{R\ge b\}}h)d\Leb_X\\
        &=\sum_{0\le i\le N-1}\frac{1}{z^{i+2}}\int \mathbbm{1}_{\{R>a\}}\circ (F^R)^{i}\mathbbm{1}_HT(\mathbbm{1}_{\{R\ge b\}}h)d\Leb_X=0.
    \end{align*}

    So (\ref{12})$=0$ when $a>a_N$.

    Without restrictions on $a$, and by using  $|z|\ge 1-\delta$ we have
     \begin{align*}
       &\Big|\int_{\{R>a\}}\sum_{0\le i\le N-1}\frac{T^i}{z^{i+1}}[T-\mathring{T}]\frac{I}{z}(\mathbbm{1}_{\{R\ge b\}}h)d\Leb_X\Big|\\
        &\le \sum_{0\le i\le N-1}\frac{1}{(1-\delta)^{i+2}}\Big|\int_{\{R>a\}}T^i[T-\mathring{T}](\mathbbm{1}_{\{R\ge b\}}h)d\Leb_X\Big|\\
        &=\sum_{0\le i\le N-1}\frac{1}{(1-\delta)^{i+2}}\Big|\int \mathbbm{1}_{\{R>a\}}\circ (F^R)^{i}\mathbbm{1}_HT(\mathbbm{1}_{\{R\ge b\}}h)d\Leb_X\Big|\\
        &=\sum_{0\le i\le N-1}\frac{1}{(1-\delta)^{i+2}}\Big|\int \mathbbm{1}_{\{R>a\}}\circ (F^R)^{i}\mathbbm{1}_HR_{\ge b}(h)d\Leb_X\Big|
    \end{align*}

    Using the spectrum property in Definition \ref{defyoung},
the uniform boundedness of $T^i$ on $B$, and Lemma \ref{Rnbound},
\begin{align*}
 &\left|\int_{\{R>a\}}T^i[R_{\ge b}-\mathring R_{\ge b}](h)\,d\Leb_X\right|\\
 &\precsim\mu_X(R>a)
       \|T^i[R_{\ge b}-\mathring R_{\ge b}](h)\|
 \precsim\mu_X(R>a)\mu_X(R\ge b).
\end{align*}
On the other hand, $R_{\ge b}-\mathring R_{\ge b}=\mathbbm1_HR_{\ge b}$,
the $L^1$ contraction of $T^i$, and the spectrum property in Definition \ref{defyoung} give
\begin{align*}
 &\left|\int_{\{R>a\}}T^i[R_{\ge b}-\mathring R_{\ge b}](h)\,d\Leb_X\right|\precsim |\mathbbm1_HR_{\ge b}(h)|_1
 \precsim\mu_X(H)^{\eta_s}\mu_X(R\ge b).
\end{align*}
Taking the smaller of these two bounds and interpolating yields
\begin{align*}
 &\left|\int_{\{R>a\}}T^i[R_{\ge b}-\mathring R_{\ge b}](h)\,d\Leb_X\right|\precsim\mu_X(R>a)^{\frac{k+\beta+1}{2(k+\beta)}}
 \mu_X(H)^{\frac{\eta_s(k+\beta-1)}{2(k+\beta)}}\mu_X(R\ge b).
\end{align*}
Summing the factors $(1-\delta)^{-i-2}$ gives the required estimate.

    So, without restrictions on $a$, \[|(\ref{12})|\precsim \frac{1}{(1-\delta)^{N+3}}\mu_X(R>a)^{\frac{k+\beta+1}{2(k+\beta)}}\mu_X(H)^{\frac{\eta_s(k+\beta-1)}{2(k+\beta)}}\mu_X(R\ge b).\]
    
    Next we study (\ref{11}): use (\ref{16}), $[T-\mathring{T}](\cdot)=T(\mathbbm{1}_{H}\circ F^R \cdot)$ and $|z|\ge 1-\delta$, \begin{align}
        \Big|&\int_{\{R>a\}}\sum_{0\le i\le N-1}\frac{T^i}{z^{i+1}}[T-\mathring{T}]\Big\{[z-\mathring{T}]^{-1}-\frac{I}{z}\Big\}(\mathbbm{1}_{\{R\ge b\}}h)d\Leb_X\Big|\nonumber\\
        &=\Big|\sum_{i\le N-1}\frac{1}{z^{i+1}}\int \mathbbm{1}_{R>a}\circ (F^R)^{i+1}\mathbbm{1}_{H}\circ F^R \Big\{[z-\mathring{T}]^{-1}-\frac{I}{z}\Big\}(\mathbbm{1}_{\{R\ge b\}}h)d\Leb_X\Big|\nonumber\\
         &=\Big|\sum_{i\le N-1}\frac{1}{z^{i+2}}\int \mathbbm{1}_{R>a}\circ (F^R)^{i+1}\mathbbm{1}_{H}\circ F^R [z-\mathring{T}]^{-1}\mathring{T}(\mathbbm{1}_{\{R\ge b\}}h)d\Leb_X\Big|\nonumber\\
          &\le \sum_{i\le N-1}\frac{1}{(1-\delta)^{i+2}}\Big|\int \mathbbm{1}_{R>a}\circ (F^R)^{i+1}\mathbbm{1}_{H}\circ F^R [z-\mathring{T}]^{-1} \mathring{R}_{\ge b}(h)d\Leb_X\Big|\label{28}.
    \end{align}

    By Lemma \ref{avoidsingularityatfinitetime}, there is $a_N>0$ such that for any $a>a_N$,\[\sum_{i\le N-1}\frac{1}{(1-\delta)^{i+2}}\Big|\int \mathbbm{1}_{R>a}\circ (F^R)^{i+1}\mathbbm{1}_{H}\circ F^R [z-\mathring{T}]^{-1} \mathring{R}_{\ge b}(h)d\Leb_X\Big|=0.\]

    So (\ref{11})$=0$ when  $a>a_N$.
    
    Without restrictions on $a$, we use the uniform resolvent bound
    \begin{gather}
        \sup_{z\in\partial B_\delta(1),\ \mu_X(H)\in(0,\sigma)}
        \max\{\|[z-\mathring T]^{-1}\|,\|[z-T]^{-1}\|\}<\infty,
        \label{17}
    \end{gather}
which follows from Proposition \ref{extendliverani1}. The integrals in
(\ref{28}) can be written as
\[
 \int_{\{R>a\}}T^i[T-\mathring T]
       [z-\mathring T]^{-1}\mathring R_{\ge b}(h)\,d\Leb_X.
\]
The spectrum property of Definition \ref{defyoung}, (\ref{17}), the uniform boundedness of $T^i,T$ and
$\mathring T$ on $B$, and Lemma \ref{Rnbound} bound their absolute
values by $C\mu_X(R>a)\mu_X(R\ge b)$. Alternatively, $L^1$ contraction
and $[T-\mathring T](\cdot)=\mathbbm1_HT(\cdot)$ give the bound
\[
 |\mathbbm1_HT[z-\mathring T]^{-1}\mathring R_{\ge b}(h)|_1
 \precsim\mu_X(H)^{\eta_s}\mu_X(R\ge b).
\]
Interpolating the two estimates and summing the geometric factors
in (\ref{28}) yields, without restrictions on $a$,
\[
 |(\ref{11})|\precsim\frac{1}{(1-\delta)^{N+3}}
 \mu_X(R>a)^{\frac{k+\beta+1}{2(k+\beta)}}
 \mu_X(H)^{\frac{\eta_s(k+\beta-1)}{2(k+\beta)}}\mu_X(R\ge b).
\]

 Next we estimate (\ref{14}): using  (\ref{17}) and $|z|\ge 1-\delta$ we have
 \begin{align*}
     &\Big|\int_{\{R>a\}}\frac{T^N}{z^N}[z-T]^{-1}[T-\mathring{T}]\frac{I}{z}(\mathbbm{1}_{\{R\ge b\}}h)d\Leb_X\Big|\\
     &\precsim \Leb_X(R>a)\Big|\Big|\frac{T^N}{z^N}[z-T]^{-1}[T-\mathring{T}]\frac{I}{z}(\mathbbm{1}_{\{R\ge b\}}h)\Big|\Big|\\
      &\precsim \frac{\Leb_X(R>a)}{(1-\delta)^{N+1}}\Big|\Big| [z-T]^{-1}T^{N}[T-\mathring{T}](\mathbbm{1}_{\{R\ge b\}}h)\Big|\Big|\\
     &\precsim \frac{\Leb_X(R>a)}{(1-\delta)^{N+1}}\Big|\Big|T^{N} [T-\mathring{T}](\mathbbm{1}_{\{R\ge b\}}h)\Big|\Big|
\end{align*}

Note that $[T-\mathring{T}](\cdot)=\mathbbm{1}_HT(\cdot)$ and use Lasota-Yorke inequalities in Definition \ref{defyoung} we can continue our estimate as
  \begin{align*}
       &\precsim \frac{\Leb_X(R>a)}{(1-\delta)^{N+1}}\Big[\bar{\theta}^{N}|| [T-\mathring{T}](\mathbbm{1}_{\{R\ge b\}}h)||+|[T-\mathring{T}](\mathbbm{1}_{\{R\ge b\}}h)|_1\Big]\\
        &\precsim \frac{\Leb_X(R>a)}{(1-\delta)^{N+1}}\Big[\bar{\theta}^{N}|| [T-\mathring{T}](\mathbbm{1}_{\{R\ge b\}}h)||+| \mathbbm{1}_HT(\mathbbm{1}_{\{R\ge b\}}h)|_1\Big].
  \end{align*}
 By using Lemma \ref{Rnbound}, $|| [T-\mathring{T}](\mathbbm{1}_{\{R\ge  b\}}h)||\le ||\mathring{R}_{\ge b}(h)||+||R_{\ge b}(h)||$ and $\frac{d\mu_X}{d\Leb_X}=h=C^{\pm 1}$, we can continue our estimate \begin{align*}
  &\precsim \frac{\mu_X(R>a)}{(1-\delta)^{N+1}}\Big[\bar{\theta}^{N}||  R_{\ge b}(h)||+\bar{\theta}^{N}||\mathring{R}_{\ge b}(h)||+| \mathbbm{1}_HR_{\ge b}(h)|_1\Big]\\
      &\precsim \frac{\mu_X(R>a)}{(1-\delta)^{N+1}}\Big[\bar{\theta}^{N}\mu_X(R\ge b)+| \mathbbm{1}_H|^{\eta_s}_1||R_{\ge b}(h)||\Big]\\
      &\precsim \frac{\mu_X(R>a)}{(1-\delta)^{N+1}}\Big[\bar{\theta}^{N}\mu_X(R\ge b)+ \mu_X(H)^{\eta_s}\mu_X(R\ge b)\Big].
  \end{align*}
  
  Therefore we have \[|(\ref{14})|\precsim \frac{\mu_X(R>a)}{(1-\delta)^{N+1}}\Big[\bar{\theta}^{N}\mu_X(R\ge b)+ \mu_X(H)^{\eta_s}\mu_X(R\ge b)\Big].\]
 
Now we estimate (\ref{13}):
\begin{align*}
     &\Big|\int_{\{R>a\}}\frac{T^N}{z^N} [z-T]^{-1}[T-\mathring{T}]\Big\{[z-\mathring{T}]^{-1}-\frac{I}{z}\Big\}(\mathbbm{1}_{\{R\ge b\}}h)d\Leb_X\Big|\\
     &=\Big|\int_{\{R>a\}}\frac{T^{N}}{z^{N+1}} [z-T]^{-1}[T-\mathring{T}][z-\mathring{T}]^{-1} \mathring{T}(\mathbbm{1}_{\{R\ge b\}}h)d\Leb_X\Big|\\
     &=\Big|\int_{\{R>a\}}\frac{T^{N}}{z^{N+1}} [z-T]^{-1}[T-\mathring{T}][z-\mathring{T}]^{-1} \mathring{R}_{\ge b}(h)d\Leb_X\Big|\\
     &\precsim \Leb_X(R>a)\Big|\Big|\frac{T^{N}}{z^{N+1}} [z-T]^{-1}[T-\mathring{T}][z-\mathring{T}]^{-1}\mathring{R}_{\ge b}(h)\Big|\Big|.
 \end{align*}

 Using  $|z|\ge 1-\delta$, (\ref{17}), Lemma \ref{Rnbound} and Lasota-Yorke inequalities in Definition \ref{defyoung}, we can continue our estimate
 \begin{align*}
     &\precsim \Leb_X(R>a)\frac{\bar{\theta}^N}{(1-\delta)^{N+1}}||[z-T]^{-1}[T-\mathring{T}][z-\mathring{T}]^{-1}\mathring{R}_{\ge b}(h)||\\
     &\quad +\Leb_X(R>a)\frac{1}{(1-\delta)^{N+1}}|[z-T]^{-1}[T-\mathring{T}][z-\mathring{T}]^{-1} \mathring{R}_{\ge b}(h)|_1\\
     &\precsim \Leb_X(R>a)\frac{\bar{\theta}^N}{(1-\delta)^{N+1}} ||\mathring{R}_{\ge b}(h)||\\
     &\quad +\Leb_X(R>a)\frac{1}{(1-\delta)^{N+1}}|[z-T]^{-1}[T-\mathring{T}][z-\mathring{T}]^{-1} \mathring{R}_{\ge b}(h)|_1\\
     &\precsim \Leb_X(R>a)\frac{\bar{\theta}^N}{(1-\delta)^{N+1}} \Leb_X(R\ge b)\\
     &\quad +\Leb_X(R>a)\frac{1}{(1-\delta)^{N+1}}|[z-T]^{-1}[T-\mathring{T}][z-\mathring{T}]^{-1} \mathring{R}_{\ge b}(h)|_1,
 \end{align*}

 Using (\ref{15}) and (\ref{17}), we can continue to estimate 
 \begin{align*}
     &\precsim \Leb_X(R>a)\frac{\bar{\theta}^N}{(1-\delta)^{N+1}} \Leb_X(R\ge b)\\
     &\quad +\frac{\Leb_X(R>a)}{(1-\delta)^{N+1}}\Big|\Big[\sum_{0\le i\le N-1}\frac{T^i}{z^{i+1}}+\frac{T^N}{z^N} [z-T]^{-1}\Big][T-\mathring{T}][z-\mathring{T}]^{-1}\mathring{R}_{\ge b}(h)\Big|_1 \\
     &\precsim \Leb_X(R>a)\frac{\bar{\theta}^N}{(1-\delta)^{N+1}} \Leb_X(R\ge b)\\
     &\quad +\Leb_X(R>a)\frac{1}{(1-\delta)^{N+1}}\Big|\Big[\sum_{0\le i\le N-1}\frac{T^i}{z^{i+1}}\Big][T-\mathring{T}][z-\mathring{T}]^{-1}\mathring{R}_{\ge b}(h)\Big|_1\\
     &\quad + \Leb_X(R>a)\frac{1}{(1-\delta)^{N+1}}\Big|[z-T]^{-1} \frac{T^N}{z^N}[T-\mathring{T}][z-\mathring{T}]^{-1}\mathring{R}_{\ge b}(h)\Big|_1\\
     &\precsim \Leb_X(R>a)\frac{\bar{\theta}^N}{(1-\delta)^{N+1}} \Leb_X(R\ge b)\\
     &\quad +\Leb_X(R>a)\frac{1}{(1-\delta)^{N+1}}\Big|\Big[\sum_{0\le i\le N-1}\frac{T^i}{z^{i+1}}\Big][T-\mathring{T}][z-\mathring{T}]^{-1} \mathring{R}_{\ge b}(h)\Big|_1\\
     &\quad + \Leb_X(R>a)\frac{1}{(1-\delta)^{N+1}}\Big|\Big| \frac{T^N}{z^N}[T-\mathring{T}][z-\mathring{T}]^{-1} \mathring{R}_{\ge b}(h)\Big|\Big|.
 \end{align*}

 Using $[T-\mathring{T}](\cdot)=T(\mathbbm{1}_{H}\circ F^R \cdot)=\mathbbm{1}_HT(\cdot)$, $|h^{-1}|_{\infty}< \infty$, (\ref{17}) and Lasota-Yorke inequalities again, we can continue to estimate \begin{align*}
     &\precsim \Leb_X(R>a)\frac{\bar{\theta}^N}{(1-\delta)^{N+1}} \Leb_X(R\ge b) \\
     &\quad+\Leb_X(R>a)\frac{1}{(1-\delta)^{2N+3}}|\mathbbm{1}_{H}\circ F^R[z-\mathring{T}]^{-1} \mathring{R}_{\ge b}(h)|_1\\
     &\quad + \Leb_X(R>a)\frac{1}{(1-\delta)^{2N+1}}\Big|\Big| T^N[T-\mathring{T}][z-\mathring{T}]^{-1} \mathring{R}_{\ge b}(h)\Big|\Big|\\
     &\precsim \Leb_X(R>a)\frac{\bar{\theta}^N}{(1-\delta)^{N+1}} \Leb_X(R\ge b) \\
     & \quad +\Leb_X(R>a)\frac{1}{(1-\delta)^{2N+3}}\mu_X(H)^{\eta_s}||[z-\mathring{T}]^{-1} \mathring{R}_{\ge b}(h)||\\
     &\quad + \Leb_X(R>a)\frac{1}{(1-\delta)^{2N+1}} \Big[\bar{\theta}^N||[T-\mathring{T}][z-\mathring{T}]^{-1}\mathring{R}_{\ge b}(h)||\\
     &\quad +|\mathbbm{1}_{H}T[z-\mathring{T}]^{-1}\mathring{R}_{\ge b}(h)|_1\Big]\\
     &\precsim \Leb_X(R>a)\frac{\bar{\theta}^N}{(1-\delta)^{N+1}} \Leb_X(R\ge b) +\frac{\Leb_X(R>a)}{(1-\delta)^{2N+3}}\mu_X(H)^{\eta_s}||\mathring{R}_{\ge b}(h)||\\
     &\quad + \Leb_X(R>a)\frac{1}{(1-\delta)^{2N+1}} \Big[\bar{\theta}^N||\mathring{R}_{\ge b}(h)||+\mu_X(H)^{\eta_s} ||\mathring{R}_{\ge b}(h)||\Big]\\
     &\precsim \Leb_X(R>a)\frac{\bar{\theta}^N}{(1-\delta)^{2N+1}} \Leb_X(R\ge b)\\
     &\quad +\Leb_X(R>a)\frac{1}{(1-\delta)^{2N+3}}\mu_X(H)^{\eta_s}\Leb_X(R\ge b),
 \end{align*}where the last ``$\precsim$" is due to Lemma \ref{Rnbound}. Therefore, using $\frac{d\mu_X}{d\Leb_X}=C^{\pm 1}$ we have\[|(\ref{13})|\precsim \mu_X(R>a)\frac{\bar{\theta}^N}{(1-\delta)^{2N+1}} \mu_X(R\ge b)+\frac{\mu_X(R>a)}{(1-\delta)^{2N+3}}\mu_X(H)^{\eta_s}\mu_X(R\ge b).\]
 
 We conclude a proof of $-\int_{\{R>a\}}[\mathring{\Proj}(1)-\Proj(1)](\mathbbm{1}_{\{R\ge b\}}h)d\Leb_X=(\ref{11})+(\ref{12})+(\ref{13})+(\ref{14})$ with all the estimates which we have obtained for (\ref{11}), (\ref{12}), (\ref{13}) and (\ref{14}). 
\end{proof}

With the help of Lemma \ref{extendliverani2} we can now give a complete estimate for the error term (\ref{19}).

\begin{lemma}\label{lastapproximateproj}
    For any fixed hole $H$ satisfying $\mu_X(H)\in (0, \sigma)$ and $t \gg 1$, \begin{align*}
    &\Big|\sum_{a+b=t, b>0}\int_{\{R>a\}}[\mathring{\Proj}(1)-\Proj(1)](\mathbbm{1}_{\{R\ge b\}}h)d\Leb_X\Big|\\
    &\precsim \sum_{a+b=t, b>0}\mu_X(R>a)\mu_X(R\ge b)\Big[(\diam_{\theta}H)^{\frac{\log \bar{\theta}-2\log (1-\delta)}{\log \theta}}+(\diam_{\theta}H)^{\frac{\eta_s(k+\beta-1)}{2(k+\beta)}-\frac{2\log(1-\delta)}{\log \theta}}\Big],
    \end{align*}where the constant in ``$\precsim$" does not depend on $t$ and any hole $H$ satisfying $\mu_X(H)\in (0,\sigma)$.
\end{lemma}
\begin{proof}
    By the conditions of Lemma \ref{extendliverani2}, we consider the following \begin{align*}
    &\sum_{a+b=t, b>0}\int_{\{R>a\}}[\mathring{\Proj}(1)-\Proj(1)](\mathbbm{1}_{\{R\ge b\}}h)d\Leb_X\\
    &=\sum_{a< t/2}\int_{\{R>a\}}[\mathring{\Proj}(1)-\Proj(1)](\mathbbm{1}_{\{R\ge t-a\}}h)d\Leb_X\\
    & \quad+\sum_{t>a\ge t/2}\int_{\{R>a\}}[\mathring{\Proj}(1)-\Proj(1)](\mathbbm{1}_{\{R\ge t-a\}}h)d\Leb_X.
    \end{align*}

   By Lemma \ref{extendliverani2}, when $t>a\ge t/2 \ge a_N$, \begin{align*}
       &\Big|\sum_{t>a\ge t/2}\int_{\{R>a\}}[\mathring{\Proj}(1)-\Proj(1)](\mathbbm{1}_{\{R\ge t-a\}}h)d\Leb_X\Big|\\
       &\precsim \sum_{t>a\ge t/2}\frac{\mu_X(R>a)}{(1-\delta)^{N+1}}\Big[\bar{\theta}^{N}\mu_X(R\ge t-a)+ \mu_X(H)^{\eta_s}\mu_X(R\ge t-a)\Big]\\
    &\quad +\mu_X(R>a)\frac{\bar{\theta}^N}{(1-\delta)^{2N+1}} \mu_X(R\ge t-a)+\frac{\mu_X(R>a)}{(1-\delta)^{2N+3}}\mu_X(H)^{\eta_s}\mu_X(R\ge t-a)\\
    &\precsim \sum_{t>a\ge t/2}\mu_X(R>a)\mu_X(R\ge t-a)\Big[\frac{\bar{\theta}^N}{(1-\delta)^{2N+1}}+\frac{\mu_X(H)^{\eta_s}}{(1-\delta)^{2N+3}} \Big]\\
    &\precsim t^{-k-\beta}\Big[\frac{\bar{\theta}^N}{(1-\delta)^{2N+1}}+\frac{\mu_X(H)^{\eta_s}}{(1-\delta)^{2N+3}} \Big].
   \end{align*}

   If $a<t/2$, then
   \begin{align*}
       &\Big|\sum_{a< t/2}\int_{\{R>a\}}[\mathring{\Proj}(1)-\Proj(1)](\mathbbm{1}_{\{R\ge t-a\}}h)d\Leb_X\Big|\\
       &\precsim \sum_{a< t/2}\frac{1}{(1-\delta)^{N+3}}\mu_X(R>a)^{\frac{k+\beta+1}{2(k+\beta)}}\mu_X(H)^{\frac{\eta_s(k+\beta-1)}{2(k+\beta)}}\mu_X(R\ge t-a)\\
    &\quad +\frac{\mu_X(R>a)}{(1-\delta)^{N+1}}\Big[\bar{\theta}^{N}\mu_X(R\ge t-a)+ \mu_X(H)^{\eta_s}\mu_X(R\ge t-a)\Big]\\
    &\quad +\mu_X(R>a)\frac{\bar{\theta}^N}{(1-\delta)^{2N+1}} \mu_X(R\ge t-a)+\frac{\mu_X(R>a)}{(1-\delta)^{2N+3}}\mu_X(H)^{\eta_s}\mu_X(R\ge t-a) \\
    &\precsim \sum_{a< t/2}\frac{1}{(1-\delta)^{N+3}}\mu_X(R>a)^{\frac{k+\beta+1}{2(k+\beta)}}\mu_X(H)^{\frac{\eta_s(k+\beta-1)}{2(k+\beta)}}\mu_X(R\ge t-a)\\
    &\quad +\mu_X(R>a)\mu_X(R\ge t-a)\Big[\frac{\bar{\theta}^N}{(1-\delta)^{2N+1}}+\frac{\mu_X(H)^{\eta_s}}{(1-\delta)^{2N+3}} \Big]\\
    &\precsim t^{-k-\beta}\Big[\frac{\bar{\theta}^N}{(1-\delta)^{2N+1}}+\frac{\mu_X(H)^{\frac{\eta_s(k+\beta-1)}{2(k+\beta)}}}{(1-\delta)^{2N+3}} \Big],
   \end{align*}where the last estimate follows from
\[
 \sum_{a\ge0}\mu_X(R>a)^{\frac{k+\beta+1}{2(k+\beta)}}
 \precsim\sum_{a\ge0}(1+a)^{-(k+\beta+1)/2}<\infty,
\]
since $k+\beta>1$. Now we can determine $N:=\frac{\log \diam_{\theta}H}{\log \theta}-2$ (small $\sigma$ ensures $N>1$). If $t\ge 2a_{\frac{\log \diam_{\theta}H}{\log \theta}-2}$, then \begin{align*}
       &\Big|\sum_{a< t}\int_{\{R>a\}}[\mathring{\Proj}(1)-\Proj(1)](\mathbbm{1}_{\{R\ge t-a\}}h)d\Leb_X\Big|\\
       &\precsim t^{-k-\beta} \Big[\frac{\bar{\theta}^N}{(1-\delta)^{2N+1}}+\frac{\mu_X(H)^{\frac{\eta_s(k+\beta-1)}{2(k+\beta)}}}{(1-\delta)^{2N+3}} \Big]\\
       &\precsim t^{-k-\beta}\Big[(\diam_{\theta}H)^{\frac{\log \bar{\theta}-2\log (1-\delta)}{\log \theta}}+\mu_X(H)^{\frac{\eta_s(k+\beta-1)}{2(k+\beta)}}(\diam_{\theta}H)^{\frac{-2\log(1-\delta)}{\log \theta}}\Big]\\
       &\precsim t^{-k-\beta}\Big[(\diam_{\theta}H)^{\frac{\log \bar{\theta}-2\log (1-\delta)}{\log \theta}}+(\diam_{\theta}H)^{\frac{\eta_s(k+\beta-1)}{2(k+\beta)}-\frac{2\log(1-\delta)}{\log \theta}}\Big],
   \end{align*}where in the last ``$\precsim$" the exponents of $\diam_{\theta}H$ are positive due to the choice of $\delta$ in Lemma \ref{localspectrumpic}, and we use $\mu_X(H)\precsim \Leb_X(H) \precsim \diam_{\theta}H$ due to the expansion of $F^R$ in Definition \ref{defyoung}. We conclude a proof by noting that $t^{-k-\beta}\approx \sum_{a+b=t, b>0}\mu_X(R>a)\mu_X(R\ge b)$.
\end{proof}

Now we can finish a proof of this subsection.
\begin{lemma}\label{lastlemmaofk=1}
   For any fixed hole $H$ satisfying $\mu_X(H)\in (0, \sigma)$ and $n \gg 1$, \begin{align*}
        &\frac{1}{2\pi}\int_0^{2\pi}e^{-int}\int \Big[R_{>}(e^{it})\circ \frac{\mathring{\Proj}(1)}{1-\mathring{\lambda}(1)}\circ \mathring{R}_{\ge }(e^{it})\Big](h) d\Leb_Xdt\\
&=\frac{\mathring{\lambda}(1)}{1-\mathring{\lambda}(1)} [1+O(\diam^{\epsilon}_{\theta}H)] \sum_{a+b=n, b>0}\mu_X(R>a)\mu_X(R\ge b),
    \end{align*} where a constant in $O(\cdot)$ does not depend on any hole $H$ satisfying $\mu_X(H)\in (0,\sigma)$ and $\epsilon:=\min \{\frac{\log \bar{\theta}-2\log (1-\delta)}{\log \theta}, \frac{\eta_s(k+\beta-1)}{2(k+\beta)}-\frac{2\log(1-\delta)}{\log \theta}\}$.
\end{lemma}
 \begin{proof}
      Combining Lemma \ref{beginapproximateproj} with Lemma \ref{lastapproximateproj}, we have
   \begin{align*}
        &\frac{1}{2\pi}\int_0^{2\pi}e^{-int}\int \Big[R_{>}(e^{it})\circ \frac{\mathring{\Proj}(1)}{1-\mathring{\lambda}(1)}\circ \mathring{R}_{\ge }(e^{it})\Big](h) d\Leb_Xdt\\
        &=\frac{\mathring{\lambda}(1)}{1-\mathring{\lambda}(1)}\sum_{a+b=n,b>0}\mu_X(R>a)\mu_X(R\ge b)\\
        &\quad \times \Big\{1+O\Big[(\diam_{\theta}H)^{\frac{\log \bar{\theta}-2\log (1-\delta)}{\log \theta}}+(\diam_{\theta}H)^{\frac{\eta_s(k+\beta-1)}{2(k+\beta)}-\frac{2\log(1-\delta)}{\log \theta}}\Big]\Big\}
    \end{align*}and the proof is concluded.
\end{proof}

\subsection{The final step of the proof}\label{18}
This subsection finalizes a proof of Theorem \ref{thm} by summarizing the previous considerations. 
\begin{proof}[Proof of Theorem \ref{thm}] By  Lemma \ref{approfortrucation}, Lemma \ref{approfor1}, Lemma \ref{lastlemmaofk=1} and using $t^{-k-\beta}\approx \sum_{a+b=t, b>0}\mu_X(R>a)\mu_X(R\ge b)$, we have
\begin{align*}
        &\frac{1}{2\pi}\int_0^{2\pi}e^{-int}\int \Big[R_{>}(e^{it})\circ [I-\mathring{R}(e^{it})]^{-1}\circ \mathring{R}_{\ge }(e^{it})\Big](h) d\Leb_Xdt\\
        &=\frac{\mathring{\lambda}(1)[1+O(\diam^{\epsilon}_{\theta}H)]}{1-\mathring{\lambda}(1)} \sum_{a+b=n,b>0}\mu_X(R>a)\mu_X(R\ge b)+O_H(n^{-k-\beta-1})+ O_{\sigma,\sigma_1'}(n^{-k-\beta})\\
        & =\frac{\mathring{\lambda}(1)+O_{\sigma,\sigma_1'}[\diam^{\epsilon}_{\theta}H+1-\mathring{\lambda}(1)]}{1-\mathring{\lambda}(1)} \sum_{a+b=n,b>0}\mu_X(R>a)\mu_X(R\ge b)+O_H(n^{-k-\beta-1})
    \end{align*}where the constant in $O_{\sigma,\sigma_1'}(\cdot)$ does not depend on any hole $H$ satisfying $\mu_X(H) \in (0,\sigma)$.  By Lemma \ref{renewaleq2}, \begin{align*}
    \frac{\mu_{\Delta}(\tau_H>n)}{\mu_{\Delta}(X)}&=\sum_{i\ge n}\mu_X(R>i)+O_H(n^{-k-\beta-1})\\
    &\quad +\frac{\mathring{\lambda}(1)+e_n(H)}{1-\mathring{\lambda}(1)}\sum_{a+b=n,b>0}\mu_X(R>a)\mu_X(R\ge b)
    \end{align*}where the uniform $O_{\sigma,\sigma_1'}(n^{-k-\beta})$ term is
absorbed into $e_n(H)b_n/(1-\mathring\lambda(1))$, using
$b_n\approx n^{-k-\beta}$. Thus $e_n(H)$ has the bound in
Theorem \ref{thm}. This concludes the proof with $\epsilon=\min\{\frac{\eta_s(k+\beta-1)}{2(k+\beta)}-\frac{2\log(1-\delta)}{\log \theta}, \frac{\log \bar{\theta}-2\log (1-\delta)}{\log \theta}\}$.
    \end{proof}

\section{Applications to explicit models}
\label{appsection}

We will apply our Theorem \ref{defyoung} to various slowly mixing systems via inducing, i.e., constructing a polynomial tower satisfying the conditions in Definition \ref{defyoung}. Most of these conditions follow from the tower construction. The conditions on spectrum and simple eigenvalues need more delicate arguments.

\subsection{One-dimensional intermittent maps}\label{applsv}
\subsubsection{Examples}
In this subsection, we consider two classes of one-dimensional intermittent maps. The first class \cite{Young} is a circle map with degree $d$. Assume that $d > 1$, and there is a point $0$ in $S^1$, such that $Df>1$ in $S^1\setminus \{0\}$; $f$ is $C^2$ in $S^1\setminus \{0\}$;  $f(0)=0$, $Df(0)=1$, and for all $x\approx 0$, $xD^2f(x)\approx |x|^{\alpha}$ for some $\alpha \in (0,1)$. Another class is the Liverani-Saussol-Vaienti (LSV) intermittent map.
\begin{eqnarray}\label{lsvmap}
f(x) =
\begin{cases}
x+2^{\alpha}x^{1+\alpha},      & 0\le x \le \frac{1}{2}\\
2x-1,  & \frac{1}{2} < x \le 1 \\
\end{cases}.
\end{eqnarray}

Since these maps have the same constructions of Young towers \cite{Young}, we only present the discussion for LSV maps in what follows. In \cite{lsv, Young} it was proved for LSV maps that there is a unique SRB measure $\mu_{\alpha}$ such that $\frac{d\mu_{\alpha}}{d\Leb_{[0,1]}}$ is locally Lipschitz on $(0,1]$ and the mixing rate is $O(n^{1-1/\alpha})$. Various limiting behaviors of this system have been obtained, e.g., in sequential setting, see \cite{Sudcds,  sutams}. Let $S:=\bigcup_{n \ge 0}f^{-n}\{0,1\}$, and fix a center $z_0 \in S^c$. We can find a base $X\subsetneq [0,1]$ as follows: choose a sufficiently small $a_m:=f|_L^{-m}\{1\}$,  such that $z_0 \in  (a_m,1]=:X$. Here $f|_L$ denotes the leftmost branch of $f$.

For this intermittent map, $([0,1], f, \mu_{\alpha})$ can be modeled by a first return Young tower $\Delta:=\{(x,n) \in X \times \{0,1,2, \cdots\}: n<R(x)\}$ where $R(x):=\inf\{n\ge 1: f^n(x)\in X\}$, defined on $X$, is the first return time to $X$ with $k:=\lfloor \alpha^{-1} \rfloor, \beta:=\alpha^{-1}-k$. Each $\{R=n\}$ is a finite union of subintervals in $X$. The dynamics $F:\Delta \to \Delta$ sends $(x,n)$ to $(x,n+1)$ if $n+1<R(x)$, and $(x,n)$ to $(f^R(x),0)$ if $n=R(x)-1$. The first return map $f^R:X \to X$ is a uniformly expanding Gibbs-Markov map with big images, $\inf_{x\in X}|Df^R(x)|=:\theta^{-1}>1$. The hole $H \subsetneq X$ studied here contains $z_0$.  Let the transfer operator of $f^R$ be $T$ in Definition \ref{defyoung}.

\subsubsection{Verify conditions of Definition \ref{defyoung}}
The polynomial tails $\Leb_X(R>n)\approx n^{-k-\beta},\Leb_X(R=n)\precsim n^{-k-\beta-1}$ are the classical conclusions of LSV maps. To verify conditions on spectrum and simple eigenvalues in Definition \ref{defyoung}, we apply Theorem 2.1 of \cite{Demersexp} and choose the Banach space  $(B, ||\cdot||)$ to be the space of functions of bounded variation on $X$, i.e., $||\phi||:=\int |\phi| d\Leb_X+ \bigvee_X \phi $ where, for any  subinterval $J\subseteq X$,
\[\bigvee_J \phi:=\inf_{\phi=\psi \text{ a.e.}}\sup_{\{x_i\}}\{\sum_i |\psi(x_i)-\psi(x_{i+1})|: \text{ any } x_1< x_2 <\cdots<x_n \in J\}.\]Obviously, $|\phi|_{\infty} \precsim ||\phi||$, $B$ is compactly-embedded into $L^1(X)$ and  $\mathbbm{1}_{X} \in B$.
\begin{lemma}
There are constants $C>0$ and small $\epsilon_1'>0$ such that for any fixed hole $H$ satisfying $\mu_X(H)\le \epsilon_1'$,  $||\mathring{R}(z)^n\phi ||\le C |z|^n\bar{\theta}^n ||\phi||+C|z|^n|\phi|_1, ||{R}(z)^n\phi ||\le C |z|^n\bar{\theta}^n ||\phi||+C|z|^n|\phi|_1$ hold for any $\phi \in B$.   
\end{lemma}
\begin{proof}
   When $z=1$, Proposition 2.5 of \cite{Demersexp} already proved these inequalities for this Banach space. When $z\neq 1$, on each induced cylinder the twist $z^{\sum_{i\le n-1}R\circ (F^R)^i}$ is constant, with modulus at most \(|z|^n\) after \(n\) induced iterates, it does not affect the argument of Proposition 2.5 of \cite{Demersexp}, which concludes the proof.
\end{proof}
\begin{lemma}
There are constants $C>0$ and small $\epsilon_1'>0$ such that for any fixed hole $H$ satisfying $\mu_X(H)\le \epsilon_1'$,  $||\mathring{R}_{n}||\le C\Leb_X(\{R=n\})$, $||R_{n}||\le C\Leb_X(\{R=n\})$. $|\mathbbm{1}_{\{R>n\}}\phi|_1\le C\Leb_X(R>n)\|\phi\|,   |\mathbbm{1}_E\phi|_1\le C\Leb_X(E)\|\phi\|$ hold for any $\phi \in B$.
\end{lemma}
\begin{proof}
 Since $|\phi|_{\infty} \precsim ||\phi||$,  $|\mathbbm{1}_{\{R>n\}}\phi|_1\le C\Leb_X(R>n)\|\phi\|,  |\mathbbm{1}_E\phi|_1\le C\Leb_X(E)^{\eta_s}\|\phi\|$ hold with $\eta_s=1$. Now we prove the remaining estimates. 
 
 If $\mu_X(H)>0$, then consider $I_i:=[a_i,b_i]\subseteq \{R=n\}$ which is one of the connected components of  $\{R=n\}\bigcap (F^R)^{-1}H^c $. Choose a small $\epsilon_1'>0$ such that for any hole $H$ with $\mu_X(H) \in (0,\epsilon_1')$, $\mu_X(F^RI_i)$ has a positive lower bound independent of $n,I_i$. By the Lipschitz distortion of $F^R$ (see Lemma 5 of \cite{Young}), $\sup_{x\in I_i}\frac{1}{|DF^R(x)|}\precsim \frac{|I_i|}{|F^RI_i|}\precsim |I_i|$ and $\bigvee_{I_i} \frac{1}{|DF^R|}\precsim \sup_{x\in I_i}\frac{1}{|DF^R(x)|} \bigvee_{I_i} F^R \precsim |I_i||F^RI_i|\precsim |I_i|$. Then  \begin{align*}
        \int \big|\mathring{R}_{n}\phi\big|d\Leb_X&=\int \Big|T\big[\mathbbm{1}_{(F^R)^{-1}H^c}\mathbbm{1}_{\{R=n\}}\phi\big]\Big|d\Leb_X\\
        & \le \int T(\mathbbm{1}_{\{R=n\}}|\phi|)d\Leb_X\\
        &=\int_{\{R=n\}}|\phi|d\Leb_X\le \Leb_X(\{R=n\})|\phi|_{\infty}\le \Leb_X(\{R=n\})||\phi||
         \end{align*}
        \begin{align*}
        \bigvee_X \mathring{R}_{n}(\phi) &=\bigvee_X \mathring{T}(\mathbbm{1}_{\{R=n\}}\phi) \le \sum_{I_i \subseteq \{R=n\}}\bigvee_X \mathring{T}(\phi\mathbbm{1}_{I_i}) \le  \sum_{I_i \subseteq \{R=n\}} \bigvee_X \frac{\phi\mathbbm{1}_{I_i}}{|DF^R|}\\
        &\le \sum_{I_i\subseteq \{R=n\}} \bigvee_{I_i} \frac{\phi}{|DF^R|}+\frac{|\phi(a_i)|}{|DF^R(a_i)|}+\frac{|\phi(b_i)|}{|DF^R(b_i)|}\\
        &\le \sum_{I_i\subseteq \{R=n\}} \bigvee_{I_i} \frac{1}{|DF^R|}\max_{I_i}|\phi|+\bigvee_{I_i}\phi\max_{I_i}\frac{1}{|DF^R|}+\frac{|\phi(a_i)|}{|DF^R(a_i)|}+\frac{|\phi(b_i)|}{|DF^R(b_i)|}\\
        &\precsim \sum_{I_i\subseteq \{R=n\}} \Leb_X(I_i)||\phi|| \precsim \Leb_X(\{R=n\}) ||\phi||,
    \end{align*} where the last line is due to the distortion of $F^R$ and the constants in ``$\precsim$" do not depend on $n,I_i$ or any fixed hole $H$ satisfying $\mu_X(H)\le \epsilon_1'$. Hence $||\mathring{R}_{n}|| \precsim \Leb_X(\{R=n\})$. The same arguments can be applied to $R_{n}$. 
\end{proof}

\begin{lemma}\label{spectrumofRS1}
    If $e^{it}\neq 1$, then $R(e^{it})$ does not have an eigenvalue $1$, i.e., $[I-R(e^{it})]^{-1}$ exists. If $e^{it}=1$, then $1$ is an isolated simple eigenvalue, and $R(1)$ does not have other eigenvalues on $S^1$. 
\end{lemma}
\begin{proof}
    Since $F^R$ is mixing and $R(1): B\to B$ has a spectral gap, and a leading isolated simple eigenvalue $1$, but does not have other eigenvalues on $S^1$. We now only need to address the case $e^{it}\neq 1$.
    
    If $B$ were a Banach space $H $ of bounded and locally H\"older functions with a H\"older norm $||\cdot||_H$, the claims in this lemma are classical and have been proven (see  \cite{gouezel, sarig}). However, our $B$ is a Banach space of functions of bounded variation. Therefore, we need additional argument. Suppose that $e^{it}\neq 1$, $R(e^{it})\phi=\phi, \phi \in B$. For any $\epsilon>0$, one can find a smooth function $f_{\epsilon}$ such that $|\phi-f_{\epsilon}|_1\le \epsilon$. Since $\sup_n||\frac{\sum_{i \le n}R(e^{it})^{i}f_{\epsilon}}{n}||_H< \infty$, the Arzela-Ascoli theorem gives a subsequence convergence $\lim_{n_k \to \infty}\frac{\sum_{i \le n_k}R(e^{it})^if_{\epsilon}}{n_k}=\phi_{\epsilon}$ which is a locally H\"older function. Hence, $R(e^{it})\phi_{\epsilon}=\phi_{\epsilon}$ and \[|\phi_{\epsilon}-\phi|_1\le \lim_{n_k\to \infty}\frac{\sum_{i\le n_k}|R(e^{it})^i (f_{\epsilon}-\phi)|_1}{n_k}\le \lim_{n_k\to \infty}\frac{\sum_{i\le n_k}|f_{\epsilon}-\phi|_1}{n_k}\le \epsilon.\] 

    Choose $\epsilon< |\phi|_1$, then the nonzero locally H\"older $\phi_{\epsilon}$ satisfies $R(e^{it})\phi_{\epsilon}=\phi_{\epsilon}$, which contradicts the fact that $R(e^{it})$ does not have the eigenvalue $1$ corresponding to locally H\"older-type eigenvectors. 
\end{proof}

Therefore the conclusion (\ref{31}) in Theorem \ref{thm} holds. To explicitly analyze the term $\mathring{\lambda}(1)$ in (\ref{31}), we need the following lemma making use of \cite{liveranikeller,Demersexp}. 
\medskip
\noindent

\textbf{Convention}: from now on, \(p\)-periodic point means that the least period of this point is \(p\).

\begin{lemma}\label{approofeigen}
    $\mathring{\lambda}(1)=1+O(\mu_X(H))=1-c_{z_0} \mu_X(H)+o(\mu_X(H))$, where 
    \begin{eqnarray*}
c_{z_0}=\begin{cases}
1 & z_0 \text{ is not periodic}\\
1-\prod_{i=0}^{p-1}|DF(F^i(z_0))|^{-1}=1-\prod_{i=0}^{p-1}|Df(f^i(z_0))|^{-1}  & z_0 \text{ is } p\text{-periodic} \\
\end{cases},
\end{eqnarray*}where the constants in $O(\cdot), o(\cdot)$  are independent of any hole $H$ satisfying $\mu_X(H)\in (0,\sigma)$.
\end{lemma}
\begin{proof}
     Since the center $z_0 \in X \bigcap S^c$, so $z_0$ satisfies the conditions on $z_0$ in Definition \ref{defyoung}. For our class of intermittent maps, they satisfy the conditions (P) and $I_{\text{cont}}$ of Theorem 2.1 of \cite{Demersexp}, and conditions (A1)-(A6) in \cite{liveranikeller}. Therefore, by Theorem 2.1 of \cite{liveranikeller},
         \begin{eqnarray*}
c_{z_0}=\begin{cases}
1 & z_0 \text{ is not periodic}\\
1-\prod_{i=0}^{q-1}|DF^R((F^R)^i(z_0))|^{-1}  & z_0 \text{ is } q\text{-periodic point of }F^R \\
\end{cases}.
\end{eqnarray*}

  On the other hand, $\Delta$ is a first return tower and $z_0 \in X$, so the periodicity of $z_0$ in $\Delta$ is the same as the one in $[0,1]$. If  $z_0$ is a $q$-periodic point of $F^R$ or $p$-periodic point of $F$, then $\prod_{i=0}^{q-1}|DF^R((F^R)^i(z_0))|^{-1}=\prod_{i=0}^{q-1}|Df^R((f^R)^i(z_0))|^{-1}=\prod_{i=0}^{p-1}|Df(f^i(z_0))|^{-1}=\prod_{i=0}^{p-1}|DF(F^i(z_0))|^{-1}$. Hence we conclude the proof.
\end{proof}

\subsubsection{Conclusions}

Now we can conclude the following theorem by summarizing the arguments in this subsection. Recall that $k=\lfloor \alpha^{-1} \rfloor$, $\theta:=[\inf_{x\in X}|Df^R(x)|]^{-1}$.
\begin{theorem}\label{thm2}
    For the LSV intermittent system $([0,1],f,\mu_{\alpha})$ considered in this subsection, fix a center $z_0 \in X\bigcap  S^c$ and let $\mu_X:=\frac{\mu_{\alpha}|_X}{\mu_{\alpha}(X)}$. Then there are constants  $\epsilon, \sigma>0$ such that for any fixed hole $H$ satisfying $\mu_X(H)\in (0,\sigma)$, 
\[\frac{\mu_{\alpha}(\tau_H>n)}{\mu_{\alpha}(X)}=\sum_{i \ge n}\mu_X(R> i)+\frac{1+O(\mu_X(H)+\diam_{\theta}^{\epsilon}H)}{c_{z_0} \cdot \mu_X(H)+o(\mu_X(H))}\cdot b_n +O_H(n^{- \alpha^{-1} -1})\]
holds for any $n\in \mathbb{N}$,  where the \(O(\cdot)\) estimate and the small-hole \(o(\cdot)\) estimate are uniform over $n\in \mathbb{N}$ and the fixed hole $H$ satisfying $\mu_X(H)\in (0,\sigma)$. The constant in \(O_H(n^{-\alpha^{-1}-1})\) depends on the fixed hole \(H\). The value of $\epsilon>0$ is the same as the one in Theorem \ref{thm}, and 
$b_n=\sum_{a+b=n, b>0}\mu_X(R> a)\mu_X(R\ge b)\approx n^{-\alpha^{-1}},  \sum_{i \ge n}\mu_X(R\ge i)\approx n^{-\alpha^{-1}+1}$
 do not depend on $H$.
 \begin{eqnarray*}
c_{z_0}=\begin{cases}
1 & z_0 \text{ is not periodic}\\
1-\prod_{i=0}^{p-1}|Df(f^i(z_0))|^{-1}  & z_0 \text{ is } p\text{-periodic} \\
\end{cases}.
\end{eqnarray*}
\end{theorem}
\begin{proof}
      Let $\pi:\Delta \to [0,1]$, $\pi(x,n)=f^n(x)$. Then $\pi_* \mu_{\Delta}=\mu_{\alpha}$. Observe that 
     \begin{align*}
         \mu_{\alpha}(\tau_H>n)&=\mu_{\alpha}(\inf\{n\ge 1: f^n  \in H\}>n)\\
         &=\mu_{\Delta}(\inf\{n\ge 1: F^n \in \pi^{-1}H\}>n)\\
         &=\mu_{\Delta}(\inf\{n\ge 1: F^n \in H\times \{0\}\}>n).
     \end{align*}

We identify $H\times \{0\}$ with $H$, and conclude the proof by applying Theorem \ref{thm} and Lemma \ref{approofeigen}.
\end{proof}

\begin{remark}
    The results on the first hitting statistics for the LSV maps in Theorem \ref{thm2} naturally complement the results of \cite{mldp, Demersexp}.
\end{remark}

\begin{remark}\label{noninvariant}
    The escape rates for LSV maps in \cite{Demers} considered a non-SRB measure as an initial measure. Various escape rates were obtained there depending on the regularities of the density functions of this non-SRB measure. The proof of our Theorem \ref{thm2} assumes invariance of an SRB measure, and it uses a tower $\Delta$ in order to construct an operator renewal equation (\ref{ORE}). This method can be extended to non-SRB initial measures by studying a suitable twisted transfer operator with a different potential. The same observation in Remark \ref{remarkonore} holds even for non-SRB measures. We expect to address various statistical properties (including CLT, large deviations principles and asymptotic expansions of escape rates with explicit coefficients) for non-SRB initial measures in another paper.
\end{remark}

\begin{remark}\label{centeratsigular}
    If $z_0\in S$, then the escape becomes quite different. For example, if $z_0=0$, then $\mu_{\alpha}(\tau_H>n)$ decays exponentially because the intermittency of the system is killed as the hole is opened near the neutral fixed point. If $z_0\in S\setminus \{0\}$, then the escape rates will decay polynomially (\cite{sujsp}). To obtain coefficients of the asymptotic expansion requires a new technique, which will be discussed in our upcoming paper.
\end{remark}

\subsection{Multidimensional non-Markovian non-conformal
intermittent maps}

\subsubsection{Examples}

An example of a higher-dimensional non-Markovian intermittent map
is constructed in \cite{highdimlsv}. Suppose that
$f:[0,1]\times S^1\to[0,1]\times S^1$, $
f(x,\vartheta)
:=
\bigl(f_1(x,\vartheta),f_2(\vartheta)\bigr),
$
where
\[
f_1(x,\vartheta)
=
\begin{cases}
x\bigl(1+x^\alpha u(x,\vartheta)\bigr),
    &0\le x\le \dfrac34,\\[4pt]
4x-3,
    &\dfrac34<x\le1,
\end{cases}
\qquad
f_2(\vartheta)=4\vartheta\pmod1.
\]
Here $\alpha\in(0,1)$, and
$u:[0,3/4]\times S^1\to(0,\infty)$ satisfies the following
conditions:
\begin{enumerate}
\item
Regularity:
$u\in C^2([0,3/4]\times S^1)$.

\item
Positivity at the neutral set:
$u(0,\vartheta)=c_0>0$ for every $\vartheta\in S^1$.

\item
Domain invariance:
$
x\bigl(1+x^\alpha u(x,\vartheta)\bigr)\le1 $ for all 
$(x,\vartheta)\in[0,3/4]\times S^1.
$

\item
Non-contracting derivative:
$
|Df(x,\vartheta)v|\ge |v|$
for all 
$(x,\vartheta)\in[0,3/4]\times S^1, v\in\mathbb R^2.
$

\item
Uniform expansion away from the neutral set:
for every $\delta>0$, $f$ is uniformly expanding on 
$[\delta,1]\times S^1$.

\item
Boundary expansion:
$
f_1(3/4,\vartheta)>15/16
\text{ for every }\vartheta\in S^1.
$

\item
Small variation of the coefficient function:
$
\left\|x\frac{\partial u}{\partial x}\right\|_\infty
\text{ and }
\left\|\frac{\partial u}{\partial\vartheta}\right\|_\infty
$
are sufficiently small to ensure some properties to hold (see below).

Let $
k=\lfloor\alpha^{-1}\rfloor,
\beta=\alpha^{-1}-k,
k+\beta=\alpha^{-1}>1.
$
Define
\[
M_i
:=
\left\{
(x,\vartheta)\in[0,1]\times S^1:
0\le x\le
f_1\left(\frac34,\frac{i+\vartheta}{4}\right)
\right\},
\quad i=0,1,2,3,
\]
and let
$
M:=\bigcup_{i=0}^3M_i, 
b(\vartheta):=f_1(3/4,\vartheta), 
m(\vartheta):=\max_{0\le i\le3}b((i+\vartheta)/4).
$
\item Trapping
condition: 
$4m(\vartheta)-3
\le m(4\vartheta\pmod1)
\text{ for every }\vartheta\in S^1$.
This ensures $f(M)=M$.
A sufficient condition for this is
$
4\max b-3\le\min b.
$
This holds for sufficiently small perturbations of any constant
$b\in(15/16,1)$.

\end{enumerate}

We induce on $
X:=\bigl([3/4,1]\times S^1\bigr)\cap M.
$
Let $R:X\to\mathbb N$ be the first return time to $X$, and write
$
G=f^R:X\to X.
$
For $n\ge1$ and $1\le j\le4^n$, put
\[
X_{n,j}
=
\left\{
(y,\vartheta)\in X:
R(y,\vartheta)=n,\quad
\frac{j-1}{4^n}<\vartheta<\frac{j}{4^n}
\right\}.
\]
These sets form a partition of $X$ up to a null set, denoted
by $\alpha^X$. On each element $a$ of the partition, the map
$G_a=G|_a$ is a diffeomorphism onto its image.
For $n=1$,
\[
X_{1,j}
=
\left\{
(y,\vartheta)\in X:
y>15/16,\quad
\frac{j-1}{4}<\vartheta<\frac j4
\right\}.
\]
For $n\ge2$,
$
G(X_{n,j})
\in
\left\{
\bigl([3/4,1]\times S^1\bigr)\cap M_i
\right\}_{i=0}^3.
$
Thus
$
\mathcal I:=\{Ga:a\in\alpha^X\}
$
is finite. These images might not be unions of the elements of partition in $\alpha^X$. Hence $G$ could be not a Markov map. Such phenomenon is similar to the AFN maps considered in \cite{melbourneinvention}. Let $T$ be the transfer operator of $G$ with respect to $\Leb_X$.

\subsubsection{Functional spaces and closed-system estimates}

Let $B$ be a complex Banach space of functions of
bounded variation on $X$. For any function $v$, define\[
\operatorname{Var}(v)
=
\sup_{\substack{
\Phi\in C_c^1(\operatorname{int}X)\\
|\Phi|_\infty\le1}}
\int_X v\,\operatorname{div}\Phi\,dx.
\]
A useful property is the lower semicontinuity of $\operatorname{Var}$, i.e., if $v_j \in B$ approximates $v \in B$ in $L^1$, then $\operatorname{Var}v \le \liminf_j \operatorname{Var}v_j$. Another useful property is that any $v\in B$ can be approximated by some smooth functions $v_j$ such that $|v_j-v|_1\to 0$ and $\operatorname{Var}v_j \to \operatorname{Var}v$ . Now define a Banach space
\[
B:=
\left\{
v\in L^1(X):
\operatorname{Var}(v)<\infty
\right\},
\qquad
\|v\|=|v|_1+\operatorname{Var}(v).
\]
It satisfies $|v|_2\le C\|v\|$, $B$ is compactly embedded into $L^1(X)$ and $\mathbbm{1}_X\in B$.
For every measurable $E\subseteq X$, the H\"older's inequality gives that for any $v \in B$,
\begin{equation}\label{md:smallset}
|1_Ev|_1
\precsim
\Leb_X(E)^{1/2}|v|_2
\precsim \Leb_X(E)^{1/2}\|v\|.
\end{equation}
Thus in Definition \ref{defyoung} $\eta_{\mathrm s}=1/2$. 
Let $
JG:=|\det DG|,
g:=(JG)^{-1}.
$
The following facts are given in
\cite{highdimlsv}.

\begin{lemma}[Closed-system estimates, see
\cite{highdimlsv}]\label{porpertyofhighlsv}
\begin{enumerate}
\item
In the interior of every partition element, $
|DG(y,\vartheta)v|\ge4|v|,
 v\in\mathbb R^2.
$
Moreover,
\[
\operatorname{Leb}_X(R>n)\approx n^{-k-\beta},
\qquad
\operatorname{Leb}_X(R=n)\approx n^{-1-k-\beta}.
\]
\item
Let $G=(G_1,G_2)$. On a return-$n$ branch, $n\ge2$,
\[
DG_a
=
\begin{pmatrix}
\partial_yG_1&\partial_\vartheta G_1\\
0&4^n
\end{pmatrix},
\qquad
\partial_yG_1\approx n^{1+k+\beta},
\qquad
|\partial_\vartheta G_1|\le C4^n.
\]
\[
\|(DG_a)^{-1}\|\le\frac14,
\qquad
D_0:=
\sup_{a\in\alpha^X}
\left\|
(\nabla g)(DG_a)^{-1}JG_a
\right\|_\infty
<\infty.
\]
The return strips have the form
\[
X_{n,j}
=
\left\{
(y,\vartheta):
\xi_n(\vartheta)<y<\xi_{n-1}(\vartheta),\quad
\frac{j-1}{4^n}<\vartheta<\frac j{4^n}
\right\}
\]
with
$
\xi_{n-1}(\vartheta)-\xi_n(\vartheta)
\approx n^{-1-k-\beta}.
$
\item
The map $f:M\to M$ has a unique absolutely continuous invariant
probability measure $\mu$. 
$
\mu_X:=\frac{\mu|_X}{\mu(X)}
$
is mixing for $G$, and
$
h:=\frac{d\mu_X}{d\operatorname{Leb}_X}\in B,
h= C^{\pm 1},
\int_Xh\,d\operatorname{Leb}_X=1.
$

\item
There exist $C\ge1$ and $\bar\theta\in(0,1)$, such that
\[
\|R(\zeta)^m v\|
\le
C|\zeta|^m\bar\theta^m\|v\|
+
C|\zeta|^m|v|_1
\]
for any $v\in B$, $m\ge1$ and $|\zeta|\le1$. $T=R(1):B\to B$ has a simple isolated eigenvalue
$1$, it has no other spectrum on $S^1$, and its remaining
spectrum is contained in a disk of radius strictly less than $1$.
Furthermore,
$
I-R(\zeta):B\to B
$
is invertible for every $\zeta\in S^1\setminus\{1\}$.

\item
For every $v\in B$ and $n\ge1$,
$\|R_n\|\le Cn^{-1-k-\beta}\le C\Leb_X(R=n), 
$ $|1_{\{R=n\}}v|_1
 \le Cn^{-1-k-\beta}\|v\| \le C\Leb_X(R=n)\|v\|$. This implies that $|1_{\{R>n\}}v|_1\le C\Leb_X(R>n)\|v\|$.
\end{enumerate}
\end{lemma}
This Lemma \ref{porpertyofhighlsv} verifies the conditions in the Definition \ref{defyoung} for closed systems. To verify the conditions in the Definition \ref{defyoung} for open systems, we need some preliminary estimates below.

\subsubsection{Preliminary estimates}

Denote the measure of boundary integrals by $m_1$. For $a\in\alpha^X$ and a smooth function $q$, define
\[
\mathcal B_a(q)
:=
\int_{\partial a}
|q|\frac{J_{\partial}G_a}{JG_a}\,dm_1,
\]
where $J_{\partial}G_a$ is the one dimensional Jacobian for $G_a|_{\partial a}: \partial a \to  \partial G_a(a)$. In what follows, on a two dimensional domain, its integral $\int (\cdot)$ means $\int (\cdot) dx$ where $dx$ is the two dimensional Lebesgue measure. We now include Lemma 4.6 of \cite{highdimlsv} and slightly sharpen Lemma 4.5 of \cite{highdimlsv}, which turns out to be important in our argument.
\begin{lemma} For sufficiently small $
\left\|x\frac{\partial u}{\partial x}\right\|_\infty
\text{ and }
\left\|\frac{\partial u}{\partial\vartheta}\right\|_\infty
$, there exist $ C>0,\kappa_0\in (0,3/4)$, such that for any $n\ge 1$, any smooth function $q$ and any $a \in \alpha^X$, \begin{equation}\label{md:anisotropic}
\mathcal B_a(q)
\le
C\int_a|q|
+
C4^{-n}\int_a|\partial_\vartheta q|
+
Cn^{-1-k-\beta}\int_a|\partial_yq|,
\quad R(a)=n\ge2.
\end{equation} and for any $N_0\ge 1$, there is $C_{N_0}>0$ such that
\begin{equation}\label{md:short-boundary}
\mathcal B_a(q)
\le
C_{N_0}\int_a|q|
+
\kappa_0\int_a|\nabla q|,
\quad R(a)\le N_0.
\end{equation}
\end{lemma}
\begin{proof}
 (\ref{md:short-boundary}) follows from Lemma 4.6 of \cite{highdimlsv}. (\ref{md:anisotropic}) sharpens Lemma 4.5 of \cite{highdimlsv}, follows from  Lemma B.4 and the proof on pages 28-29 in \cite{highdimlsv}. Here they require additional smallness of $
\left\|x\frac{\partial u}{\partial x}\right\|_\infty
\text{ and }
\left\|\frac{\partial u}{\partial\vartheta}\right\|_\infty
$.\end{proof}

Now we choose a hole $H$ and its center. Let
$
S^+
=
\partial X
\cup\bigcup_{I\in\mathcal I}\partial I
\cup\bigl(X\cap\{\vartheta=0\}\bigr)
$
and let 
$
S
=
\bigcup_{j\ge0}G^{-j}
\left(\bigcup_{a\in\alpha^X}\partial a\right)
$
be the singular set in Definition~\ref{defyoung}. Fix 
\begin{align}\label{centerz0}
  z_0\in \operatorname{int}X\setminus(S\cup S^+).
  \end{align}

Choose $r_{\mathrm{loc}}>0$ such that
$
\overline{B_{4r_{\mathrm{loc}}}(z_0)}
\subset\operatorname{int}X\setminus S^+
$. Consider any hole
$
H:=B_r(z)\ni z_0, 0<r\le r_{\mathrm{loc}}/2.
$ Here $z$ is the geometric center of the ball $H$ and
$z_0$ is the fixed point contained in the hole, called a center of $H$ as in
Definition~\ref{defyoung}. Obviously, 
$
B_{2r_{\mathrm{loc}}}(z)
\subsetneq B_{3r_{\mathrm{loc}}}(z_0).
$ Since $G$ has  finite image, for any $a\in \alpha^X$,
$
B_{2r_{\mathrm{loc}}}(z)\subsetneq Ga$ 
or $B_{2r_{\mathrm{loc}}}(z)\cap Ga=\emptyset.
$ Now we can sharpen Corollary 4.7 of \cite{highdimlsv} by estimating the branch boundaries away from the hole.
\begin{lemma}\label{md:localized-boundary}
 Given such hole $H$, choose a sufficiently large $N_0$, there is $C_{N_0}>0$ such that for any smooth $v$,
\begin{equation*}
\sum_{a\in\alpha^X}\mathcal B_a(v)
\le
C_{N_0}|v|_1
+
\kappa_0
\int_{G^{-1}(X\setminus B_{r_{\mathrm{loc}}}(z))}
|\nabla v|.
\end{equation*}    
\end{lemma}
\begin{proof}    
Choose a bump function $\chi_z$ such that,
\[
\chi_z\in C_c^\infty(B_{2r_{\mathrm{loc}}}(z)),
\quad
0\le\chi_z\le1,
\quad
\chi_z=1\ \text{on }B_{r_{\mathrm{loc}}}(z),
\quad
|\nabla\chi_z|\le C/r_{\mathrm{loc}}
\] for some universal constant $C>0$.
For smooth $v$, let
$
p_a:=(1-\chi_z\circ G_a)v.
$
Since $\chi_z=0$ on $\partial(Ga)$, hence
$
\mathcal B_a(v)=\mathcal B_a(p_a).
$ Now we choose a sufficiently large $N_0>0$ such that $C4^{-N_0}+
CN_0^{-1-k-\beta}\le \kappa_0$ where $C$ is the one in (\ref{md:anisotropic}).

If $R(a)\le N_0$, $
\nabla p_a
=
(1-\chi_z\circ G_a)\nabla v
-
v(DG_a)^T(\nabla\chi_z)\circ G_a
$. Since there are finitely many such $a$, there exists $C_{N_0}>0$ depending on $N_0$ such that
\[
\int_a|\nabla p_a|
\le
\int_{a\cap G^{-1}(X\setminus B_{r_{\mathrm{loc}}}(z))}
|\nabla v|
+
C_{N_0}\int_a|v|.
\]
Applying \eqref{md:short-boundary}, we obtain
\[
\mathcal B_a(v)
\le
C_{N_0}\int_a|v|
+
\kappa_0
\int_{a\cap G^{-1}(X\setminus B_{r_{\mathrm{loc}}}(z))}
|\nabla v|.
\]

If $n=R(a)>N_0$, the derivative estimates for $G_a$ in
Lemma~\ref{porpertyofhighlsv} give
\[
4^{-n}
|\partial_\vartheta(\chi_z\circ G_a)|
\le C,
\qquad
n^{-1-k-\beta}
|\partial_y(\chi_z\circ G_a)|
\le C.
\]
Applying \eqref{md:anisotropic} to $p_a$ and using the large $N_0$ yield\begin{align*}
\mathcal B_a(v)
&\le
C\int_a|v|+
C(4^{-n}+n^{-1-k-\beta})
\int_{a\cap G^{-1}(X\setminus B_{r_{\mathrm{loc}}}(z))}
|\nabla v|\\
&\le
C\int_a|v|+
\kappa_0
\int_{a\cap G^{-1}(X\setminus B_{r_{\mathrm{loc}}}(z))}
|\nabla v|.
\end{align*}

Summing over all $a$ proves the desired results.
\end{proof}

We now give a circular boundary estimate for the hole $H$. Recall that $H=B_r(z)$, where $0<r\le r_{\mathrm{loc}}/2$.
\begin{lemma}There are constants $ C_{\mathrm{loc}}, r_{\mathrm{loc}}>0$, for any smooth $w$,
\begin{equation}
\label{md:circle-trace}
\int_{\partial H}|w|\,dm_1
\le
\frac{4}{3r_{\mathrm{loc}}}
\int_{B_{r_{\mathrm{loc}}}(z)\setminus H}|w(x)|\,dx
+
\int_{B_{r_{\mathrm{loc}}}(z)\setminus H}
|\nabla w(x)|\,dx,
\end{equation} and for any $w\in B$,
\begin{equation}
\label{md:hole-multiplier}
\|\mathbbm{1}_{H^c}w\|
\le C_{\mathrm{loc}}\|w\|,
\end{equation}where the constants $ C_{\mathrm{loc}}, r_{\mathrm{loc}}$ are independent of $r$ and of the geometric center
$z$.    
\end{lemma}
\begin{proof}
For a smooth function $w$, write
$
W(t,\vartheta):=
w\bigl(z+t(\cos\vartheta,\sin\vartheta)\bigr).
$
For $r\le s\le r_{\mathrm{loc}}$, we have
\[
|W(r,\vartheta)|
\le
|W(s,\vartheta)|
+
\int_r^s|\partial_tW(t,\vartheta)|\,dt.
\]

Multiplying the probability density
$2s/(r_{\mathrm{loc}}^2-r^2)ds$ and integrating in $s$ give
\[
\begin{aligned}
|W(r,\vartheta)|\le
\frac{2}{r_{\mathrm{loc}}^2-r^2}
\int_r^{r_{\mathrm{loc}}}s|W(s,\vartheta)|ds+
\frac{2}{r_{\mathrm{loc}}^2-r^2}
\int_r^{r_{\mathrm{loc}}}
s\left(\int_r^s|\partial_tW(t,\vartheta)|\,dt\right)ds.
\end{aligned}
\]

Integrating over $\partial H$, we obtain\begin{align*}
&\int_{\partial H}|w|\,d m_1
=
r\int_0^{2\pi}|W(r,\vartheta)|\,d\vartheta\\
&\le
\frac{2r}{r_{\mathrm{loc}}^2-r^2}
\int_0^{2\pi}\int_r^{r_{\mathrm{loc}}}
s|W(s,\vartheta)|\,ds\,d\vartheta\\
&\quad+
\frac{2r}{r_{\mathrm{loc}}^2-r^2}
\int_0^{2\pi}\int_r^{r_{\mathrm{loc}}}
\int_r^s
s|\partial_tW(t,\vartheta)|\,dt\,ds\,d\vartheta\\
&=\frac{2r}{r_{\mathrm{loc}}^2-r^2}\int_{B_{r_{\mathrm{loc}}}(z)\setminus H}|w(x)|\,dx+
\frac{2r}{r_{\mathrm{loc}}^2-r^2}
\int_0^{2\pi}\int_r^{r_{\mathrm{loc}}}
\int_r^s
s|\partial_tW(t,\vartheta)|\,dt\,ds\,d\vartheta\\
&\le \frac{4}{3r_{\mathrm{loc}}}\int_{B_{r_{\mathrm{loc}}}(z)\setminus H}|w(x)|\,dx+
\frac{2r}{r_{\mathrm{loc}}^2-r^2}
\int_0^{2\pi}\int_r^{r_{\mathrm{loc}}}
\int_r^s
s|\partial_tW(t,\vartheta)|\,dt\,ds\,d\vartheta
\end{align*}where the second ``$=$" is due to
$dx=s\,ds\,d\vartheta$ and the last ``$\le$" is due to $r\le r_{\mathrm{loc}}/2$. For the derivative term, by the Fubini theorem, $\int_r^{r_{\mathrm{loc}}}
\int_r^s s|\partial_tW(t,\vartheta)|\,dt\,ds=
\int_r^{r_{\mathrm{loc}}}
\int_t^{r_{\mathrm{loc}}}
s|\partial_tW(t,\vartheta)|\,ds\,dt=
\frac12
\int_r^{r_{\mathrm{loc}}}
(r_{\mathrm{loc}}^2-t^2)
|\partial_tW(t,\vartheta)|\,dt.
$ Hence we can continue to estimate
\begin{align*}
&= \frac{4}{3r_{\mathrm{loc}}}\int_{B_{r_{\mathrm{loc}}}(z)\setminus H}|w(x)|\,dx+\frac{r}{r_{\mathrm{loc}}^2-r^2}
\int_0^{2\pi}\int_r^{r_{\mathrm{loc}}}
(r_{\mathrm{loc}}^2-t^2)
|\partial_tW(t,\vartheta)|\,dt\,d\vartheta\\
&=\frac{4}{3r_{\mathrm{loc}}}\int_{B_{r_{\mathrm{loc}}}(z)\setminus H}|w(x)|\,dx+
\int_0^{2\pi}\int_r^{r_{\mathrm{loc}}}
\frac{r(r_{\mathrm{loc}}^2-t^2)}
     {t(r_{\mathrm{loc}}^2-r^2)}
|\partial_tW(t,\vartheta)|\,t\,dt\,d\vartheta\\
&\le \frac{4}{3r_{\mathrm{loc}}}\int_{B_{r_{\mathrm{loc}}}(z)\setminus H}|w(x)|\,dx+
\int_0^{2\pi}\int_r^{r_{\mathrm{loc}}}
|\partial_tW(t,\vartheta)|\,t\,dt\,d\vartheta\\
& \le \frac{4}{3r_{\mathrm{loc}}}\int_{B_{r_{\mathrm{loc}}}(z)\setminus H}|w(x)|\,dx+
\int_{B_{r_{\mathrm{loc}}}(z)\setminus H}
|\nabla w(x)|\,dx
\end{align*}where in the last ``$\le $" we use $dx=t\,dt\,d\vartheta$ and $|\partial_tW(t,\vartheta)|\le
\left|
\nabla w\bigl(z+t(\cos\vartheta,\sin\vartheta)\bigr)
\right|$. Hence (\ref{md:circle-trace}) holds and \begin{align*}
\operatorname{Var}(\mathbbm{1}_{H^c}w)&= \int_{X\setminus\overline H}|\nabla w(x)|\,dx
+
\int_{\partial H}|w|\,dm_1\\
&\le
\int_{X\setminus\overline H}|\nabla w(x)|\,dx
+
\int_{B_{r_{\mathrm{loc}}}(z)\setminus H}
|\nabla w(x)|\,dx+
\frac{4}{3r_{\mathrm{loc}}}
\int_{B_{r_{\mathrm{loc}}}(z)\setminus H}|w(x)|\,dx\\
&\le
2\operatorname{Var}(w)
+
\frac{4}{3r_{\mathrm{loc}}}\int_X|w(x)|\,dx.
\end{align*}

For $w\in B$, we approximate $w$ with smooth
 $w_j$ in $L^1$ such that
\[
|w_j-w|_1 \to 0,
\qquad
\operatorname{Var}(w_j) \to \operatorname{Var}(w).
\]
Then the lower semicontinuity of $\operatorname{Var}$ yields\begin{align*}
\operatorname{Var}(\mathbbm{1}_{H^c}w)
&\le
\liminf_{j\to\infty}
\operatorname{Var}(\mathbbm{1}_{H^c}w_j)\le
2\operatorname{Var}(w)
+
\frac{4}{3r_{\mathrm{loc}}}\int_X|w(x)|\,dx.
\end{align*}Together with
$
|\mathbbm{1}_{H^c}w|_1\le |w|_1
$, 
this proves (\ref{md:hole-multiplier}).
\end{proof}

Now we estimate the inverse-branches. Let $w_a:
=
\left(\frac{v}{JG_a}\right)\circ G_a^{-1}$ for any $a\in\alpha^X$.
\begin{lemma}
There are constant $D_0, C>0$, for any smooth function $v$, any measurable $E\subseteq  X$,
\begin{equation}
\label{md:branch-mass}
\sum_{a\in\alpha^X}\int_{Ga\cap E}|w_a|
=
\int_{G^{-1}E}|v|
\end{equation} 
\begin{equation}
\label{md:branch-gradient}
\sum_{a\in\alpha^X}\int_{Ga\cap E}|\nabla w_a|
\le
D_0\int_{G^{-1}E}|v|
+
\frac14\int_{G^{-1}E}|\nabla v|.
\end{equation}
\begin{equation}
\label{md:return-branch-gradient}
\begin{aligned}
\sum_{R(a)=n}\int_{Ga\cap E}|\nabla w_a|
\le{}
D_0\int_{\{R=n\}\cap G^{-1}E}|v|+
Cn^{-1-k-\beta}
\int_{\{R=n\}\cap G^{-1}E}|\nabla v|.
\end{aligned}
\end{equation}
\begin{equation}
\label{md:boundary-change}
\int_{\partial(Ga)}|w_a|\,dm_1
=
\mathcal B_a(v).
\end{equation}
\end{lemma}
\begin{proof}
By changing variables we get
(\ref{md:branch-mass}) and (\ref{md:boundary-change}). With $g_a=(JG_a)^{-1}$, $\nabla w_a
=
\bigl[v\nabla g_a+g_a\nabla v\bigr]
(DG_a)^{-1}\circ G_a^{-1}$,
 bounded distortion
$
\frac{|\nabla g_a(DG_a)^{-1}|}{g_a}\le D_0
$ and $||(DG_a)^{-1}||\le 1/4$ in Lemma~\ref{porpertyofhighlsv}, we have 
\begin{equation*}
\int_{Ga\cap E}|\nabla w_a|
\le
D_0\int_{a \cap G^{-1}E}|v|
+
\frac14\int_{a \cap G^{-1}E}|\nabla v|.
\end{equation*} Hence, summing all $a\in \alpha^X$ gives (\ref{md:branch-gradient}). A stronger inverse estimate, on branches with $R(a)=n$, $||(DG_a)^{-1}||\le C n^{-1-k-\beta}$ gives (\ref{md:return-branch-gradient}).
\end{proof}

\subsubsection{Verifying conditions in Definition \ref{defyoung}}
Lemma \ref{porpertyofhighlsv} and (\ref{md:smallset}) already verified the conditions in Definition \ref{defyoung} for closed systems. Recall that the center $z_0$ of the hole $H$ is (\ref{centerz0}). We now verify the conditions in Definition \ref{defyoung} for open systems. Because $
\mathring{R}_n=\mathbbm{1}_{H^c}R_n,
$ \eqref{md:hole-multiplier}  and Lemma \ref{porpertyofhighlsv} give that for any $v \in B$, $
\|\mathring{R}_nv\| \precsim \|R_nv\| \precsim n^{-1-k-\beta}\|v\|\precsim \Leb_X(R=n) \|v\|$. Now we verify uniform open Lasota--Yorke inequalities.
\begin{lemma}
For any $m \ge 1$ and any $v\in B$,
\begin{equation}
\label{md:open-twisted-ly}
\|\mathring{R}(\zeta)^m v\|
\le
|\zeta|^m\rho^m\|v\|
+
C_{\mathrm{loc}}|\zeta|^m|v|_1.
\end{equation}
\begin{equation}
\label{md:open-ly}
\operatorname{Var}(\mathring{T}v)
\le
\rho\operatorname{Var}(v)+C_{\mathrm{loc}}|v|_1
\end{equation}where $
\rho:
=
\max\left\{\frac14+\kappa_0,\frac12\right\}<1
$ and $C_{\mathrm{loc}}$ do not depend on any $H$ with center $z_0$ and radius $r \in (0,r_{\mathrm{loc}}/2]$.
\end{lemma}
\begin{proof}
For smooth $v$,
$
\mathring{T}v
=
\mathbbm{1}_{H^c}Tv
=
\sum_{a\in\alpha^X}\mathbbm{1}_{Ga\setminus H}w_a.
$ By the definition of $r_{\mathrm{loc}}$ and the choice of $z_0$, either $Ga\setminus H=Ga$ or $H\subsetneq B_{2r_{\mathrm{loc}}}(z)\subsetneq Ga$. By (\ref{md:branch-gradient}), 
\[
\sum_a\int_{Ga\setminus H}|\nabla w_a|
\le
D_0|v|_1+\frac14\left(\int_{G^{-1}(X\setminus B_{r_{\mathrm{loc}}}(z))}|\nabla v|+\int_{G^{-1}(B_{r_{\mathrm{loc}}}(z)\setminus H)}|\nabla v|\right).
\]
By Lemma \ref{md:localized-boundary} and
\eqref{md:boundary-change}, 
\[
\sum_a\int_{\partial(Ga)}|w_a|\,dm_1
\le
C_{N_0}|v|_1+\kappa_0 \int_{G^{-1}(X\setminus B_{r_{\mathrm{loc}}}(z))}|\nabla v|. 
\]

Applying \eqref{md:circle-trace}, \eqref{md:branch-mass} and
\eqref{md:branch-gradient}, we obtain that there is a constant $C>0$,
\[
\sum_{\{a:B_{2r_{\mathrm{loc}}}(z)\subset Ga\}}
\int_{\partial H}|w_a|\,dm_1
\le
C|v|_1+\frac14 \int_{G^{-1}(B_{r_{\mathrm{loc}}}(z)\setminus H)}|\nabla v|.
\]
We now estimate $\operatorname{Var}(\mathring{T}v)$: \begin{align*}
    \operatorname{Var}(\mathring{T}v)
& \le  \sum_a \int_{Ga\setminus H}|\nabla w_a|+ \sum_a \int_{\partial(Ga)}|w_a| +\sum_{\{a:B_{2r_{\mathrm{loc}}}(z)\subset Ga\}} \int_{\partial H}|w_a|\\
&\le C_{\mathrm{loc}}|v|_1
+
\left(\frac14+\kappa_0\right)\int_{G^{-1}(X\setminus B_{r_{\mathrm{loc}}}(z))}|\nabla v|
+
\frac12\int_{G^{-1}(B_{r_{\mathrm{loc}}}(z)\setminus H)}|\nabla v|,
\end{align*} that is, $\operatorname{Var}(\mathring{T}v)
\le
\rho\operatorname{Var}(v)+C_{\mathrm{loc}}|v|_1$, where $
\rho:
=
\max\left\{\frac14+\kappa_0,\frac12\right\}<1
$ and $C_{\mathrm{loc}}>0$ do not depend on any $H$ with center $z_0$ and radius $r \in (0,r_{\mathrm{loc}}/2]$.

For $v\in B$, choose smooth $v_j$ approximating $v$ in $L^1$ such that $\operatorname{Var}(v_j) \to \operatorname{Var}(v)$. Since $|\mathring{T}v-\mathring{T}v_j|_1\to 0$, the lower semicontinuity of $\operatorname{Var}$ gives
\[\operatorname{Var}(\mathring{T}v)\le \liminf_{j}\operatorname{Var}(\mathring{T}v_j)\le \liminf_{j}\rho\operatorname{Var}(v_j)+C_{\mathrm{loc}}|v_j|_1=\rho\operatorname{Var}(v)+C_{\mathrm{loc}}|v|_1.\]
Since $R$ is constant on each $a\in \alpha^X$ and
$|\zeta|^{R(a)}\le|\zeta| \le 1$. Thus
\[
\operatorname{Var}(\mathring{R}(\zeta)v)
\le
|\zeta|
\left(\rho\operatorname{Var}(v)+C_{\mathrm{loc}}|v|_1\right),
\]
and $
|\mathring{R}(\zeta)v|_1\le|\zeta|\,|v|_1.
$
These two yield the Lasota-Yorke inequality.\end{proof}

Hence we verified all conditions in Definition \ref{defyoung}. So Theorem \ref{thm} holds. 

\subsubsection{Compute $
1-\mathring\lambda(1)
$ in Theorem \ref{thm}}
We now discuss the location of center $z_0$ of $H$ case by case and compute the $
1-\mathring\lambda(1)
$ appearing in Theorem \ref{thm}. 
\begin{lemma}\label{md:eigenvalue-pointwise}
Fix
$
z_0\in (\operatorname{int}X)\setminus(S\cup S^+)
$
and let $H=B_r(z)\ni z_0$.
For sufficiently small holes $H$, let $\mathring\lambda(1)\in(0,1)$
be the simple leading eigenvalue of
$
\mathring T=\mathring R(1):B\to B.
$

\begin{enumerate}
\item[(i)] \textup{Aperiodic points.}
If
$
G^n z_0\ne z_0
$ for any $n$,
then
\begin{equation}
1-\mathring\lambda(1)
=
\mu_X(H)+o(\mu_X(H)).
\end{equation}

\item[(ii)] \textup{Periodic points.}
Suppose that $z_0$ has period $p$ for $G$
and is a Lebesgue point of $h$, that is,
$
\int_{B_r(z_0)}
|h-h(z_0)|\,d\Leb_X=o(r^2).
$
For a sufficiently small centered hole $H=B_r(z_0)$,
\begin{equation}
\mathring\lambda(1)
=
1-
\left(1-|\det DG^p(z_0)|^{-1}\right)\mu_X(H)
+
o(\mu_X(H)).
\end{equation}
\end{enumerate}

In both cases, the constant in $o(\cdot)$ depends on $z_0$, is uniform over small holes $H$ shrinking to the fixed
 $z_0$.
\end{lemma}

Before proving Lemma \ref{md:eigenvalue-pointwise}, note that the function in $B$ can be unbounded, $\eta_s=1/2$, the estimates above do not verify condition (A6) in \cite{liveranikeller}. We therefore need the following alternative argument.

\begin{lemma}\label{smialrtoliveranikeller}
Similar to \cite{liveranikeller}, for any $j\ge0$, define
\begin{equation*}
q_j(H)
:=
\frac{
\displaystyle
\int_X
(T-\mathring T)\mathring T^j(T-\mathring T)h\,d\Leb_X
}{
\mu_X(H)
}
=
\frac{
\mu_X\left(
H\cap\bigcap_{i=1}^{j}G^{-i}(H^c)
\cap G^{-(j+1)}H
\right)
}{
\mu_X(H)
}.
\end{equation*}
(For $j=0$, the $\bigcap_{i=1}^{j}G^{-i}(H^c)$
is omitted).
Suppose that
$\lim_{r\to 0}q_j(H)= q_j$ for any $j$, then\begin{equation*}
\lim_{\mu_X(H)\to 0}\frac{1-\mathring\lambda(1)}{\mu_X(H)}=
1-\sum_{j\ge0}q_j.
\end{equation*}     
\end{lemma}
\begin{proof}
By Lemma~\ref{localspectrumpic}, there are constants $C>0$, $\kappa\in(0,1)$, $\mathring\lambda(1) \in (0,1)$, $\mathring{\operatorname{Proj}}(1)$, $\mathring h$ such that
$\int_X\mathring h\,d\Leb_X=1$, $\|\mathring h\|\le C, \int_X
\mathring{\operatorname{Proj}}(1)h\,d\Leb_X
\to 1,$
and
\begin{equation}
\label{md:open-spectral-remainder}
\left\|
\mathring T^m
-
\mathring\lambda(1)^m
\mathring{\operatorname{Proj}}(1)
\right\|
\le C\kappa^m.
\end{equation}
\begin{equation}
\label{md:hole-size-bounds}
\mu_X(H)\approx r^2,
\qquad
|\mathbbm1_H|_1\le Cr^2,
\qquad
\|\mathbbm1_H\|\le Cr,
\end{equation} where the constants are uniform over sufficiently small
holes $H$. 
By (\ref{md:smallset}) and boundedness of $T$ imply that for any $v\in B$,
\begin{equation}
\label{md:loss-functional}
\left|
\int_X(T-\mathring T)v\,d\Leb_X
\right|
=
\left|\int_HTv\,d\Leb_X\right|
\le Cr\|v\|.
\end{equation}
\begin{equation}
\label{md:preliminary-eigenvalue-bound}
0<1-\mathring\lambda(1)
=
\int_X(T-\mathring T)\mathring h\,d\Leb_X
\le Cr.
\end{equation}

We can now prove this lemma. Note that $q_j(H)$ are the measures of disjoint first-return events, so
$
0\le\sum_{j=0}^{m-1}q_j(H)\le1.
$ Using $h=C^{\pm 1}$, (\ref{md:smallset}), (\ref{md:hole-size-bounds}) and the Lasota--Yorke inequality, we obtain\begin{align} q_j(H)
&\le
\frac{\mu_X(H\cap G^{-(j+1)}H)}{\mu_X(H)}\le
\frac{C}{\mu_X(H)}
\int_H T^{j+1}\mathbbm1_H\,d\Leb_X \nonumber\\
&\le
\frac{Cr}{\mu_X(H)}
\left(
C\rho^{j+1}\|\mathbbm1_H\|
+C|\mathbbm1_H|_1
\right)\le C\rho^{j+1}+Cr. \label{md:first-return-bound}
\end{align}

This implies 
$
q_j\le C\rho^{j+1}$ by letting $r\to 0$. The identity $Th=h$ gives 
$
\mathring T^m h
=
h-
\sum_{j=0}^{m-1}
\mathring T^j(T-\mathring T)h
$, which implies \begin{align*}
    \int_X(T-\mathring T)\mathring T^m h\,d\Leb_X&=\int_X(T-\mathring T)\Big[h-
\sum_{j=0}^{m-1}
\mathring T^j(T-\mathring T)h\Big]\,d\Leb_X\\
&=\mu_X(H)
\left(1-\sum_{j=0}^{m-1}q_j(H)\right).
\end{align*} This identity, together with $\mathring T
\mathring{\operatorname{Proj}}(1)h=\mathring{\lambda}(1)
\mathring{\operatorname{Proj}}(1)h$, (\ref{md:smallset}), \eqref{md:open-spectral-remainder},
\eqref{md:loss-functional}, and
\eqref{md:hole-size-bounds} imply
\begin{equation}
\label{md:eigenvalue-finite-return}
\begin{split}
&
\mathring\lambda(1)^m
\left(
\int_X
\mathring{\operatorname{Proj}}(1)h\,d\Leb_X
\right)
\frac{1-\mathring\lambda(1)}{\mu_X(H)}=
1-\sum_{j=0}^{m-1}q_j(H)
+
O\left(\frac{\kappa^m}{r}\right).
\end{split}
\end{equation}

Choose $L>1/|\log\kappa|$ and let
$
m=m(r):=\left\lceil L\log(1/r)\right\rceil.
$
Then
$
\frac{\kappa^{m(r)}}r \to 0,
rm(r) \to 0.
$
By \eqref{md:preliminary-eigenvalue-bound},
$
0\le
1-\mathring\lambda(1)^{m(r)}
\le
m(r)(1-\mathring\lambda(1))
\le Crm(r)\to 0.
$

For any $N<m(r)$, by~\eqref{md:first-return-bound}, $\sum_{j=N}^{m(r)-1}q_j(H)
\le C\rho^N+Crm(r).
$
Therefore,
\[
\begin{split}
\left|
\sum_{j=0}^{m(r)-1}q_j(H)-\sum_{j\ge0}q_j
\right|
&\le
\sum_{j=0}^{N-1}|q_j(H)-q_j|
+C\rho^N+Crm(r).
\end{split}
\]
Let $r\to0$ and then $N\to\infty$, we have $\lim_{r\to 0}\sum_{j=0}^{m(r)-1}q_j(H)=\sum_{j\ge0}q_j$.
Hence, in (\ref{md:eigenvalue-finite-return}) where $m=m(r)$, we let $r \to 0$ and conclude this lemma.
\end{proof}

\begin{proof}[Proof of Lemma \ref{md:eigenvalue-pointwise}]\ We now calculate $q_j=\lim_{r\to 0}q_j(H)$ in two cases. \par
\textit{Aperiodic case.}
Fix $j\ge0$.
Since $z_0\notin S$, the map $G^{j+1}$ is continuous
on a neighborhood of $z_0$.
Since $G^{j+1}z_0\ne z_0$ and
$H\subset B_{2r}(z_0)$, for sufficiently small $r$,
\[
H\cap G^{-(j+1)}H=\emptyset.
\]
Thus $q_j(H)=0$ eventually for every fixed $j$,
and Lemma \ref{smialrtoliveranikeller} gives
$
\lim_{r\to 0}\frac{1-\mathring\lambda(1)}{\mu_X(H)}= 1.
$
\textit{Periodic case.}
Let
$
A=DG^p(z_0),
 H=B_r(z_0).
$
Since the periodic orbit avoids the singularities,
$G^p$ is a diffeomorphism on a neighborhood of $z_0$.
Let $\phi$ denote its local inverse fixing $z_0$.
Then
$
\phi(z_0)=z_0,
D\phi(z_0)=A^{-1}.
$
Uniform expansion gives
$
\|A^{-1}\|\le4^{-p}<1.
$
By continuity of $D\phi$, its norm is strictly less
than $1$ on a sufficiently small ball about $z_0$.
Thus for every sufficiently small $r$ and
every $y\in H$,\[
|\phi(y)-z_0|
\le
\left(\sup_{x\in H}\|D\phi(x)\|\right)
|y-z_0|
<r.
\]
Thus $\phi(H)\subsetneq H.$
Iterating this inclusion gives that for any $\ell\ge1$,
$
\phi^\ell(H)\subseteq \phi(H)\subsetneq H.
$ Now we calculate
\[
q_j(H)
=
\frac{
\mu_X\left(
H\cap
\bigcap_{i=1}^{j}G^{-i}(H^c)
\cap G^{-(j+1)}H
\right)
}{
\mu_X(H)
},
\qquad j\ge0.
\]

If $p$ does not divide $j+1$, then
$
G^{j+1}z_0\ne z_0.
$
Continuity of $G^{j+1}$ near $z_0$ implies that
$
H\cap G^{-(j+1)}H=\emptyset
$
for all sufficiently small $r$. Hence
$
q_j(H)=0
$
eventually.

Now fix $\ell\ge1$ and consider $j+1=\ell p$.
For sufficiently small $r$, local injectivity gives
$H\cap G^{-\ell p}H=\phi^\ell(H).$ This is because, if $x\in H$ and $G^{\ell p}x\in H$, 
the local inverse gives
$
x=\phi^\ell(G^{\ell p}x)\in\phi^\ell(H).
$ The
converse is due to  $\phi^\ell(H)\subsetneq H$ and
$G^{\ell p}\phi^\ell(H)=H$.

If $\ell\ge2$, then
$
H\cap G^{-\ell p}H
=
\phi^\ell(H)
\subsetneq \phi(H)
=
H\cap G^{-p}H,
$ i.e.,
every point in $H$ returning at time $\ell p$ 
already returns at time $p$. Hence $\ell \ge 2$ is impossible, so
$
q_{\ell p-1}(H)=0
$
for every $\ell\ge2$, when $r$ is sufficiently
small. 

It remains to compute the case $\ell=1$. Since $p$ is the least period, the sets
$G^iH$ and $H$ are disjoint for $1\le i<p$ when
$r$ is sufficiently small. Since $z_0$ is a Lebesgue point, then $\mu_X(H)=h(z_0)\Leb_X(H)+o(r^2).$  These give
\[
q_{p-1}(H)
=
\frac{\mu_X(H\cap G^{-p}H)}{\mu_X(H)}
=
\frac{\mu_X(\phi(H))}{\mu_X(H)}=\frac{\Leb_X(\phi(H))}{\Leb_X(H)}+o(1)=|\det DG^p(z_0)|^{-1}+o(1)
\]where the last ``$=$" is due to the continuity of $D \phi$ at $z_0$. Hence
$
\lim_{r \to 0}q_{p-1}(H)=
|\det DG^p(z_0)|^{-1}.
$
By Lemma \ref{smialrtoliveranikeller}, we obtain
$
\lim_{r \to 0}\frac{1-\mathring\lambda(1)}{\mu_X(H)}=
1-|\det DG^p(z_0)|^{-1}.
$
\end{proof}

\subsubsection{Conclusions}

We now conclude the following theorem by summarizing all arguments above. The proof is the same as for Theorem \ref{thm2}, lifting the systems to a first return tower over the base $X$. So we will skip the proof. Recall that $k=\lfloor \alpha^{-1} \rfloor$, $\theta:=4^{-1}$, and $
z_0\in \operatorname{int}X\setminus(S\cup S^+)
$ as in (\ref{centerz0}).
\begin{theorem}\label{thm3}
    For the two dimensional nonmarkovian intermittent system $(M,f,\mu)$ considered in this subsection. Fix $z_0\in\operatorname{int}X\setminus(S\cup S^+)$, let $\mu_X:=\frac{\mu|_X}{\mu(X)}$ and $h=\frac{d\mu_X}{d\Leb_X}$. Consider
a sufficiently small hole $H$
with the following restrictions:
\begin{enumerate}
\item If $z_0$ is aperiodic for $f$, let $H=B_r(z)\ni z_0$.
\item If $z_0$ has period $p$ for $f$, suppose that $z_0$
is a Lebesgue point of $h$ and let $H=B_r(z_0)$.
\end{enumerate}
Define
\[
 c_{z_0}=
 \begin{cases}
  1, & z_0\text{ is aperiodic for }f,\\[1mm]
  1-|\det Df^p(z_0)|^{-1},
     & z_0\text{ has period }p\text{ for }f.
 \end{cases}
\]

Then there are constants  $\epsilon, \sigma>0$ such that for any fixed hole $H$ satisfying $\mu_X(H)\in (0,\sigma)$,  
\[\frac{\mu(\tau_H>n)}{\mu(X)}=\sum_{i \ge n}\mu_X(R>i)+\frac{1+O(\mu_X(H)+\diam_{\theta}^{\epsilon}H)}{c_{z_0} \cdot \mu_X(H)+o(\mu_X(H))}\cdot b_n +O_H(n^{- \alpha^{-1} -1})\]
holds for any $n\in \mathbb{N}$, where the \(O(\cdot)\) estimate and the small-hole \(o(\cdot)\) estimate are uniform over $n\in \mathbb{N}$ and the fixed hole $H$ satisfying $\mu_X(H)\in (0,\sigma)$. The constant in \(O_H(n^{-\alpha^{-1}-1})\) depends on the fixed hole \(H\). The value of $\epsilon>0$ is the same as the one in Theorem \ref{thm}, and 
$b_n=\sum_{a+b=n, b>0}\mu_X(R> a)\mu_X(R\ge b)\approx n^{-\alpha^{-1}},  \sum_{i \ge n}\mu_X(R\ge i)\approx n^{-\alpha^{-1}+1}$
 do not depend on $H$.
\end{theorem}
\begin{remark}
    Although we choose specifically  a hole $H\subsetneq X$, we believe that $H$ and its center $z_0$ can be chosen in $X^c\setminus \partial M$, by choosing a bigger base $X$ including $z_0$, and by building the first return tower over $X$. Such construction would require a similar inducing scheme and verification of analogues of Lemma \ref{porpertyofhighlsv} and the subsequent open-system estimates.
\end{remark}

\section{Where orbits prefer to visit for multidimensional intermittent maps}\label{corproof}
\begin{corollary}[Where orbits prefer to visit]\label{where}\ \par
Consider the intermittent systems in Theorem \ref{thm2} and Theorem \ref{thm3}, given two sufficiently small holes $H_1, H_2$ there with $\mu_X(H_1)=\mu_X(H_2)$ centering at $z_1,z_2 \in X$ respectively, there is a sufficiently large $n_0\in \mathbb{N}$, if $n\ge n_0$, then $\mu(\tau_{H_1}>n)>\mu(\tau_{H_2}>n)$ if $c_{z_1}^{-1}>c_{z_2}^{-1}$, and vice versa. A larger value of \(\mu(\tau_{H_1}>n)\) means that an orbit is less likely to have hit \(H_1\) by time \(n\), and is more likely to visit $H_1$ in the time interval $(n, \infty)$. This result characterizes for the first time the phenomenon of where the orbits prefer to go in phase spaces of slowly mixing non-uniformly expanding dynamical systems. (It was proved before in \cite{bunimovichijm} for the most uniformly expanding
dynamical systems).
\end{corollary}

    \begin{proof}[Proof of Corollary \ref{where}]
    First of all we note that, from Theorem \ref{thm2} and Theorem \ref{thm3}, $c_{z_1}$ (resp. $c_{z_2}$) only depends on $z_1$ (resp. $z_2$). If $c_{z_1}^{-1}>c_{z_2}^{-1}$, say $c_{z_1}^{-1}>c_{z_2}^{-1}+\epsilon_0$, consider a sufficiently small $\mu_X(H_1)=\mu_X(H_2) \in (0,\sigma)$ and fix $H_1,H_2$, such that \[\frac{1+O(\mu_X(H_1)+\diam_{\theta}^{\epsilon}H_1)}{c_{z_1}+o(\mu_X(H_1))/\mu_X(H_1)}>\frac{1+O(\mu_X(H_2)+\diam_{\theta}^{\epsilon}H_2)}{c_{z_2} +o(\mu_X(H_2))/  \mu_X(H_2)}+\epsilon_0/2,\]
    which implies \[\frac{1+O(\mu_X(H_1)+\diam_{\theta}^{\epsilon}H_1)}{c_{z_1} \cdot \mu_X(H_1)+o(\mu_X(H_1))}>\frac{1+O(\mu_X(H_2)+\diam_{\theta}^{\epsilon}H_2)}{c_{z_2} \cdot \mu_X(H_2)+o(\mu_X(H_2))}+\epsilon_0/2.\]

    Suppose that the constants of $O_H(n^{-k-\beta-1})$ in (\ref{31}) are $C_{H_1}, C_{H_2}$. There is $n_0$ such that for any $n\ge n_0$, $\max\{C_{H_1},C_{H_2}\}b_n^{-1}n^{-k-1}\le \epsilon_0/8$, so \begin{align*}
    &\frac{1+O(\mu_X(H_1)+\diam_{\theta}^{\epsilon}H_1)}{c_{z_1} \cdot \mu_X(H_1)+o(\mu_X(H_1))}-C_{H_1}b_n^{-1}n^{-k-1}\\
    &\quad >\frac{1+O(\mu_X(H_2)+\diam_{\theta}^{\epsilon}H_2)}{c_{z_2} \cdot \mu_X(H_2)+o(\mu_X(H_2))}+\epsilon_0/4+C_{H_2}b_n^{-1}n^{-k-1}. 
    \end{align*}
    
    Since $b_n$ and $\sum_{i \ge n}\mu_X(R>i)$ are independent of $H_1,H_2$, \begin{align*}
    &\sum_{i \ge n}\mu_X(R>i)+\frac{1+O(\mu_X(H_1)+\diam_{\theta}^{\epsilon}H_1)}{c_{z_1} \cdot \mu_X(H_1)+o(\mu_X(H_1))}b_n-C_{H_1}n^{-k-1}\\
    &\quad >\sum_{i \ge n}\mu_X(R>i)+\frac{1+O(\mu_X(H_2)+\diam_{\theta}^{\epsilon}H_2)}{c_{z_2} \cdot \mu_X(H_2)+o(\mu_X(H_2))}b_n+b_n\epsilon_0/4+C_{H_2}n^{-k-1}
    \end{align*} which implies that 
    \begin{align*}
    &\frac{\mu(\tau_{H_1}>n)}{\mu(X)}\ge \sum_{i \ge n}\mu_X(R>i)+\frac{1+O(\mu_X(H_1)+\diam_{\theta}^{\epsilon}H_1)}{c_{z_1} \cdot \mu_X(H_1)+o(\mu_X(H_1))}b_n-C_{H_1}n^{-k-1}\\
    &\quad >\sum_{i \ge n}\mu_X(R>i)+\frac{1+O(\mu_X(H_2)+\diam_{\theta}^{\epsilon}H_2)}{c_{z_2} \cdot \mu_X(H_2)+o(\mu_X(H_2))}b_n+b_n\epsilon_0/4+C_{H_2}n^{-k-1}\\
    &\quad \ge \frac{\mu(\tau_{H_2}>n)}{\mu(X)}
    \end{align*}holds when $H_1,H_2$ are fixed and small, and any $n\ge n_0:=n_0(H_1,H_2)$.
\end{proof}

\section{Appendix}\label{appendix}

We present here a general operator perturbation result of independent interest for a family of bounded linear operators $\{P_{v}:v \in  I\subsetneq \mathbb{R}^d \text{ or }\mathbb{C}^d\}$, which was used in Lemma \ref{localspectrumpic} to obtain a clearer spectral picture. Roughly speaking, the Proposition \ref{extendliverani1} below says that the spectral picture of $P_v$ is similar to that of $P_{o}$, when $v \approx o \in I$. In camparison to Keller-Liverani \cite{liveranikeller1}, we relax the range of index set $I$, and the conditions on weak perturbations $\tau_v$ (see the definition below).

\begin{proposition}\label{extendliverani1}
Let $(B, ||\cdot||)$ be a Banach space having a second norm $|\cdot|$ (with respect to which generically $B$ is not complete). For any bounded linear operator $Q: B\to B$, define $|||Q|||:=\sup_{||\phi||\le 1}|Q\phi|$. We consider a family of bounded linear operators $\{P_v: B \to B, v\in I\}$, and a  fixed $o \in I$ with
the following properties: there are $C , M>0, \theta \in (0,1), \theta<M$, such that for all $v \in I, n \in \mathbb{N}, \phi \in B$, 
   \begin{enumerate}
        \item Continuity: $|\phi|\le C ||\phi||$
        \item Weak norms boundedness: $|P^n_v|\le C M^n$
        \item Uniform Lasota-Yorke inequalities: $||P_v^n\phi||\le C \theta^n||\phi||+CM^n|\phi|$
        \item If $z \in \sigma(P_v)$ and $|z|> \theta$, then $z$ does not belong to the residual spectrum of $P_v$. Here, ``$z$ is in the residual spectrum of $P_v$" means that $z-P_v$ is injective, but it does not have a dense range.
        \item Weak perturbations: $\tau_v:=|||P_o-P_v|||\to 0$ as $v \to o$.
    \end{enumerate}

Fix $\delta>0$ and $r\in(\theta, M)$. Let $\eta:=\frac{\log r/\theta}{\log M/\theta} \in (0,1)$. Define $V_{\delta, r}:=\{z\in \mathbb{C}: |z|< r \text{ or } \dist(z, \sigma(P_o))< \delta\}$. Then there are constants $\epsilon_{\delta, r}, b_{\delta,r}$ and $a_{r}>0$ ($a_r$, which does not depend on $\delta$), such that for any $\dist(v,o)\le \epsilon_{\delta, r}, \phi\in B$, 
\begin{align}
\sup_{z\notin V_{\delta,r}}
\|(z-P_v)^{-1}\phi\|
&\le a_r\|\phi\|+b_{\delta,r}|\phi|,\label{3}
\\[1ex]
\sup_{z\notin V_{\delta,r}}
\left|\!\left|\!\left|
(z-P_v)^{-1}-(z-P_o)^{-1}
\right|\!\right|\!\right|
&\lesssim_{\delta,r}
\tau_v^{1/2}
+\tau_v^{\frac{\eta}{2(1-\eta)}} \label{a3}
\end{align}
 where the constant in ``$\precsim_{\delta,r}$" does not depend on $v$. 

If $\lambda$ is an isolated eigenvalue of $P_o$ with $|\lambda|> \theta$, then choose any $\delta< \min\{a_r^{-1},|\lambda|-r\}/2$, such that $B_{2\delta}(\lambda)\bigcap \sigma(P_o)=\{\lambda\}$, and define a projection $\Proj_{v}^{\lambda,\delta}:=\frac{1}{2\pi i}\int_{\partial B_{\delta}(\lambda)}(z-P_v)^{-1}dz$. Then there is a constant $\epsilon_{\delta, r}'>0$, such that for any $\dist(v,o)\le \epsilon_{\delta, r}'$, \begin{align}\label{a4}
rank(\Proj_{v}^{\lambda,\delta})=rank(\Proj_{o}^{\lambda,\delta}), \quad ||\Proj_{v}^{\lambda,\delta}||\le \delta(a_r+b_{\delta,r}),   
\end{align}
\begin{align}\label{a5}
    |||\Proj_{v}^{\lambda,\delta}-\Proj_{o}^{\lambda,\delta}|||\precsim_{\delta,r} \tau_v^{1/2}+\tau_{v}^{\frac{\eta}{2(1-\eta)}}.
\end{align} 

Suppose that $\partial B_{r}(0)\bigcap \sigma(P_o)=\emptyset$ and choose $\delta\le
\operatorname{dist}\bigl(\partial B_r(0),\sigma(P_o)\bigr)
$. Define now another projection $\Proj_{v}^{r}:=\frac{1}{2\pi i}\int_{\partial B_{r}(0)}(z-P_v)^{-1}dz$. Then for any $\dist(v,o)\le \epsilon_{\delta,r}$, \begin{align}\label{a6}
    ||P_v^n\circ \Proj_{v}^{r}||\le r^{n+1}(a_r+b_{\delta,r}). 
\end{align}
\end{proposition}

\begin{proof}[Proof of Proposition \ref{extendliverani1}] We will follow the scheme of \cite{liveranikeller1}, but will describe the differences. Let us start with addressing (\ref{3}). 

By Lasota-Yorke inequalities we have $||P^n_v||\le 2C M^n$.  (\ref{3}) is true when $|z|> 4CM$, because $||(z-P_v)^{-1}||\le (2CM)^{-1} $. So, we assume that $|z|\le 4CM$ and $z \notin V_{\delta, r}$, (which implies that $(z-P_o)^{-1}$ exists), and suppose that $(z-P_v)g=h$. Since $|z|>r$ and $z^ng=(z^n-P_v^n)g+P_v^ng$, then Lasota-Yorke inequalities imply that\begin{align*}
    r^n||g||\le |z|^n||g||&\le ||(z^n-P_v^n)g||+||P_v^ng||\\
    &\precsim ||\sum_{i \le n-1}z^iP_v^{n-1-i}(z-P_v)g||+\theta^n||g||+M^n|g|\\
    &\precsim \sum_{i \le n-1}|z|^iM^{n-1-i}||h||+\theta^n||g||+M^n|g|.
\end{align*}

Hence we have $||g||\le C'/r\sum_{i \le n-1}|z/r|^i|M/r|^{n-1-i}||h||+C'(\theta/r)^n||g||+C'(M/r)^n|g|$ for some $C'>0$, which depends only on $C$. Choose $n=n_{C',\theta,r}$, such that $C'(\theta/r)^n<1$. Then  \begin{gather}\label{1}
    ||g||\precsim_{r}||h||+|g|\precsim_{r}||(z-P_v)g||+|g|.
\end{gather}

Next, we need to estimate $|g|$. Since $\sup_{z \notin V_{\delta, r}}||(z-P_o)^{-1}|| \precsim_{\delta,r}1$, and $g=(z-P_o)^{-1}(z-P_o)g=\sum_{n=0}^{N-1}\frac{P_o^n}{z^{n+1}}(z-P_o)g+(z-P_o)^{-1}\frac{P_o^N}{z^{N}}(z-P_o)g$ for any $N \in \mathbb{N}$, then by the conditions of weak norms boundedness, continuity, and Lasota-Yorke inequalities we have, \begin{align}
    |g|&\precsim\sum_{i=0}^{N-1}\frac{M^i}{r^{i+1}}|(z-P_o)g|+||(z-P_o)^{-1}||\Big|\Big|\frac{P_o^N}{z^{N}}(z-P_o)g\Big|\Big|\nonumber\\
    &\precsim_{\delta,r} (M/r)^N|(z-P_o)g|+(\theta/r)^N||(z-P_o)g||  \nonumber\\
    &\precsim_{\delta,r} (M/r)^N|(z-P_o)g|+(\theta/r)^N(|z|+2CM)||g||\nonumber\\
    &\precsim_{\delta,r} (M/r)^N|(z-P_v)g|+\big[(M/r)^N\tau_v+(\theta/r)^N(|z|+2CM)\big]||g||\label{2},
\end{align} where the last line is due to the weak perturbations, i.e., $|(z-P_o)g|\le |(z-P_v)g|+|(P_v-P_o)g|\le |(z-P_v)g|+\tau_{v}||g||$. 

Therefore, by (\ref{1}) and (\ref{2}), when $|z|\le 4CM$, there are constants $C_r, C_{\delta,r}>0$, such that  
\begin{align}\label{21}
    ||g||\le  C_r ||(z-P_v)g||+C_{\delta,r}(M/r)^N|(z-P_v)g|+C_{\delta,r}\big[(M/r)^N\tau_v+(\theta/r)^N\big]||g||
\end{align}for any $N\in \mathbb{N}$. Now we choose $N=N_{\delta,r}>0$ so that $C_{\delta,r}(\theta/r)^{N}\le 1/4$. Choose a small $\epsilon_{\delta,r}>0$, such that $C_{\delta,r}(M/r)^{N_{\delta,r}}\tau_v \le 1/4$, when $\dist(v,o)\le \epsilon_{\delta,r}$. Hence, by (\ref{21}), there are such constants $a_r, b_{\delta,r}>0$, 
\[||g||\le a_r||(z-P_v)g||+b_{\delta,r}|(z-P_v)g|\]holds, when $\dist(v,o)\le \epsilon_{\delta,r}$. On the other hand, $z$ is not in the residual spectrum. All these imply that, when $\dist(v,o)\le \epsilon_{\delta,r}$, the operator $z-P_v$ is injective and has a dense range. Hence, $(z-P_v)^{-1}$ exists, is bounded for any $z\notin V_{\delta, r}$, and $\sup_{ z\notin V_{\delta, r}}||(z-P_v)^{-1}\phi||\le a_r ||\phi||+b_{\delta,r}|\phi|$ for any $\phi \in B$. In particular, $\sup_{ z\notin V_{\delta, r}}||(z-P_v)^{-1}||\le a_r+b_{\delta,r}$ for any $\dist(v,o)\le \epsilon_{\delta,r}$. Thus  (\ref{3}) is proved.

Next we address (\ref{a3}). Choose a different $N$, such that $(\theta/r)^N\le \tau_{v}^{\frac{\eta}{2(1-\eta)}}\le (\theta/r)^{N-1} $. This implies $(M/r)^{N}\le \tau_{v}^{-1/2}M/r$ and $(M/r)^N\tau_v+(\theta/r)^N\le \tau_v^{1/2}M/r+\tau_{v}^{\frac{\eta}{2(1-\eta)}}$. The (\ref{2}) implies that 
 there is a constant $B_{\delta,r}>0$, such that 
\begin{gather}\label{29}
|g|\le B_{\delta,r} (M/r)^N|(z-P_v)g|+B_{\delta,r} \big[(M/r)^N\tau_v+(\theta/r)^N(|z|+1)\big]||g||.
\end{gather}

Now for any $\phi \in B$, $h:=(P_v-P_o)(z-P_o)^{-1}\phi$, and for any $\dist(v,o)\le \epsilon_{\delta,r}$, it follows from (\ref{29}) and $\sup_{ z\notin V_{\delta, r}}||(z-P_v)^{-1}||\le a_r+b_{\delta,r}$ that, when $|z|\le 4CM$, \begin{align*}
\big|\big[(z-P_v)^{-1}-&(z-P_o)^{-1}\big]\phi\big|=|(z-P_v)^{-1}h|\\
&\precsim_{\delta,r}\tau_v^{-1/2}|h|+(\tau_v^{1/2}+\tau_{v}^{\frac{\eta}{2(1-\eta)}})||(z-P_v)^{-1}h||\\
        & \precsim_{\delta,r} \tau_v^{-1/2}\tau_v||(z-P_o)^{-1}||||\phi||\\
        &\quad  +(\tau_v^{1/2}+\tau_{v}^{\frac{\eta}{2(1-\eta)}})||(z-P_v)^{-1}||||P_v-P_o||||(z-P_o)^{-1}||||\phi||\\
        & \precsim_{\delta,r} (\tau_v^{1/2}+\tau_{v}^{\frac{\eta}{2(1-\eta)}})||\phi||,
    \end{align*} and if $|z|> 4CM$, then
   \begin{align*}
\big|\big[(z-P_v)^{-1}-&(z-P_o)^{-1}\big]\phi\big|=|(z-P_v)^{-1}h|\\
&\precsim_{\delta,r}\tau_v^{-1/2}|h|+(\tau_v^{1/2}+\tau_{v}^{\frac{\eta}{2(1-\eta)}})||(z-P_v)^{-1}h||+\tau_{v}^{\frac{\eta}{2(1-\eta)}} |z| ||(z-P_v)^{-1}h||\\
        & \precsim_{\delta,r} (\tau_v^{1/2}+\tau_{v}^{\frac{\eta}{2(1-\eta)}})||\phi|| +2\tau_{v}^{\frac{\eta}{2(1-\eta)}} ||P_v-P_o||||(z-P_o)^{-1}||||\phi|| \\
        & \precsim_{\delta,r} (\tau_v^{1/2}+\tau_{v}^{\frac{\eta}{2(1-\eta)}})||\phi||,
    \end{align*} where the second ``$\precsim_{\delta,r}$" is due to $|z| ||(z-P_v)^{-1}||\le 2$.
    
    Thus, $\sup_{z\notin V_{\delta, r}}|||(z-P_v)^{-1}-(z-P_o)^{-1}|||\precsim_{\delta,r} \tau_v^{1/2}+\tau_{v}^{\frac{\eta}{2(1-\eta)}}$ for any $\dist(v,o)\le \epsilon_{\delta,r}$. Hence (\ref{a3}) holds.

    Now we address (\ref{a4}). At first, we prove that $rank(\Proj_{v}^{\lambda,\delta})\le rank(\Proj_{o}^{\lambda,\delta})$. Suppose that $0\neq \phi \in \Proj_v^{\lambda, \delta}(B)$. Then $\Proj_v^{\lambda, \delta}\phi=\phi$. It follows from (\ref{3}) that for any $\dist(v,o)\le \epsilon_{\delta,r}$, $||\phi||=||\Proj_{v}^{\lambda,\delta}\phi||\le \frac{1}{2\pi}\int_{\partial B_{\delta}(\lambda)}||(z-P_v)^{-1}\phi|||dz|\le \delta a_r||\phi||+\delta b_{\delta,r}|\phi|$. So, since $\delta< \min\{1/a_r, |\lambda|-r\}/2$, then $||\phi||\precsim_{\delta, r}|\phi|$. On the other hand, it follows from (\ref{a3}) that \begin{align*}
        |\phi-\Proj_{o}^{\lambda,\delta}\phi|&=|\Proj_{v}^{\lambda,\delta}\phi-\Proj_{o}^{\lambda,\delta}\phi|\\
        &\le \frac{||\phi||}{2\pi }\int_{\partial B_{\delta}(\lambda)}|||(z-P_v)^{-1}-(z-P_o)^{-1}||||dz|\\
        &\precsim_{\delta,r}\delta(\tau_v^{1/2}+\tau_{v}^{\frac{\eta}{2(1-\eta)}})||\phi||\precsim_{\delta,r}\delta(\tau_v^{1/2}+\tau_{v}^{\frac{\eta}{2(1-\eta)}})|\phi|.
    \end{align*} Then there is small $\epsilon_{\delta,r}'\in (0, \epsilon_{\delta,r})$, such that for any $\dist(v,o)\le \epsilon_{\delta,r}', |\phi-\Proj_{o}^{\lambda,\delta}\phi|\le 0.5|\phi|$, i.e., $|\phi|\le 2|\Proj_{o}^{\lambda,\delta}\phi|$. If $rank(\Proj_{v}^{\lambda,\delta})> rank(\Proj_{o}^{\lambda,\delta})=:m$, then $\Proj_{o}^{\lambda,\delta}$ maps $\Proj_{v}^{\lambda,\delta}B$ (which has dimension greater than \(m\)) to the $m$-dimensional set $\Proj_{o}^{\lambda,\delta}B$. So there must exist $0\neq \phi\in \Proj_{v}^{\lambda,\delta}B$, such that $\Proj_{o}^{\lambda,\delta}\phi=0$, then $|\phi|\le 2|\Proj_o^{\lambda,\delta}\phi|=0$, i.e., $\phi=0$, which leads to a contradiction!
    
    The same arguments give $rank(\Proj_{v}^{\lambda,\delta})\ge rank(\Proj_{o}^{\lambda,\delta})$. Therefore, $rank(\Proj_{v}^{\lambda,\delta})= rank(\Proj_{o}^{\lambda,\delta})$ for any $\dist(v,o)\le \epsilon_{\delta,r}'$. $|||\Proj_{v}^{\lambda,\delta}-\Proj_{o}^{\lambda,\delta}|||\precsim_{\delta,r} \tau_v^{1/2}+\tau_{v}^{\frac{\eta}{2(1-\eta)}}$ and $||\Proj_{v}^{\lambda,\delta}||\le \delta(a_r+b_{\delta,r})$ directly follow from (\ref{3}), (\ref{a3}) and the definitions of $\Proj_{v}^{\lambda,\delta}, \Proj_{o}^{\lambda,\delta}$.

    Now we address (\ref{a6}). Suppose that $\partial B_{r}(0)\bigcap \sigma(P_o)=\emptyset$, when $\dist(v,o)\le \epsilon_{\delta,r}$. Then  we have that \[||P_v^n\circ \Proj_{v}^{r}||=\Big|\Big|\frac{1}{2\pi i}\int_{\partial B_{r}(0)}z^n(z-P_v)^{-1}dz\Big|\Big|\le  r^{n+1}(a_r+b_{\delta,r})\] by using (\ref{3}), which concludes a proof.
\end{proof}

\section*{Acknowledgements} Y. Su thanks Prof. Carlangelo Liverani, Prof. Ian Melbourne and Prof. Sandro Vaienti for helpful discussions.

\bibliography{bibtext}

@article {Demersexp,
    AUTHOR = {Bruin, Henk and Demers, Mark F. and Todd, Mike},
     TITLE = {Hitting and escaping statistics: mixing, targets and holes},
   JOURNAL = {Adv. Math.},
  FJOURNAL = {Advances in Mathematics},
    VOLUME = {328},
      YEAR = {2018},
     PAGES = {1263--1298},
      ISSN = {0001-8708,1090-2082},
   MRCLASS = {37A25 (37C30 37D25 37D35 37E05)},
  MRNUMBER = {3771153},
MRREVIEWER = {Elismar\ da Rosa Oliveira},
       DOI = {10.1016/j.aim.2017.12.020},
       URL = {https://doi.org/10.1016/j.aim.2017.12.020},
}

@article {sujsp,
    AUTHOR = {\textbf{Yaofeng Su}},
     TITLE = {Polynomial escape rates via maximal large deviations},
   JOURNAL = {J. Stat. Phys.},
  FJOURNAL = {Journal of Statistical Physics},
    VOLUME = {192},
      YEAR = {2025},
    NUMBER = {5},
     PAGES = {Paper No. 73, 7},
      ISSN = {0022-4715,1572-9613},
   MRCLASS = {37A50 (37C40 37E05 60F10)},
  MRNUMBER = {4905601},
       DOI = {10.1007/s10955-025-03455-4},
       URL = {https://doi.org/10.1007/s10955-025-03455-4},
}

@article{bunimovich2026,
doi = {10.1088/1361-6544/ae7ba7},
url = {https://doi.org/10.1088/1361-6544/ae7ba7},
year = {2026},
month = {jul},
publisher = {IOP Publishing},
volume = {39},
number = {7},
pages = {075020},
author = {Bunimovich, Leonid and Kovalenko, Kirill},
title = {For how long time evolution of chaotic or random systems can be predicted},
journal = {Nonlinearity}
}

@article{highdimlsv,
title = {Sharp statistical properties for a family of multidimensional nonMarkovian nonconformal intermittent maps},
journal = {Advances in Mathematics},
volume = {388},
pages = {107853},
year = {2021},
issn = {0001-8708},
doi = {https://doi.org/10.1016/j.aim.2021.107853},
url = {https://www.sciencedirect.com/science/article/pii/S0001870821002929},
author = {Peyman Eslami and Ian Melbourne and Sandro Vaienti}
}

@article{rh1,
author = {Leonid A. Bunimovich and Carl P. Dettmann},
title = {A Billiard in an Open Circle and the Riemann Zeta Function},
journal = {Exp. Math.},
volume = {0},
number = {0},
pages = {1--25},
year = {2024},
publisher = {Taylor \& Francis},
doi = {10.1080/10586458.2024.2423180},

URL = { 
      https://doi.org/10.1080/10586458.2024.2423180
  

},
eprint = { https://doi.org/10.1080/10586458.2024.2423180

}

}

@article{rh,
  title = {Open Circular Billiards and the Riemann Hypothesis},
  author = {Bunimovich, L. A. and Dettmann, C. P.},
  journal = {Phys. Rev. Lett.},
  volume = {94},
  issue = {10},
  pages = {100201},
  numpages = {4},
  year = {2005},
  month = {Mar},
  publisher = {American Physical Society},
  doi = {10.1103/PhysRevLett.94.100201},
  url = {https://link.aps.org/doi/10.1103/PhysRevLett.94.100201}
}

@article {demers2,
    AUTHOR = {Demers, Mark F.},
     TITLE = {Markov extensions and conditionally invariant measures for
              certain logistic maps with small holes},
   JOURNAL = {Ergodic Theory Dynam. Systems},
  FJOURNAL = {Ergodic Theory and Dynamical Systems},
    VOLUME = {25},
      YEAR = {2005},
    NUMBER = {4},
     PAGES = {1139--1171},
      ISSN = {0143-3857,1469-4417},
   MRCLASS = {37E05 (37A05 37C30 37D50)},
  MRNUMBER = {2158400},
MRREVIEWER = {Nikolai\ Chernov},
       DOI = {10.1017/S0143385704000963},
       URL = {https://doi.org/10.1017/S0143385704000963},
}

@article {demers1,
    AUTHOR = {Demers, Mark F.},
     TITLE = {Markov extensions for dynamical systems with holes: an
              application to expanding maps of the interval},
   JOURNAL = {Israel J. Math.},
  FJOURNAL = {Israel Journal of Mathematics},
    VOLUME = {146},
      YEAR = {2005},
     PAGES = {189--221},
      ISSN = {0021-2172,1565-8511},
   MRCLASS = {37E05 (37A30)},
  MRNUMBER = {2151600},
MRREVIEWER = {Nikolai\ Chernov},
       DOI = {10.1007/BF02773533},
       URL = {https://doi.org/10.1007/BF02773533},
}

@article{Sudcds,
    AUTHOR = {Su, Yaofeng},
     TITLE = {Almost surely invariance principle for non-stationary and
              random intermittent dynamical systems},
   JOURNAL = {Discrete Contin. Dyn. Syst.},
  FJOURNAL = {Discrete and Continuous Dynamical Systems. Series A},
    VOLUME = {39},
      YEAR = {2019},
    NUMBER = {11},
     PAGES = {6585--6597},
      ISSN = {1078-0947},
   MRCLASS = {37A50 (37E05 60F17)},
  MRNUMBER = {4026993},
       DOI = {10.3934/dcds.2019286},
       URL = {https://doi.org/10.3934/dcds.2019286},
}

@article{Bunimovich_2007,
doi = {10.1209/0295-5075/80/40001},
url = {https://dx.doi.org/10.1209/0295-5075/80/40001},
year = {2007},
month = {oct},
publisher = {},
volume = {80},
number = {4},
pages = {40001},
author = {L. A. Bunimovich and C. P. Dettmann},
title = {Peeping at chaos: Nondestructive monitoring of chaotic systems by measuring long-time escape rates},
journal = {Europhysics Letters}
}

@ARTICLE{Sutams,
       author = {{Su}, Yaofeng},
        title = "{Vector-valued Almost Sure Invariance Principle For Non-stationary Dynamical Systems}",
      journal = {accepted by Trans. AMS},
        YEAR = {2022},
 primaryClass = {math.DS},
       DOI = {10.1090/tran/8609},
       URL = {https://doi.org/10.1090/tran/8609}
}

@article {bunimovichijm,
    AUTHOR = {Bunimovich, Leonid A. and Yurchenko, Alex},
     TITLE = {Where to place a hole to achieve a maximal escape rate},
   JOURNAL = {Israel J. Math.},
  FJOURNAL = {Israel Journal of Mathematics},
    VOLUME = {182},
      YEAR = {2011},
     PAGES = {229--252},
      ISSN = {0021-2172,1565-8511},
   MRCLASS = {37A30 (28D05 37D20)},
  MRNUMBER = {2783972},
MRREVIEWER = {Henk\ Bruin},
       DOI = {10.1007/s11856-011-0030-8},
       URL = {https://doi.org/10.1007/s11856-011-0030-8},
}

@article {boldingbun,
    AUTHOR = {Bolding, M. and Bunimovich, L. A.},
     TITLE = {Where and when orbits of chaotic systems prefer to go},
   JOURNAL = {Nonlinearity},
  FJOURNAL = {Nonlinearity},
    VOLUME = {32},
      YEAR = {2019},
    NUMBER = {5},
     PAGES = {1731--1771},
      ISSN = {0951-7715,1361-6544},
   MRCLASS = {37H10 (37D45)},
  MRNUMBER = {3942599},
MRREVIEWER = {M.\ L.\ Blank},
       DOI = {10.1088/1361-6544/ab0c34},
       URL = {https://doi.org/10.1088/1361-6544/ab0c34},
}

@article {fresta,
    AUTHOR = {Freitas, Ana Cristina Moreira and Freitas, Jorge Milhazes and
              Todd, Mike},
     TITLE = {Speed of convergence for laws of rare events and escape rates},
   JOURNAL = {Stochastic Process. Appl.},
  FJOURNAL = {Stochastic Processes and their Applications},
    VOLUME = {125},
      YEAR = {2015},
    NUMBER = {4},
     PAGES = {1653--1687},
      ISSN = {0304-4149,1879-209X},
   MRCLASS = {37A50 (37B20 37C25 60G10 60G70)},
  MRNUMBER = {3310360},
       DOI = {10.1016/j.spa.2014.11.011},
       URL = {https://doi.org/10.1016/j.spa.2014.11.011},
}

@article {haydn,
    AUTHOR = {Haydn, N. and Yang, F.},
     TITLE = {Local escape rates for {$\phi$}-mixing dynamical systems},
   JOURNAL = {Ergodic Theory Dynam. Systems},
  FJOURNAL = {Ergodic Theory and Dynamical Systems},
    VOLUME = {40},
      YEAR = {2020},
    NUMBER = {10},
     PAGES = {2854--2880},
      ISSN = {0143-3857,1469-4417},
   MRCLASS = {37A25 (37A05 37B10 37D35)},
  MRNUMBER = {4138913},
MRREVIEWER = {Marks\ Ruziboev},
       DOI = {10.1017/etds.2019.21},
       URL = {https://doi.org/10.1017/etds.2019.21},
}

@article {sarig,
    AUTHOR = {Sarig, Omri},
     TITLE = {Subexponential decay of correlations},
   JOURNAL = {Invent. Math.},
  FJOURNAL = {Inventiones Mathematicae},
    VOLUME = {150},
      YEAR = {2002},
    NUMBER = {3},
     PAGES = {629--653},
      ISSN = {0020-9910,1432-1297},
   MRCLASS = {37A25 (28D05 47A35)},
  MRNUMBER = {1946554},
MRREVIEWER = {Yves\ Lacroix},
       DOI = {10.1007/s00222-002-0248-5},
       URL = {https://doi.org/10.1007/s00222-002-0248-5},
}

@article{pollicot,
  title={Escape rates for Gibbs measures},
  author={Andrew Ferguson and Mark Pollicott},
  journal={Ergodic Theory and Dynamical Systems},
  year={2010},
  volume={32},
  pages={961 - 988},
  url={https://doi.org/10.1017/S0143385711000058}
}

@article{bunimovich2012fair,
  title={Fair dice-like hyperbolic systems},
  author={Bunimovich, Leonid},
  journal={Dynamical Systems and Group Actions},
  volume={467},
  pages={79--87},
  year={2012},
  publisher={American Mathematical Society}
}

@article{DemersLSY,
doi = {10.1088/0951-7715/19/2/008},
url = {https://dx.doi.org/10.1088/0951-7715/19/2/008},
year = {2005},
month = {dec},
publisher = {},
volume = {19},
number = {2},
pages = {377},
author = {Mark F Demers and Lai-Sang Young},
title = {Escape rates and conditionally invariant measures},
journal = {Nonlinearity}
}

@article{BRUIN_DEMERS_MELBOURNE_2010, title={Existence and convergence properties of physical measures for certain dynamical systems with holes}, volume={30}, DOI={10.1017/S0143385709000200}, number={3}, journal={Ergodic Theory and Dynamical Systems}, author={BRUIN, HENK and DEMERS, MARK and MELBOURNE, IAN}, year={2010}, pages={687–728}}

@article {gouezel,
    AUTHOR = {Gou\"{e}zel, S\'{e}bastien},
     TITLE = {Sharp polynomial estimates for the decay of correlations},
   JOURNAL = {Israel J. Math.},
  FJOURNAL = {Israel Journal of Mathematics},
    VOLUME = {139},
      YEAR = {2004},
     PAGES = {29--65},
      ISSN = {0021-2172,1565-8511},
   MRCLASS = {37A25 (28D05 47A35)},
  MRNUMBER = {2041223},
MRREVIEWER = {Yves\ Lacroix},
       DOI = {10.1007/BF02787541},
       URL = {https://doi.org/10.1007/BF02787541},
}

@article{liveranikeller1,
     author = {Keller, Gerhard and Liverani, Carlangelo},
     title = {Stability of the spectrum for transfer operators},
     journal = {Annali della Scuola Normale Superiore di Pisa - Classe di Scienze},
     pages = {141--152},
     publisher = {Scuola normale superiore},
     volume = {Ser. 4, 28},
     number = {1},
     year = {1999},
     mrnumber = {1679080},
     zbl = {0956.37003},
     language = {en},
     url = {http://www.numdam.org/item/ASNSP_1999_4_28_1_141_0/}
}

@article {liveranikeller,
    AUTHOR = {Keller, Gerhard and Liverani, Carlangelo},
     TITLE = {Rare events, escape rates and quasistationarity: some exact
              formulae},
   JOURNAL = {J. Stat. Phys.},
  FJOURNAL = {Journal of Statistical Physics},
    VOLUME = {135},
      YEAR = {2009},
    NUMBER = {3},
     PAGES = {519--534},
      ISSN = {0022-4715,1572-9613},
   MRCLASS = {37A25 (37A50 37C30 37E05)},
  MRNUMBER = {2535206},
MRREVIEWER = {Gary\ Froyland},
       DOI = {10.1007/s10955-009-9747-8},
       URL = {https://doi.org/10.1007/s10955-009-9747-8},
}

@article {melbourneinvention,
    AUTHOR = {Melbourne, Ian and Terhesiu, Dalia},
     TITLE = {Operator renewal theory and mixing rates for dynamical systems
              with infinite measure},
   JOURNAL = {Invent. Math.},
  FJOURNAL = {Inventiones Mathematicae},
    VOLUME = {189},
      YEAR = {2012},
    NUMBER = {1},
     PAGES = {61--110},
      ISSN = {0020-9910,1432-1297},
   MRCLASS = {37A25 (37A40 37A50 60K05)},
  MRNUMBER = {2929083},
MRREVIEWER = {Nathaniel\ F. G. Martin},
       DOI = {10.1007/s00222-011-0361-4},
       URL = {https://doi.org/10.1007/s00222-011-0361-4},
}

@article{Raizen,
  title={Optical billiards for atoms},
  author={Milner, V and Hanssen, JL and Campbell, WC and Raizen, MG},
  journal={Physical Review Letters},
  volume={86},
  number={8},
  pages={1514},
  year={2001},
  publisher={APS}
}

@article{Davidson,
  title={Observation of chaotic and regular dynamics in atom-optics billiards},
  author={Friedman, Nir and Kaplan, Ariel and Carasso, Dina and Davidson, Nir},
  journal={Physical review letters},
  volume={86},
  number={8},
  pages={1518},
  year={2001},
  publisher={APS}
}

@article {Demers,
    AUTHOR = {Demers, Mark F. and Fernandez, Bastien},
     TITLE = {Escape rates and singular limiting distributions for
              intermittent maps with holes},
   JOURNAL = {Trans. Amer. Math. Soc.},
  FJOURNAL = {Transactions of the American Mathematical Society},
    VOLUME = {368},
      YEAR = {2016},
    NUMBER = {7},
     PAGES = {4907--4932},
      ISSN = {0002-9947,1088-6850},
   MRCLASS = {37C40 (37C30 37D25 37E05)},
  MRNUMBER = {3456165},
MRREVIEWER = {M.\ L.\ Blank},
       DOI = {10.1090/tran/6481},
       URL = {https://doi.org/10.1090/tran/6481},
}

@article {mldp,
    AUTHOR = {Bunimovich, Leonid A. and Su, Yaofeng},
     TITLE = {Maximal large deviations and slow recurrences in weakly
              chaotic systems},
   JOURNAL = {Adv. Math.},
  FJOURNAL = {Advances in Mathematics},
    VOLUME = {432},
      YEAR = {2023},
     PAGES = {Paper No. 109267, 58},
      ISSN = {0001-8708,1090-2082},
   MRCLASS = {37A50 (37A25)},
  MRNUMBER = {4630843},
MRREVIEWER = {Ivan\ Podvigin},
       DOI = {10.1016/j.aim.2023.109267},
       URL = {https://doi.org/10.1016/j.aim.2023.109267},
}

@article {lsv,
    AUTHOR = {Liverani, Carlangelo and Saussol, Beno\^{i}t and Vaienti,
              Sandro},
     TITLE = {A probabilistic approach to intermittency},
   JOURNAL = {Ergodic Theory Dynam. Systems},
  FJOURNAL = {Ergodic Theory and Dynamical Systems},
    VOLUME = {19},
      YEAR = {1999},
    NUMBER = {3},
     PAGES = {671--685},
      ISSN = {0143-3857,1469-4417},
   MRCLASS = {37D25 (28D05 37A99 37C40 37E05)},
  MRNUMBER = {1695915},
MRREVIEWER = {Makoto\ Mori},
       DOI = {10.1017/S0143385799133856},
       URL = {https://doi.org/10.1017/S0143385799133856},
}

@article {Young,
    AUTHOR = {Young, Lai-Sang},
     TITLE = {Recurrence times and rates of mixing},
   JOURNAL = {Israel J. Math.},
  FJOURNAL = {Israel Journal of Mathematics},
    VOLUME = {110},
      YEAR = {1999},
     PAGES = {153--188},
      ISSN = {0021-2172,1565-8511},
   MRCLASS = {37D25 (37A25 37D50)},
  MRNUMBER = {1750438},
MRREVIEWER = {Benoit\ Saussol},
       DOI = {10.1007/BF02808180},
       URL = {https://doi.org/10.1007/BF02808180},
}

\end{document}